\documentclass[11pt,reqno]{amsart}
\usepackage[margin=1in]{geometry}
\usepackage{amsmath,amssymb,amsthm,mathtools,microtype}
\usepackage[hidelinks]{hyperref}
\usepackage{zref-clever}
\zcsetup{cap=true,abbrev=false,labelhook=true}
\numberwithin{equation}{section}
\newtheorem{theorem}{Theorem}[section]
\newtheorem{lemma}[theorem]{Lemma}
\newtheorem{proposition}[theorem]{Proposition}
\newtheorem{corollary}[theorem]{Corollary}
\AddToHook{env/theorem/begin}{\zcsetup{reftype=theorem}}
\AddToHook{env/lemma/begin}{\zcsetup{reftype=lemma}}
\AddToHook{env/proposition/begin}{\zcsetup{reftype=proposition}}
\AddToHook{env/corollary/begin}{\zcsetup{reftype=corollary}}
\newcommand{\E}{\mathbb E}
\newcommand{\Pp}{\mathbb P}
\newcommand{\1}{\mathbf 1}
\DeclareMathOperator{\Var}{Var}
\DeclareMathOperator{\Cov}{Cov}
\DeclareMathOperator{\TV}{TV}
\DeclareMathOperator{\Sym}{Sym}

\newcommand{\norm}[1]{\lVert#1\rVert}
\title[Nearly-minimax variance estimation]{Nearly-minimax variance estimation under rough random design}
\author{P. M. Aronow}
\author{Patrick Lopatto}
\thanks{Authors are listed in alphabetical order.}
\date{7 September 2026}
\subjclass[2020]{62G08, 62C20}
\keywords{variance estimation, nonparametric regression, random design, nearly-minimax estimation}
\hypersetup{pdftitle={Nearly-minimax variance estimation under rough random design},pdfauthor={P. M. Aronow and Patrick Lopatto}}
\begin{document}
\begin{abstract}
We identify the minimax exponent for constant conditional variance
estimation under rough random design.  The unknown design density is bounded above
and away from zero, with no smoothness assumption, and the conditional error
laws may depend on the covariates and have uniformly bounded fourth
moments.
For an $s$-H\"older regression
function with $s>1$ in dimension $d>4s$, the minimax root-mean-square
risk lies between $cn^{-2(s+1)/(d+4)}e^{-C\sqrt{\log n}}$ and
$Cn^{-2(s+1)/(d+4)}$. These bounds show that the rate proposed by Robins is not
uniformly attainable over this model class. For $0<s\le1$, we show the exact minimax rate is
$n^{-1/2}\vee n^{-4s/(d+4s)}$; for $s>1$ and $d\le4s$, it is
$n^{-1/2}$. 
\end{abstract}
\maketitle
\tableofcontents

\section{Introduction}\label{sec:introduction}
We consider a problem arising from the study of higher-order influence
functions for nonlinear semiparametric functionals
\cite{RobinsEtAl2008,RobinsEtAl2009EJS}. We observe independent copies
of $(X,Y)$, where $X\in[0,1]^d$ has an unknown density $p$ bounded
above and away from zero, and let
\[
 f(x)=\E(Y\mid X=x),\qquad \Var(Y\mid X=x)=V.
\]
The regression function belongs to a fixed $s$-H\"older ball. Let
$Q_x$ denote the conditional law of $\varepsilon=Y-f(X)$ given $X=x$.
It has mean zero, variance $V$, and uniformly bounded fourth moment,
but may otherwise depend on $x$. No smoothness is imposed on $p$.
We aim to estimate $V$, which satisfies
\begin{equation}\label{eq:variance-quadratic-identity}
 V=\E(Y^2)-\int_{[0,1]^d}f(x)^2p(x)\,\mathrm dx.
\end{equation}
Under the uniform H\"older-norm and conditional fourth-moment bounds,
the empirical mean of $Y^2$ estimates $\E(Y^2)$ at the root-$n$ rate.
Thus the only potentially non-root-$n$ component in
\eqref{eq:variance-quadratic-identity} is the quadratic regression
functional. Quadratic regression functionals are studied in
\cite{HuangFan1999,LiuEtAl2021}. Higher-order influence functions and
semiparametric minimax lower bounds for related nonlinear functionals
are developed in \cite{RobinsEtAl2008,RobinsEtAl2009EJS,RobinsEtAl2017}.
The non-root-$n$ results in \cite[Section~3.3]{LiuEtAl2021} assume that
the  design density belongs to a H\"older class of smoothness
strictly greater than $2(1+\epsilon)s$, in addition to being bounded
above and away from zero.

For $0<s<1$ and $d>4s$, Robins et al.\ construct an estimator with
root-mean-square rate $n^{-4s/(d+4s)}$; see
\cite[Section~4]{RobinsEtAl2008}. Their argument uses the identity
\[
 \E\left[\frac{(Y_i-Y_j)^2}{2}\,\middle|\,X_i=x,X_j=y\right]
 =V+\frac{(f(x)-f(y))^2}{2}.
\]
This identity motivates averaging squared differences over pairs whose
covariates are within distance $h$. The expected number of such pairs
is of order $n^2h^d$. Averaging the corresponding squared differences
produces a stochastic fluctuation of order $(n^2h^d)^{-1/2}$, while
$s$-H\"older smoothness gives bias $h^{2s}$. Balancing the two terms
gives $h\asymp n^{-2/(d+4s)}$ and the stated rate. Difference-based
estimators for multivariate fixed-design regression are studied by
Munk et al.\ \cite{MunkEtAl2005}. For $s$-H\"older regression on an
equally spaced $d$-dimensional grid, the corresponding root-mean-square
minimax rate is $n^{-2s/d}\vee n^{-1/2}$. The matching lower bounds
are established by Wang et al.\ \cite{WangBrownCaiLevine2008} in one
dimension and Cai et al.\ \cite{CaiLevineWang2009} in multiple
dimensions.

Shen et al.\ \cite{ShenGaoWittenHan2020} also study variance estimation
under random design, with a common error law independent of $X$.
Their multivariate upper-bound theorem is restricted to H\"older
smoothness at most one in each coordinate, so their results do not resolve
the higher-smoothness case considered here. Their lower-bound construction also uses design densities that
vanish on sets of positive measure, so it does not yield a lower bound for
the present class. Conversely, the present lower-bound construction uses
error laws $Q_x$ that vary with $x$, so it does not establish the same
lower bound for a model with a common error law independent of $X$.

For the present random-design problem, we consider the differentiable case, which has remained open. Richardson and Rotnitzky
\cite[Section~6]{RichardsonRotnitzky2014} record the following question,
attributed to James Robins: when $s>1$ but $s/d<1/4$, does there still
exist an estimator attaining $n^{-4s/(d+4s)}$? We study the minimax
formulation of this question, requiring a uniform root-mean-square
risk bound over the model class specified below. The squared-difference
estimator of Robins et al.\ does not resolve the question: for $s>1$,
its worst-case bias is generally of order $h^2$, rather than $h^{2s}$.
This shows only that the estimator fails to exploit the additional
smoothness; it does not rule out that another estimator may attain a
better rate.

For $s>1$ and $d>4s$, we identify the minimax root-mean-square exponent
\begin{equation}\label{eq:lambda}
 \lambda=\frac{2(s+1)}{d+4}.
\end{equation}
The minimax root-mean-square risk lies between
\(cn^{-\lambda}e^{-C\sqrt{\log n}}\) and \(Cn^{-\lambda}\).
These bounds show that the rate \(n^{-4s/(d+4s)}\) is not uniformly
attainable over the stated model class, giving a negative answer to
the minimax formulation of Robins's question. The upper
bound also improves on the fixed-design minimax rate.\footnote{An earlier version of this paper, circulated under
the title ``On Rates Attainable under Random Design: A Negative Answer
to a Problem of Robins'' \cite{AronowLopatto2026}, showed that the
proposed rate is not uniformly attainable and established the exact
minimax rate for $0<s\le1$. This version
adds an estimator and strengthens the high-smoothness lower bound to
identify the minimax exponent, and supplies a shorter proof of the
low-smoothness result.}
The ratio of the upper and lower bounds is at most $Ce^{C\sqrt{\log n}}$.
For $0<s\le1$, we give a direct proof of the exact minimax rate
$n^{-1/2}\vee n^{-4s/(d+4s)}$. For $s>1$ and $d\le4s$, the same
interpolation estimator attains the parametric rate $n^{-1/2}$.

\subsection{Main results}\label{sec:main-result}
Fix \(s>0\) and an integer \(d\ge 1\). Fix an open set
\(\mathcal U\subset\mathbb R^d\) containing \([0,1]^d\), and constants
\begin{equation}\label{e:fix-constants}
 0<p_-<1<p_+<\infty,\qquad
 L_s>0,\qquad
 0<v_-<v_+<\infty,\qquad
 v_-^2<C_4<\infty.
\end{equation}

We adopt the following notation for H\"older norms.
Set
\begin{equation}\label{eq:holder-indices}
 \ell=\lceil s\rceil-1,
 \qquad \alpha=s-\ell\in(0,1].
\end{equation}
For a scalar function \(F\) on an open set
\(\Omega\subset\mathbb R^d\), define
\begin{align}
 [D^\ell F]_{\alpha;\Omega}
 &=\max_{|\gamma|=\ell}
   \sup_{\substack{u,v\in\Omega\\u\ne v}}
   \frac{|\partial^\gamma F(u)-\partial^\gamma F(v)|}
        {\|u-v\|^\alpha},
 \label{eq:holder-seminorm}\\
 \norm{F}_{C^s(\Omega)}
 &=\sum_{|\gamma|\le\ell}
   \norm{\partial^\gamma F}_{L^\infty(\Omega)}
   +[D^\ell F]_{\alpha;\Omega},
 \label{eq:holder-norm}
\end{align}
where \(\|\cdot\|\) is the usual Euclidean norm on \(\mathbb R^d\).
Thus, when \(s\notin\mathbb N\), this is the usual
\(C^{\ell,\alpha}\) norm, while integer \(s\) is interpreted as
\(C^{s-1,1}\).

For each \(n\in\mathbb N\), one observes independent and identically
distributed pairs \((X_i,Y_i)_{i=1}^n\).
Their common distribution is indexed by a parameter
\(\theta=(p,f,V,Q)\) satisfying the following three conditions.
Here, \(p\) is a probability density on \([0,1]^d\),
\(f\colon[0,1]^d\to\mathbb R\) is a regression function,
\(V>0\), and \(Q=(Q_x)_{x\in[0,1]^d}\) is a measurable
probability kernel from \([0,1]^d\) to \(\mathbb R\).
Define the regression error by
\(\varepsilon_i=Y_i-f(X_i)\).

\begin{enumerate}
\item[\textup{(i)}]
Each \(X_i\) has Lebesgue density \(p\) on \([0,1]^d\), and \(p\)
satisfies
\[
 p_-\le p(x)\le p_+
 \quad\text{for Lebesgue-a.e. }x\in[0,1]^d.
\]

\item[\textup{(ii)}]
The function \(f\) has an extension \(\widetilde f\) to
\(\mathcal U\) such that
\[
 \norm{\widetilde f}_{C^s(\mathcal U)}\le L_s,
\]
where the norm is defined in \eqref{eq:holder-norm}.

\item[\textup{(iii)}]
For \(p(x)\,\mathrm dx\)-almost every \(x\in[0,1]^d\),
the conditional distribution of \(\varepsilon_i\) given \(X_i=x\)
is \(Q_x\), and
\[
 \int_{\mathbb R}u\,Q_x(\mathrm du)=0,\qquad
 \int_{\mathbb R}u^2\,Q_x(\mathrm du)=V,\qquad
 \int_{\mathbb R}u^4\,Q_x(\mathrm du)\le C_4,
\]
with
\[
 v_-\le V\le v_+.
\]
The conditional variance \(V\) is constant in \(x\), but the
conditional distribution \(Q_x\) may otherwise depend arbitrarily
on \(x\).
\end{enumerate}

Conditional Jensen's inequality gives \(V^2\le C_4\), so the
effective variance interval is
\([v_-,\min\{v_+,\sqrt{C_4}\}]\).
The assumption \(v_-^2<C_4\) ensures that this interval is
nondegenerate.

Let
\[
 \Theta
 =\Theta(s,d,\mathcal U,p_-,p_+,L_s,v_-,v_+,C_4)
\]
denote this class of parameters \(\theta\).
For an estimator \(\widehat V_n\) of \(V\), define
\[
 \mathcal R_n(\widehat V_n;\Theta)
 =
 \sup_{\theta\in\Theta}
 \E_\theta\bigl(\widehat V_n-V\bigr)^2,
 \qquad
 R_n(\Theta)
 =
 \inf_{\widehat V_n}
 \mathcal R_n(\widehat V_n;\Theta),
\]
where the infimum is over all measurable functions of the
observations \((X_i,Y_i)_{i=1}^n\).
We write
\[
 \mathfrak r_n
 =R_n(\Theta)^{1/2}
 =\inf_{\widehat V_n}\sup_{\theta\in\Theta}
   \left\{\E_\theta\bigl(\widehat V_n-V\bigr)^2\right\}^{1/2}
\]
for the minimax root-mean-square risk.

We now state our four main results. The constants in these results depend only on the fixed model
parameters, including the extension domain $\mathcal U$, and  $c_\epsilon$ may
also depend on $\epsilon$.
\begin{theorem}\label{thm:main}\label{thm:bracket}
Let $s>1$ and let $d>4s$ be an integer. There exist $c,C>0$ such
that, for all sufficiently large $n$,
\begin{equation}\label{eq:main-bracket}
 c n^{-\lambda}e^{-C\sqrt{\log n}}
 \le\mathfrak r_n\le Cn^{-\lambda},
 \qquad \lambda=\frac{2(s+1)}{d+4}.
\end{equation}
\end{theorem}
These bounds determine the minimax exponent. We state this as a corollary.
\begin{corollary}\label{thm:exponent}
Let $s>1$ and let $d>4s$ be an integer. For every $\epsilon>0$, there
exist $c_\epsilon,C>0$ such that, for all sufficiently large $n$,
\[
 c_\epsilon n^{-\lambda-\epsilon}\le\mathfrak r_n\le Cn^{-\lambda}.
\]
The minimax root-mean-square exponent satisfies
\[
 \lim_{n\to\infty}\frac{-\log\mathfrak r_n}{\log n}
 =\frac{2(s+1)}{d+4}.
\]
\end{corollary}
The separation from the conjectured exponent in \cite[Section~6]{RichardsonRotnitzky2014} is
\begin{equation}\label{eq:intro-gap}
 \frac{4s}{d+4s}-\lambda
 =\frac{2(s-1)(d-4s)}{(d+4s)(d+4)}>0.
\end{equation}
The lower bound therefore rules out a uniform root-mean-square risk
bound of order \(n^{-4s/(d+4s)}\), even after multiplication by any fixed
power of \(\log n\). The exponent also satisfies
\[
 \lambda-\frac{2s}{d}=\frac{2(d-4s)}{d(d+4)}>0,
\]
so the upper bound is faster than the fixed-design minimax rate.

For \(0<s\le1\), the exact minimax rate can be determined without the
higher-order interpolation construction; see \zcref{app:small-s}.
\begin{theorem}\label{thm:main2}
Let $0<s\le1$ and let $d\ge1$ be an integer. There exists $C>0$
such that, for every $n\ge1$,
\begin{equation}\label{eq:main-rate}
 C^{-1}\bigl(n^{-1/2}\vee n^{-4s/(d+4s)}\bigr)
 \le\mathfrak r_n\le
 C\bigl(n^{-1/2}\vee n^{-4s/(d+4s)}\bigr).
\end{equation}
\end{theorem}

The interpolation estimator in \zcref{sec:upper} also covers the parametric regime, yielding the following theorem. 
\begin{theorem}\label{thm:parametric}
Let $s>1$ and let $d\le4s$ be a positive integer. There exists $C>0$
such that, for every $n\ge1$,
\begin{equation}\label{eq:parametric-rate}
 C^{-1}n^{-1/2}\le\mathfrak r_n\le Cn^{-1/2}.
\end{equation}
\end{theorem}

\subsection{Ideas of the proof}
We now sketch the main ideas of the lower bound in \zcref{thm:bracket}, since this is the lengthiest argument in the paper. 

\medskip
\emph{Step 1: Moment matching for the lower bound.}
For the lower bound, we compare two priors on admissible parameters under which the values of $V$ are separated, but the induced
distributions of the data are nearly indistinguishable. We take the
conditional law of $Y$ given $X=x$ to be supported on three points,
with each of its three probabilities a quadratic polynomial in
$f(x)$. Consequently, for fixed design points
$x_1,\ldots,x_k$, averaging over a prior produces density-weighted
moments of the form
\[
 \E_{\Pi}\!\left[\prod_{i=1}^k p(x_i)f(x_i)^{a_i}\right],
 \qquad a_i\in\{0,1,2\},
\]
where $\E_{\Pi}$ denotes expectation over the prior.
Our strategy is to match sufficiently many moments under the two priors so
that the resulting data distributions are close.

For $k$ observations, the
joint law involves products of response probabilities that are quadratic in
the corresponding regression values, and hence depends on mixed moments of
degree up to $2k$. The construction in \zcref{sec:local-proof} achieves the
required moment matching using polynomial interpolation together with bounded
perturbations based on Hermite polynomials. 
The interpolation step specifies how the density-weighted moments under the
two priors must differ so that, after accounting for the separation in the
variance parameter, the corresponding terms in the joint data distributions
match. Two consistency issues then remain. 
First, the targets for $k$ design points must reduce to those for
$k-1$ points when one design point is integrated out, since all of
them must arise from the same prior. Second, the interpolation
formulas initially give signed corrections, but the lower bound
requires genuine probability laws, which must be nonnegative everywhere. We first enforce the required marginal consistency and then absorb the
signed corrections into a common prior marginal for the design density. This yields a single pair
of priors, fixed before sampling, that realizes all moment targets up
to the chosen matching order, with the required accuracy.

\medskip
\emph{Step 2: Interpolation and the cost of matching.}
The difference in the variance parameter between the two priors creates
degree-two terms associated with each design point. To compensate for these
terms through the regression functions, we need a second-moment adjustment
that contributes separately at each design point without introducing mixed
terms involving two distinct design points.  After rescaling the region of the design space under consideration to the
unit cube \(Q=[0,1]^d\), fix a finite family
\(\phi=(\phi_a)_{a\in\mathcal I}\) of smooth interpolation functions. For each
configuration $\mathbf U=(U_1,\ldots,U_k)\in Q^k$ and an interpolation
cutoff $N$, we construct a positive semidefinite coefficient matrix
$E_{k,N}(\mathbf U)$ such that
\[
 \phi(U_i)^T E_{k,N}(\mathbf U)\phi(U_j)
 =w_i(\mathbf U)\mathbf 1_{\{i=j\}},\qquad 0\le w_i\le1.
\]
Thus the interpolation produces the desired diagonal covariance pattern,
except that the contribution at $U_i$ is reduced from $1$ to $w_i$.
We measure this loss by the integrated defect
\[
 \int_{Q^k}\sum_{i=1}^k(1-w_i(\mathbf U))\,d\mathbf U.
\]
The cutoff $N$ balances this defect against the size of the interpolation
coefficients: increasing $N$ makes the defect smaller but makes the
coefficients larger.

To use these matrices in the prior construction, we also need a decomposition
whose coefficients do not become too large. We write \(E_{k,N}\) as a sum or integral of products of bounded factors,
with a separate factor depending on each design point \(U_i\). This form
allows the matrix correction to be realized through the prior construction.
The total variation of the coefficients in such a representation is its
representation cost. This quantity controls the
size of the signed corrections needed to achieve the moment matching from
Step 1. We call a polynomial a \emph{cardinal polynomial at \(U_i\)} if it equals
one at \(U_i\) and zero at the other design points. At fixed \(k\), an
elementary construction using products of linear factors chosen to vanish
at the other design points gives representation cost at most
\[
 C_kN^2(1+\log N)^{k-2}
\]
and integrated defect at most
\[
 C_kN^{-d}(1+\log N)^{k-2}.
\]
For the exponent result, it suffices to work at a fixed matching order, so
these estimates are enough to prove \zcref{thm:exponent}.

To obtain the quantitative lower bound in \zcref{thm:bracket}, however, we
must match moments uniformly for every number $k\le M$ of observations in
the same rescaled region, where the matching order $M$ grows with $n$. We therefore
need bounds whose dependence on $k$ is controlled uniformly. 
The construction in
\zcref{sec:geometry} builds three-point cardinal polynomials and combines
them recursively in a balanced way, so that the constants grow only
exponentially in $k$. This removes the logarithmic factors and gives,
simultaneously for every $k$,
\[
 \text{representation cost}\le C^kN^2,
 \qquad
 \text{integrated defect}\le C^kN^{-d}.
\]
Finally, let \(\delta\) denote the amplitude of the diagonal covariance
correction imposed by the interpolation construction. Because the signed corrections needed
for moment matching grow with the representation cost, keeping both priors
nonnegative for every $k\le M$ requires
\[
 \delta=cN^{-2},\qquad c\le aD^{-M},
\]
for fixed constants $a>0$ and $D>1$.

\medskip
\emph{Step 3: Comparing the full sample.}
We next combine the regression-function perturbations constructed locally in
Steps 1 and 2 across the design space. Independent copies of these
perturbations are multiplied by smooth weights  whose squares sum to one at every point $x$. Since the variance
adjustment contributed by each local perturbation is proportional to the
square of its weight, their sum is constant in $x$. Thus the combined
construction still has constant conditional variance.

There is one complication in turning these local constructions into a
probability model for the full sample. The product of the perturbed local
design densities need not integrate exactly to one; its total mass is
random. We first replace the fixed sample size by a Poisson random sample
size with mean $\nu$. For such a Poisson sample, the random total mass of
the design density contributes an additional exponential factor to the
likelihood. We therefore reweight both priors by the same factor, chosen so that in the
Poisson mixture likelihood it cancels this random-mass contribution. After
this cancellation, the relevant mixture expectation is again taken under
the original product of independent local laws, so the local
moment-matching identities from Steps 1 and 2 apply unchanged. We then
normalize the resulting design density to integrate to one. If $h$
is the spatial scale of a local region and $\eta$ bounds the size of the
density perturbation, this final normalization changes the squared Hellinger
distance by at most order
\[
 \nu h^d\eta^4.
\]
To pass from the local comparisons to the full sample, we form a graph on
the observed design points, joining two design points whenever some local
block $Q_j$ contains both of them. Different connected components then
involve disjoint collections of local latent variables, so the conditional
response law factors over the components. We apply the local comparison
within components of each size, count such connected components, and sum
the resulting errors to obtain the full-sample comparison in
\zcref{sec:gluing}.

The scaling of this comparison explains the rate in the theorem. Recall that
$N$ is the interpolation cutoff from Step 2. Let $\delta$ denote the amplitude of the diagonal covariance correction
from Step 2, and
let $\kappa$ be the amplitude of the perturbation of the regression function.
The resulting separation between the two conditional variance parameters is
\[
 \Delta=\kappa^2\delta.
\]
Ignoring for the moment the smaller losses needed in the full proof, Step 2
gives $\delta\asymp N^{-2}$, while the $s$-H\"older constraint allows
$\kappa\asymp h^s$. Among the order $n^2$ pairs of observations, a proportion of order
$h^d$ have design points in the support of a common local weight. Thus
there are order $n^2h^d$ such pairs. The integrated interpolation defect contributes a further
factor $N^{-d}$, so the leading squared Hellinger error has order
\[
 n^2h^d\Delta^2N^{-d}
 \asymp n^2h^{d+4s}N^{-d-4}.
\]
At the heuristic scale $h^d\asymp n^{-1}$, corresponding to order one
observation per local region, keeping this error bounded requires
\[
 N\asymp n^{(d-4s)/(d(d+4))}.
\]
Substituting this choice into $\Delta\asymp h^{2s}N^{-2}$ gives
\[
 \Delta\asymp n^{-2(s+1)/(d+4)},
\]
which is the claimed minimax exponent.

The full proof must also control three remaining errors: connected components
containing more than $M$ observations, higher-degree terms beyond the
moment-matching truncation, and the error introduced when the perturbed
design density is normalized.
For the exponent statement, we keep the largest matched component size $M$
fixed and make the local regions sufficiently sparse; the resulting losses
can be absorbed into an arbitrarily small power of $n$. To obtain the
quantitative lower bound in \zcref{thm:bracket}, we instead let $M$ grow
with $n$. We retain terms through total response degree $2R$ and choose $\kappa$
small enough that the omitted higher-order terms are negligible.  At the same time, maintaining nonnegativity of the
two priors for every component size $k\le M$ forces the covariance adjustment
$\delta$ to lose a factor $e^{-CM}$. We also choose the local regions so
that their expected sample count is of order
\[
 e^{-C(\log n)/M},
\]
which makes components larger than $M$ sufficiently rare. Balancing these
two losses by taking
\[
 M\asymp\sqrt{\log n}
\]
produces the factor $e^{-C\sqrt{\log n}}$ in
\eqref{eq:main-bracket}.

\subsection{Organization and conventions}
\zcref{sec:upper} proves the upper bounds.
\zcref{sec:lower-exponent} proves the fixed-order lower bound in
\zcref{thm:exponent}, and \zcref{sec:refinement} proves the lower bound
in \zcref{thm:main} by allowing the matching order to grow. Both
arguments use the comparison and testing inputs collected in
\zcref{sec:lower-inputs}. The local priors are constructed in
\zcref{sec:local-proof}, and the full-sample comparison is proved in
\zcref{sec:gluing}. The fixed-order construction uses the elementary interpolation estimates of
\zcref{sec:simple-geometry}; the quantitative refinement uses the
uniform estimates of \zcref{sec:geometry}. The parametric lower bound
is proved at the beginning of \zcref{sec:lower-exponent}, and
\zcref{app:small-s} proves \zcref{thm:main2}.

All logarithms are natural. Constants may increase between occurrences. The total variation distance between probability measures is
$\TV(P,Q)=\sup_A|P(A)-Q(A)|$, and squared Hellinger distance is
$H^2(P,Q)=\int(\sqrt{\mathrm dP}-\sqrt{\mathrm dQ})^2$.
For a signed measure $\mu$, $\norm{\mu}_{\TV}$ denotes its full
variation norm; in particular,
$\norm{P-Q}_{\TV}=2\TV(P,Q)$ for probability measures $P,Q$.

\subsection{Acknowledgments}
P.L. was partially supported by NSF grant DMS-2450004. This paper was written
by the authors with the assistance of large language models, which included
suggesting arguments, contributing to drafting and revision, and writing code
for computational checks.

\section{An estimator from corrected response differences}\label{sec:upper}
Throughout this section, let \(s>1\) and \(d\ge1\). We construct the
estimator by dividing the design space into small blocks and using only
those blocks that contain two types of observations: a fixed collection of
points at which a local polynomial interpolant can be formed, and a pair of
points that lie much closer to each other than the diameter of the block.
The interpolation correction is exact for polynomials of coordinatewise
degree at most \(\ell\). Consequently, the deterministic part of the
corrected pair difference is the difference of two interpolation errors,
which can be controlled using the \(s\)-H\"older smoothness of \(f\).

Let
\[
 q=(\ell+1)^d.
\]
Choose \(q\) disjoint closed boxes
\(B_1,\ldots,B_q\) of positive volume in \([1/8,3/8]^d\), centered at the
points of a hypercubic lattice with \(\ell+1\) distinct nodes in each coordinate.
The boxes may be chosen sufficiently small that, for every choice of points
\(\zeta_j\in B_j\), there exist cardinal polynomials
\(L_1,\ldots,L_q\), each of coordinatewise degree at most \(\ell\), such
that
\begin{equation}\label{eq:stencil}
 L_j(\zeta_i)=\1_{\{i=j\}},\qquad
 \sup_{u\in[0,1]^d}\sum_{j=1}^q
 \bigl(|L_j(u)|+\norm{\nabla L_j(u)}_2\bigr)\le C_{d,s}.
\end{equation}
Thus the cardinal polynomials and their gradients remain uniformly bounded
when the interpolation points vary within their prescribed boxes. A specific construction will be given below.

Set
\[
 m=\lfloor n^{1/d}\rfloor,\qquad
 H=m^{-1},\qquad
 N=\left\lceil n^{(d-4s)/(d(d+4))}\right\rceil,\qquad
 k=q+2.
\]
Partition \([0,1]^d\) into \(m^d\) cubes of side length \(H\). If
\[
 b=x_b+H[0,1]^d
\]
is one such block, define its interpolation boxes to be
\[
 x_b+HB_1,\ldots,x_b+HB_q.
\]

In the reference cube \([0,1]^d\), let
\[
 R=[5/8,7/8]^d.
\]
This region is disjoint from the interpolation boxes
\(B_1,\ldots,B_q\subset[1/8,3/8]^d\), so the observations used to form
the interpolant are distinct from the two observations whose responses are
differenced. Partition \(R\) into \(N^d\) equal subcubes, each of side
length \(1/(4N)\). If \(b=x_b+H[0,1]^d\), the corresponding response-pair
region is \(x_b+HR\), and its \(N^d\) subcubes have side length
\(H/(4N)\).
Requiring the two response locations \(x,x'\) to lie in the
same such subcube therefore ensures
\[
 \|x-x'\|_2\le \frac{\sqrt d\,H}{4N}.
\]
Thus each block contains two geometrically separate regions: one supplying
the interpolation points and one supplying a pair of nearby response
locations. Boundary events may be ignored, since they have probability
zero.

For a spatial block \(b=x_b+H[0,1]^d\), we call \(b\) \emph{usable} if it contains exactly \(k=q+2\)
observations, with exactly one design point in each interpolation box
\(x_b+HB_j\), \(j=1,\ldots,q\), and exactly two further design points
lying in the same subcube of the response-pair region \(x_b+HR\).
Denote the interpolation locations by $
 z_1,\ldots,z_q$, 
with \(z_j\in x_b+HB_j\), and denote the two response-pair locations by
\(x,x'\).

Let \(\zeta_j=(z_j-x_b)/H\in B_j\), and let
\(L_1,\ldots,L_q\) be the cardinal polynomials associated with
\(\zeta_1,\ldots,\zeta_q\) in the reference cube. Define their rescalings
to \(b\) by
\[
 l_j(u)=L_j\!\left(\frac{u-x_b}{H}\right),
 \qquad u\in b.
\]
Then \(l_j(z_i)=\1_{\{i=j\}}\). Define
\begin{equation}\label{eq:statistic}
 D_b
 =
 Y(x)-Y(x')
 -
 \sum_{j=1}^q
 \bigl[l_j(x)-l_j(x')\bigr]Y(z_j),
 \qquad
 A_b
 =
 2+\sum_{j=1}^q
 \bigl[l_j(x)-l_j(x')\bigr]^2,
\end{equation}
and set
\[
 T_b=\frac{D_b^2}{A_b}.
\]

To see the role of this statistic, write \(Y(u)=f(u)+\varepsilon(u)\) at
the observed design locations. Conditional on the design, the error part of
\(D_b\) is
\[
 \varepsilon(x)-\varepsilon(x')
 -
 \sum_{j=1}^q
 \bigl[l_j(x)-l_j(x')\bigr]\varepsilon(z_j).
\]
The errors at the distinct observations are conditionally independent, each
with variance \(V\). Hence the conditional variance of this error term is
exactly \(VA_b\). The remaining deterministic term is
\[
 \left[f(x)-\sum_{j=1}^q l_j(x)f(z_j)\right]
 -
 \left[f(x')-\sum_{j=1}^q l_j(x')f(z_j)\right],
\]
the difference of the interpolation errors at \(x\) and \(x'\).
Consequently,
\[
 \E(T_b\mid X_1,\ldots,X_n)
 =
 V+\frac{1}{A_b}
 \left(
 \left[f(x)-\sum_{j=1}^q l_j(x)f(z_j)\right]
 -
 \left[f(x')-\sum_{j=1}^q l_j(x')f(z_j)\right]
 \right)^2.
\]
Thus the conditional bias of \(T_b\) is determined entirely by the
difference of these two interpolation errors.

Let \(G\) denote the number of usable blocks. We define
\[
 \widehat V_n
 =
 \frac1G\sum_{b\ {\rm usable}}T_b
\]
when \(G>0\), and set \(\widehat V_n=0\) when \(G=0\).
This is a measurable function of the observations; for instance, the two
points \(x,x'\) may be ordered by their observation indices.\footnote{Edgar Dobriban shared a similar estimator construction for the
upper bound in private communication. The results presented here were
derived independently.}

\begin{proof}[Proof of the upper bounds in \zcref{thm:main,thm:parametric}]
\leavevmode\par\medskip\noindent\emph{Step 1: Stability and polynomial reproduction.}
Let
\[
 \mathcal P_\ell
 =
 \operatorname{span}\{u^\gamma:
        0\le \gamma_r\le \ell,\ r=1,\ldots,d\},
\]
the space of polynomials of coordinatewise degree at most \(\ell\).
Its dimension is \(q=(\ell+1)^d\).

Let \(\xi_1,\ldots,\xi_q\) denote the Cartesian-product grid points at
the centers of \(B_1,\ldots,B_q\). Evaluation on these points defines a
linear map
\[
 \mathcal E_\xi:\mathcal P_\ell\longrightarrow\mathbb R^q,
 \qquad
 P\longmapsto \bigl(P(\xi_1),\ldots,P(\xi_q)\bigr).
\]
In the monomial basis of \(\mathcal P_\ell\), the matrix of
\(\mathcal E_\xi\) is a tensor product of one-dimensional Vandermonde
matrices. Since the nodes in each coordinate are distinct, this matrix is
invertible.

Now let \(\zeta_j\in B_j\), \(j=1,\ldots,q\). The corresponding
evaluation matrix depends continuously on
\((\zeta_1,\ldots,\zeta_q)\). By choosing the boxes
\(B_1,\ldots,B_q\) sufficiently small, we may therefore ensure that this
matrix is invertible for every such choice of interpolation points, with
the norm of its inverse bounded uniformly over
\(B_1\times\cdots\times B_q\).

For each \(j\), let \(L_j\in\mathcal P_\ell\) be the unique cardinal
polynomial satisfying
\[
 L_j(\zeta_i)=\1_{\{i=j\}},
 \qquad i=1,\ldots,q.
\]
Equivalently, the interpolation operator associated with
\(\zeta_1,\ldots,\zeta_q\) is
\[
 \mathcal I_\zeta g(u)
 =
 \sum_{j=1}^q g(\zeta_j)L_j(u).
\]
The uniform bound on the inverse evaluation matrix gives uniform bounds on
the coefficients of the polynomials \(L_j\). Since
\(\mathcal P_\ell\) is finite dimensional and \(u\in[0,1]^d\), these
coefficient bounds imply
\[
 \sup_{u\in[0,1]^d}
 \sum_{j=1}^q
 \bigl(|L_j(u)|+\|\nabla L_j(u)\|_2\bigr)
 \le C_{d,s},
\]
which is \eqref{eq:stencil}.

Finally, if \(P\in\mathcal P_\ell\), then both \(P\) and
\(\mathcal I_\zeta P\) belong to \(\mathcal P_\ell\) and agree at the
\(q\) interpolation points. Uniqueness of interpolation therefore gives
\[
 \mathcal I_\zeta P=P.
\]
Thus the interpolation operator reproduces every polynomial of
coordinatewise degree at most \(\ell\).

\medskip\noindent\emph{Step 2: Bias and conditional variance.}
Fix a usable block \(b=x_b+H[0,1]^d\), and let \(c_b\) denote its
center. Let \(P_b\) be the Taylor polynomial of \(f\) about \(c_b\) of
total degree \(\ell\), and write
\[
 r_b=f-P_b.
\]
Since \(f\) has \(C^s\) norm bounded by \(L_s\), the Taylor remainder
satisfies
\begin{equation}\label{eq:taylor}
 \norm{r_b}_{L^\infty(b)}\le C H^s,
 \qquad
 \norm{\nabla r_b}_{L^\infty(b)}\le C H^{s-1}.
\end{equation}
Indeed, the first estimate is the usual Taylor remainder bound for a
\(C^{\ell,\alpha}\) function. For the second, apply the same argument to
each first derivative of \(f\). When \(\alpha=1\), these estimates use
only the Lipschitz continuity of the derivatives of order \(\ell\), in
accordance with our convention \(C^s=C^{s-1,1}\) for integer \(s\);
no derivative of order \(s\) is required.

We next record the effect of placing \(x\) and \(x'\) in the same
response-pair subcube. By construction,
\[
 \norm{x-x'}_2\le \frac{\sqrt d\,H}{4N}.
\]
Moreover,
\[
 l_j(u)=L_j\!\left(\frac{u-x_b}{H}\right),
\]
so the gradient bound in \eqref{eq:stencil} gives
\[
 \sum_{j=1}^q \norm{\nabla l_j(u)}_2
 \le \frac{C}{H},
 \qquad u\in b.
\]
The mean value theorem therefore yields
\begin{equation}\label{eq:interp-difference}
 \sum_{j=1}^q |l_j(x)-l_j(x')|
 \le \frac{C}{H}\norm{x-x'}_2
 \le \frac{C}{N}.
\end{equation}
In particular,
\begin{equation}\label{eq:Ab-bounds}
 2\le A_b
 =2+\sum_{j=1}^q[l_j(x)-l_j(x')]^2
 \le C.
\end{equation}

Let
\[
 a_b=\E(D_b\mid X_{1:n})
\]
be the conditional mean of \(D_b\). Since \(P_b\) has total degree at
most \(\ell\), it belongs to the space of polynomials of coordinatewise
degree at most \(\ell\) reproduced by the interpolation operator from
Step~1. Hence
\[
 P_b(u)=\sum_{j=1}^q l_j(u)P_b(z_j),
 \qquad u\in b.
\]
Substituting \(f=P_b+r_b\) into the definition of \(D_b\) therefore
cancels the polynomial terms and gives
\[
 a_b
 =
 r_b(x)-r_b(x')
 -
 \sum_{j=1}^q
 [l_j(x)-l_j(x')]r_b(z_j).
\]
The two terms on the right have different sources. By
\eqref{eq:taylor} and the bound on \(\norm{x-x'}_2\),
\[
 |r_b(x)-r_b(x')|
 \le C H^{s-1}\norm{x-x'}_2
 \le C\frac{H^s}{N},
\]
while \eqref{eq:taylor} and \eqref{eq:interp-difference} give
\[
 \left|
 \sum_{j=1}^q[l_j(x)-l_j(x')]r_b(z_j)
 \right|
 \le C\frac{H^s}{N}.
\]
Consequently,
\begin{equation}\label{eq:ab-bound}
 |a_b|\le C\frac{H^s}{N}.
\end{equation}

Conditional on the design points, the regression errors attached to the
\(q+2\) observations in \(b\) are independent, have mean zero, and have
common variance \(V\). The random part of \(D_b\) is
\[
 \varepsilon(x)-\varepsilon(x')
 -
 \sum_{j=1}^q
 [l_j(x)-l_j(x')]\varepsilon(z_j),
\]
and hence has conditional variance \(VA_b\). It follows that
\begin{equation}\label{eq:conditional-mean}
 \E(T_b\mid X_{1:n})
 =
 \frac{VA_b+a_b^2}{A_b}
 =
 V+\frac{a_b^2}{A_b}.
\end{equation}
Thus, by \eqref{eq:Ab-bounds} and \eqref{eq:ab-bound},
\[
 \left|\E(T_b\mid X_{1:n})-V\right|
 \le C\frac{H^{2s}}{N^2}.
\]

We also need a uniform conditional variance bound. By the conditional
fourth-moment assumption and the \(L^4\) triangle inequality,
\[
 \norm{D_b}_{L^4(\,\cdot\,\mid X_{1:n})}
 \le
 |a_b|
 +C_4^{1/4}
 \left(
 2+\sum_{j=1}^q|l_j(x)-l_j(x')|
 \right)
 \le C,
\]
where we used \eqref{eq:interp-difference}. Since \(A_b\ge2\),
\[
 \E(T_b^2\mid X_{1:n})
 =
 \frac{\E(D_b^4\mid X_{1:n})}{A_b^2}
 \le C,
\]
and therefore
\begin{equation}\label{eq:conditional}
 \Var(T_b\mid X_{1:n})\le C.
\end{equation}

Finally, usable blocks involve disjoint sets of observations. Conditional
on the design, the corresponding statistics \(T_b\) are therefore
independent. On the event \(G>0\), averaging
\eqref{eq:conditional-mean} over the usable blocks and using
\eqref{eq:conditional} gives
\begin{equation}\label{eq:risk-given-design}
 \E\!\left[
   (\widehat V_n-V)^2
   \,\middle|\,X_{1:n}
 \right]
 \le
 C\left(\frac{H^{4s}}{N^4}+\frac1G\right).
\end{equation}

\medskip\noindent\emph{Step 3: The number of usable blocks.}
Recall that a usable block contains exactly one observation in each of its
$q$ interpolation boxes, exactly two observations in the same one of its
$N^d$ pair subcubes, and no other observations in the block. Thus a good
block contains
\[
 k=q+2
\]
observations in total. For a block $b$, let $u_b$ denote its probability
mass under the design density $p$, let $a_{bj}$ denote the probability mass
of its $j$th interpolation box, and let $z_{br}$ denote the probability
mass of its $r$th pair subcube. Write
\[
 (n)_k=n(n-1)\cdots(n-k+1).
\]

We first compute the probability that a fixed block is good. There are
$(n)_k/2$ ways to assign distinct observation labels to the $q$ ordered
interpolation boxes and to an unordered pair of response locations. The
factor
\[
 \prod_{j=1}^q a_{bj}
\]
is the probability that the interpolation observations fall in their
prescribed boxes, while
\[
 \sum_{r=1}^{N^d}z_{br}^2
\]
is the probability that the two response observations fall in the same
pair subcube. Finally, all remaining $n-k$ observations must lie outside
$b$, which contributes $(1-u_b)^{n-k}$. Hence
\begin{equation}\label{eq:probability}
 \pi_b:=\Pp(b\text{ good})
 =\frac{(n)_k}{2}(1-u_b)^{n-k}
        \prod_{j=1}^q a_{bj}\sum_{r=1}^{N^d}z_{br}^2.
\end{equation}

We next estimate this probability uniformly over the admissible design
densities. By the uniform upper and lower bounds on $p$, and because the
interpolation boxes occupy fixed positive fractions of a block,
\[
 u_b\asymp H^d,\qquad
 a_{bj}\asymp H^d,\qquad
 z_{br}\asymp H^dN^{-d}.
\]
Moreover, since $H=\lfloor n^{1/d}\rfloor^{-1}$,
\[
 1\le nH^d\le 2^d
\]
for all sufficiently large $n$. Since $k$ is fixed,
$(n)_k\asymp n^k$, and since $u_b\asymp n^{-1}$,
$(1-u_b)^{n-k}\asymp1$. Therefore
\[
 \sum_{r=1}^{N^d}z_{br}^2
 \asymp
 N^d(H^dN^{-d})^2
 =
 H^{2d}N^{-d},
\]
and, using $k=q+2$,
\[
 \pi_b
 \asymp
 n^k(H^d)^qH^{2d}N^{-d}
 =
 (nH^d)^kN^{-d}
 \asymp N^{-d}.
\]
There are $m^d\asymp n$ blocks, so if $G$ denotes the number of good
blocks,
\[
 \mu:=\E G=\sum_b\pi_b\asymp nN^{-d}.
\]

It remains to show that $G$ is sufficiently concentrated around its mean.
Consider two distinct blocks $b$ and $c$, and write
\[
 u=u_b,\qquad v=u_c.
\]
Because the blocks are disjoint, the factors specifying where the $k$
observations inside each block must fall are the same in the joint
probability as in $\pi_b\pi_c$ and therefore cancel in the ratio. The only
differences are that the two good-block events together use $2k$ distinct
observations, and the remaining $n-2k$ observations must lie outside
$b\cup c$. Consequently,
\begin{equation}\label{eq:ratio}
 R_{bc}
 :=
 \frac{\Pp(b,c\text{ good})}{\pi_b\pi_c}
 =
 \frac{(n)_{2k}}{(n)_k^2}
 \frac{(1-u-v)^{n-2k}}
 {(1-u)^{n-k}(1-v)^{n-k}}.
\end{equation}
We now show that $R_{bc}$ differs from one by only order $n^{-1}$.
First,
\[
 \frac{(n)_{2k}}{(n)_k^2}
 =
 \prod_{i=0}^{k-1}\left(1-\frac{k}{n-i}\right).
\]
Also,
\[
 1-u-v
 =
 (1-u)(1-v)
 \left(1-\frac{uv}{(1-u)(1-v)}\right).
\]
Taking logarithms in \eqref{eq:ratio} therefore gives
\begin{align*}
 \log R_{bc}
 &=
 \sum_{i=0}^{k-1}
 \log\left(1-\frac{k}{n-i}\right)
 +(n-2k)
 \log\left(1-\frac{uv}{(1-u)(1-v)}\right)\\
 &\hspace{2em}
 -k\log(1-u)-k\log(1-v).
\end{align*}
Since $u,v\le C/n$ and $k$ is fixed, all arguments of the logarithms
lie in $[1/2,1]$ for sufficiently large $n$. Using
$|\log(1-t)|\le2t$ for $0\le t\le1/2$, we obtain
\[
 |\log R_{bc}|
 \le
 C\left(\frac1n+nuv+u+v\right)
 \le \frac Cn.
\]
It follows that
\[
 |R_{bc}-1|\le\frac Cn.
\]
Hence
\[
 \bigl|\Cov(\1_{\{b\text{ good}\}},\1_{\{c\text{ good}\}})\bigr|
 =
 \pi_b\pi_c|R_{bc}-1|
 \le \frac{C}{n}\pi_b\pi_c.
\]
Summing the individual variances and these pairwise covariances gives
\[
 \Var(G)
 \le
 \mu+\frac Cn\mu^2.
\]
Since $\mu=\E G\le m^d\le n$, we conclude that
\[
 \Var(G)\le C\mu.
\]

Chebyshev's inequality now gives
\[
 \Pp(G<\mu/2)
 \le
 \frac{4\Var(G)}{\mu^2}
 \le
 \frac C\mu.
\]
Finally, on $\{G\ge\mu/2\}$ we have $1/G\le2/\mu$, while on
$\{0<G<\mu/2\}$ we use the trivial bound $1/G\le1$. Therefore
\begin{align}
 \Pp(G=0)
 +\E\left[\frac{\1_{\{G>0\}}}{G}\right]
 &\le
 2\Pp(G<\mu/2)+\frac2\mu
 \notag\\
 &\le\frac C\mu.
 \label{eq:inverse-count}
\end{align}

\medskip\noindent\emph{Step 4: The risk bound.}
We now combine the conditional risk bound with the control on the number of
usable blocks. Recall that $\widehat V_n=0$ when $G=0$. By the uniform
conditional fourth-moment bound,
\[
 V^2
 =
 \bigl(\E(\varepsilon^2\mid X)\bigr)^2
 \le
 \E(\varepsilon^4\mid X)
 \le C_4,
\]
so on $\{G=0\}$ the squared error is uniformly bounded by $C_4$.

On $\{G>0\}$, \eqref{eq:risk-given-design} gives
\[
 \E\!\left[
   (\widehat V_n-V)^2
   \,\middle|\,X_{1:n}
 \right]
 \le
 C\left(\frac{H^{4s}}{N^4}+\frac1G\right).
\]
Taking expectations and separating the events $\{G=0\}$ and $\{G>0\}$
therefore yields
\[
 \E_\theta(\widehat V_n-V)^2
 \le
 C\frac{H^{4s}}{N^4}
 +C\E_\theta\!\left[\frac{\1_{\{G>0\}}}{G}\right]
 +C_4\Pp_\theta(G=0).
\]
The bounds in Step~3 hold uniformly over the admissible design densities, so
\eqref{eq:inverse-count} and $\mu\asymp nN^{-d}$ imply
\[
 \Pp_\theta(G=0)
 +
 \E_\theta\!\left[\frac{\1_{\{G>0\}}}{G}\right]
 \le
 \frac C\mu
 \le
 C\frac{N^d}{n}.
\]
Hence
\begin{equation}\label{eq:upper-risk}
 \sup_{\theta\in\Theta}\E_\theta(\widehat V_n-V)^2
 \le
 C\left(\frac{H^{4s}}{N^4}+\frac{N^d}{n}\right).
\end{equation}

Suppose first that $d>4s$. Since
$H=\lfloor n^{1/d}\rfloor^{-1}\asymp n^{-1/d}$ and
\[
 N
 =
 \left\lceil n^{(d-4s)/(d(d+4))}\right\rceil
 \asymp
 n^{(d-4s)/(d(d+4))},
\]
we have
\[
 \frac{H^{4s}}{N^4}
 \asymp
 n^{-4s/d}
 n^{-4(d-4s)/(d(d+4))}
 =
 n^{-4(s+1)/(d+4)},
\]
while
\[
 \frac{N^d}{n}
 \asymp
 n^{(d-4s)/(d+4)-1}
 =
 n^{-4(s+1)/(d+4)}.
\]
Thus \eqref{eq:upper-risk} gives
\[
 \sup_{\theta\in\Theta}\E_\theta(\widehat V_n-V)^2
 \le
 Cn^{-4(s+1)/(d+4)}.
\]

If $d\le4s$, then the definition above gives $N=1$. Since
$H\asymp n^{-1/d}$,
\[
 H^{4s}\le Cn^{-4s/d}\le Cn^{-1},
\]
because $4s/d\ge1$. Therefore
\[
 \sup_{\theta\in\Theta}\E_\theta(\widehat V_n-V)^2
 \le
 C\left(n^{-4s/d}+n^{-1}\right)
 \le
 Cn^{-1}.
\]
The preceding estimates were stated for all sufficiently large $n$. For the
finitely many remaining sample sizes, the estimator identically equal to
zero has risk at most $C_4$, so enlarging the constant $C$ completes the
proof.
\end{proof}

\section{A lower bound with fixed-order moment matching}\label{sec:lower-exponent}

The proof of the lower bound in \zcref{thm:exponent} uses the same
moment-matching construction as the sharper lower bound proved later. The
difference is that, once $\epsilon>0$ is fixed, it is enough here to match
moments only up to a fixed order. We first prove the parametric
$n^{-1/2}$ lower bound, which is valid for every $s>0$. We then state the
local and full-sample comparison bounds used in both nonparametric lower
bounds and choose their parameters with the matching order fixed. The
construction of the priors and the proof that the local comparisons can be
combined over the full sample are given in
\zcref{sec:local-proof,sec:gluing}.

We begin by fixing a response distribution that will be used throughout the
lower-bound constructions. Choose constants
\[
 v_-<v<\min(v_+,\sqrt{C_4}),\qquad
 v<a_Y^2<C_4/v.
\]
Such a choice is possible because $v_-^2<C_4$. For a regression value $f$
and conditional variance $V$, let $Y$ take values in
$\{-a_Y,0,a_Y\}$ with probabilities
\begin{equation}\label{eq:ternary}
 q_{\pm a_Y}(f,V)
 =
 \frac{V+f^2\pm a_Yf}{2a_Y^2},
 \qquad
 q_0(f,V)
 =
 1-\frac{V+f^2}{a_Y^2}.
\end{equation}
These probabilities sum to one. Moreover,
\[
 \E(Y\mid X=x)
 =
 a_Y\bigl(q_{a_Y}(f,V)-q_{-a_Y}(f,V)\bigr)
 =
 f,
\]
and
\[
 \E(Y^2\mid X=x)
 =
 a_Y^2\bigl(q_{a_Y}(f,V)+q_{-a_Y}(f,V)\bigr)
 =
 V+f^2.
\]
Hence the conditional variance of $Y$ is exactly $V$. A direct calculation
also gives
\[
 \E\bigl[(Y-f)^4\mid X=x\bigr]
 =
 a_Y^2V+(6V-3a_Y^2)f^2+3f^4.
\]

At the reference point $(f,V)=(0,v)$,
\[
 q_{\pm a_Y}(0,v)=\frac{v}{2a_Y^2}>0,
 \qquad
 q_0(0,v)=1-\frac{v}{a_Y^2}>0,
\]
because $v<a_Y^2$. The centered fourth moment at this point is
$a_Y^2v<C_4$. By continuity, there are constants $\rho>0$ and $c>0$
such that whenever
\[
 |f|\le\rho,\qquad |V-v|\le\rho,
\]
with $\rho$ chosen small enough that $V\in(v_-,v_+)$, all three response
probabilities are at least $c$ and
\[
 \E\bigl[(Y-f)^4\mid X=x\bigr]\le C_4.
\]
Thus all response-distribution constraints in the model hold uniformly for
the small perturbations used below.

\begin{proof}[Proof of the lower bound in \zcref{thm:parametric}]
We prove the parametric lower bound by comparing two admissible models whose
variance parameters differ by order $n^{-1/2}$ but whose $n$-sample laws
remain a fixed distance apart. Take the uniform design density $p\equiv1$
and the zero regression function $f\equiv0$, and use the ternary response
law in \eqref{eq:ternary}. Let
\[
 V_0=v,
 \qquad
 V_1=v+\Delta_n,
 \qquad
 \Delta_n=c_0n^{-1/2},
\]
where $c_0>0$ will be chosen below. Since $v$ lies in the interior of the
admissible variance range, for sufficiently small $c_0$ both $V_0$ and
$V_1$ lie in the neighborhood considered above. The two resulting models
are therefore admissible.

Because the design distribution is the same under the two models, it is
enough to compare their response distributions. Let $P_0$ and $P_1$ denote
the corresponding one-observation laws of $Y$. Since $f=0$, their three
probabilities differ by
\[
 q_{\pm a_Y}(0,V_1)-q_{\pm a_Y}(0,V_0)
 =
 \frac{\Delta_n}{2a_Y^2},
 \qquad
 q_0(0,V_1)-q_0(0,V_0)
 =
 -\frac{\Delta_n}{a_Y^2}.
\]
All six probabilities are bounded below by a fixed positive constant when
$c_0$ is sufficiently small. Hence, using
\[
 (\sqrt{x}-\sqrt{y})^2
 =
 \frac{(x-y)^2}{(\sqrt{x}+\sqrt{y})^2},
\]
the squared Hellinger distance satisfies
\[
 H^2(P_0,P_1)
 \le
 C\sum_{y\in\{-a_Y,0,a_Y\}}
 \bigl(P_1\{Y=y\}-P_0\{Y=y\}\bigr)^2
 \le
 C\Delta_n^2.
\]

For $n$ independent observations, the product bound for Hellinger distance
gives
\[
 H^2(P_0^{\otimes n},P_1^{\otimes n})
 \le
 nH^2(P_0,P_1)
 \le
 Cn\Delta_n^2
 =
 Cc_0^2.
\]
Since total variation distance is bounded by a constant multiple of
Hellinger distance, we may choose $c_0>0$, independently of $n$, so that
\[
 \operatorname{TV}(P_0^{\otimes n},P_1^{\otimes n})\le\frac14.
\]
The two values of the target parameter are separated by
\[
 |V_1-V_0|=\Delta_n=c_0n^{-1/2}.
\]
Le Cam's two-point inequality therefore implies
\begin{equation}\label{eq:parametric-lower}
 \mathfrak r_n\ge cn^{-1/2}.
\end{equation}

The construction uses the constant regression function $f\equiv0$, which
belongs to every H\"older class under consideration, and the argument does
not otherwise depend on $s$ or $d$. Thus the lower bound holds for every
$s>0$ and every $d\ge1$.
\end{proof}

\subsection{Comparison and testing inputs}\label{sec:lower-inputs}

The local comparison will be carried out only for configurations containing
at most a prescribed number of observations, and in the sharper argument we
will also truncate the response expansion. We therefore keep two cutoffs
separate. Let $M$ denote the largest observation count for which we retain a
local comparison, and let $2R$ denote the largest total degree retained in
the response variables. Write
\[
 [m]=\{1,\ldots,m\},\qquad [0]=\varnothing.
\]
For formal variables $u=(u_1,\ldots,u_m)$ and a multi-index
$a=(a_1,\ldots,a_m)$, write
\[
 |a|=\sum_{i=1}^m a_i,
 \qquad
 [u^a]A
\]
for the coefficient of $u^a=\prod_i u_i^{a_i}$ in a formal series $A(u)$.
Since each ternary response probability is quadratic in the corresponding
regression value, only powers $u_i^0,u_i^1,u_i^2$ can enter the response
calculation. We therefore define
\begin{equation}\label{eq:response-truncation}
 \mathcal A_{m,R}
 =
 \{a\in\{0,1,2\}^m:|a|\le2R\},
 \qquad
 \mathcal T_{m,R}A
 =
 \sum_{a\in\mathcal A_{m,R}}([u^a]A)u^a.
\end{equation}
Thus $\mathcal T_{m,R}$ discards every monomial having some individual
degree greater than two or total degree greater than $2R$. Whenever we
multiply formal series below, we apply $\mathcal T_{m,R}$ after the
multiplication. Because multiplication only increases the degree of every
variable, a discarded monomial can never contribute later to a coefficient
retained by $\mathcal T_{m,R}$. This is what we mean by carrying out the
calculation in the \emph{retained algebra}.

If $R=M$, then for every $m\le M$ we retain all monomials that can enter an
$m$-observation response probability, since their total degree is at most
$2m\le2M$. This choice will be enough for the fixed-order lower bound in
this section. In \zcref{sec:refinement}, we instead let $M$ grow while
keeping $R$ fixed, and we will then have to bound the discarded
higher-degree terms. Whenever an $L^2$ norm of a coefficient function in
the design variables is used below, the norm is taken with respect to
Lebesgue measure on the corresponding product of reference cubes.

We next explain how these formal coefficients encode the actual ternary
response probabilities. Recall from \eqref{eq:ternary} that, for
$y\in\{-a_Y,0,a_Y\}$,
\[
 q_y(f,V)
 =
 \beta_y+\ell_y f+r_y(f^2+V-v),
\]
where
\[
 (\beta_{\pm a_Y},\ell_{\pm a_Y},r_{\pm a_Y})
 =
 \left(
   \frac{v}{2a_Y^2},
   \frac{\pm1}{2a_Y},
   \frac1{2a_Y^2}
 \right),
 \qquad
 (\beta_0,\ell_0,r_0)
 =
 \left(
   1-\frac v{a_Y^2},
   0,
   -\frac1{a_Y^2}
 \right).
\]
For a response vector
$y=(y_1,\ldots,y_m)\in\{-a_Y,0,a_Y\}^m$, define the coefficient
extraction map
\begin{equation}\label{eq:response-map}
 \mathcal Q_y A
 =
 \prod_{i=1}^m
 \bigl(
   \beta_{y_i}[u_i^0]
   +\ell_{y_i}[u_i^1]
   +2r_{y_i}[u_i^2]
 \bigr)A.
\end{equation}
Here $[u_i^j]$ extracts the coefficient of $u_i^j$ while leaving the
other variables unchanged, and the product applies this operation in each
of the $m$ variables. The factor $2$ in front of $[u_i^2]$ is needed
because
\[
 [u_i^2]e^{u_iF(x_i)}=\frac{F(x_i)^2}{2}.
\]
Consequently, $\mathcal Q_y$ turns the degree-one and degree-two
coefficients of the exponential series into the terms proportional to
$F$ and $F^2$ that appear in the response probability.

To incorporate the difference between the two variance parameters, consider
two local priors, denoted by $+$ and $-$, on a design density $p$ and an
auxiliary function $F$, which will later be scaled to form the regression
function $f=\kappa F$. Let $\E_+$ and $\E_-$ denote expectation under
these priors, and let $\delta>0$ denote the size of the covariance
adjustment that will separate the two variance parameters. For
$x=(x_1,\ldots,x_m)$, define the untruncated formal series
\[
 A_{\pm,m}(x;u)
 =
 e^{\mp\delta\sum_{i=1}^m u_i^2/4}
 \E_\pm\!\left[
   \prod_{i=1}^m p(x_i)
   e^{\sum_{i=1}^m u_iF(x_i)}
 \right].
\]
The exponential prefactor is chosen to encode the variance shift. Indeed,
if the actual regression perturbation has amplitude $\kappa$, so that
\[
 f=\kappa F,
 \qquad
 V_\pm=v\mp\frac{\kappa^2\delta}{2},
\]
then applying $\mathcal Q_y$ to $A_{\pm,m}(x;\kappa u)$ produces the
density-weighted joint probability of observing responses
$y_1,\ldots,y_m$ at design points $x_1,\ldots,x_m$. To see the
second-order term explicitly, the coefficient of $u_i^2$ in the
compensated exponential equals
\[
 \frac12F(x_i)^2\mp\frac{\delta}{4}.
\]
After multiplication by the factor $2r_{y_i}$ in
\eqref{eq:response-map} and by $\kappa^2$, this becomes
\[
 r_{y_i}
 \left(
   \kappa^2F(x_i)^2
   \mp\frac{\kappa^2\delta}{2}
 \right)
 =
 r_{y_i}\bigl(f(x_i)^2+V_\pm-v\bigr),
\]
which is exactly the quadratic term in $q_{y_i}(f(x_i),V_\pm)$.

This calculation also explains the second-moment matching condition used
later. For $m=1$, equality of the retained $u^2$ coefficients under the
two priors requires
\[
 \E_+[p(x)F(x)^2]-\E_-[p(x)F(x)^2]
 =
 \delta\rho_1(x),
 \qquad
 \rho_1(x)=\E_\pm p(x),
\]
where the two priors have the same design marginal $\rho_1$. Thus the
difference in the second moments of $F$ exactly compensates the difference
between $V_+$ and $V_-$. The variables $u_i$ are only formal devices for
organizing these moment identities; the observed responses themselves
remain supported on $\{-a_Y,0,a_Y\}$.

When $R<m$, truncating the formal series omits some terms in the actual
response probability. We now define this remainder directly. Introduce a
scalar amplitude $t$ and, for either prior, set
\begin{equation}\label{eq:response-remainder-definition}
 \begin{aligned}
 J_{\pm,m}(x,y;t)
 &=
 \E_\pm\prod_{i=1}^m
 p(x_i)\,
 q_{y_i}\!\left(tF(x_i),v\mp\frac{t^2\delta}{2}\right),\\
 J_{\pm,m}^{\le2R}(x,y;\kappa)
 &=
 \sum_{r=0}^{\min(R,m)}
 \kappa^{2r}[t^{2r}]J_{\pm,m}(x,y;t),\\
 \operatorname{Rem}_{\pm,m}(x,y;\kappa)
 &=
 J_{\pm,m}(x,y;\kappa)
 -
 J_{\pm,m}^{\le2R}(x,y;\kappa).
 \end{aligned}
\end{equation}
The priors constructed below are conditionally symmetric under
$F\mapsto-F$, so $J_{\pm,m}(x,y;t)$ is an even polynomial in $t$.
The quantity $J_{\pm,m}^{\le2R}$ therefore keeps exactly the terms of
degree at most $2R$ in this polynomial.

The formal-series description and the direct response expansion agree.
Indeed,
\begin{equation}\label{eq:full-retained-response-map}
 \begin{aligned}
 J_{\pm,m}(x,y;\kappa)
 &=
 \mathcal Q_y\,\mathcal T_{m,m}
 \bigl[A_{\pm,m}(x;\kappa u)\bigr],\\
 J_{\pm,m}^{\le2R}(x,y;\kappa)
 &=
 \mathcal Q_y\,\mathcal T_{m,R}
 \bigl[A_{\pm,m}(x;\kappa u)\bigr].
 \end{aligned}
\end{equation}
The first identity uses $\mathcal T_{m,m}$ only to remove monomials that
cannot affect a response probability; since the total response degree is at
most $2m$, no relevant term is lost. The second identity retains only total
degree at most $2R$. Their difference is precisely
$\operatorname{Rem}_{\pm,m}$, and this remainder vanishes whenever
$m\le R$.

Finally, for the priors used below, the design marginal
\[
 \rho_m(x)
 =
 \E_\pm\prod_{i=1}^m p(x_i)
\]
is the same under the two signs. The quantities $J_{\pm,m}$ are
density-weighted joint response probabilities; dividing by $\rho_m(x)$
gives the corresponding conditional response probabilities given the design
points. Since the design densities need only be defined almost everywhere,
all uniform bounds in the design variables may be interpreted as essential
suprema.

The next lemma constructs genuine nonnegative priors for which the retained
response coefficients agree for all observation counts in a prescribed
finite range. Its interpolation bounds may deteriorate with that fixed
range, but this affects only constants and is therefore harmless for the
minimax exponent proved in this section.

\begin{lemma}\label{lem:local}
Fix an integer $M\ge3$ and a sufficiently small constant $\eta>0$. For
$N\ge2$, define
\[
 \Lambda_N=(1+\log N)^{M-2},
 \qquad
 \delta=a_MN^{-2}\Lambda_N^{-1},
\]
where $a_M>0$ is chosen sufficiently small, independently of $N$.

For each $N\ge2$, there exist two probability laws $\Pi_+$ and $\Pi_-$
on pairs $(p,F)$, where $p$ is a design density and $F$ is an auxiliary
function that will later be scaled to form the regression function
$f=\kappa F$. These laws have the following properties.

First, under either law, $p$ is a probability density on
$Q=[0,1]^d$ satisfying
\[
 \int_Q p(x)\,dx=1,
 \qquad
 |p(x)-1|\le\eta
 \quad\text{for a.e. }x\in Q.
\]
In addition, the random density $p$ has the same distribution under
$\Pi_+$ and $\Pi_-$.

Second, there is an open neighborhood $\mathcal O\supset Q$, independent of
$N$, and a constant $L_M<\infty$, also independent of $N$, such that under
either law, $F$ almost surely extends to $\mathcal O$ and satisfies
\[
 \|F\|_{C^s(\mathcal O)}\le L_M.
\]
Moreover, under either law and conditional on $p$, $F$ and $-F$ have the
same distribution.

Let $\E_+$ and $\E_-$ denote expectations under these two laws. For
$1\le k\le M$, $x=(x_1,\ldots,x_k)\in Q^k$, and formal variables
$u=(u_1,\ldots,u_k)$, define
\[
 \mathcal M_{\pm,k}(x;u)
 =
 \E_\pm\!\left[
   \prod_{i=1}^k p(x_i)
   \exp\!\left\{\sum_{i=1}^k u_iF(x_i)\right\}
 \right]
\]
and
\[
 D_k(x;u)
 =
 \mathcal T_{k,M}\!\left[
 e^{-\delta\sum_{i=1}^k u_i^2/4}\mathcal M_{+,k}(x;u)
 -
 e^{ \delta\sum_{i=1}^k u_i^2/4}\mathcal M_{-,k}(x;u)
 \right].
\]
For one design point, all coefficients retained by $\mathcal T_{1,M}$
agree exactly under the two constructions:
\[
 D_1\equiv0.
\]
For every $2\le k\le M$, their coefficient differences 
satisfy
\begin{equation}\label{eq:fixed-contract}
 \sum_{a\in\mathcal A_{k,M}}
 \bigl\|[u^a]D_k\bigr\|_{L^2(Q^k)}^2
 \le
 C_M\delta^2N^{-d}\Lambda_N,
\end{equation}
where $C_M$ is independent of $N$.

The estimate remains valid if each variable $u_i$ is replaced by
$w_i(x_i)u_i$, where $w_i:Q\to\mathbb R$ is any measurable function
satisfying $|w_i|\le1$. More precisely, define
\[
 D_k^{(w)}(x;u)
 =
 D_k\bigl(x;
          w_1(x_1)u_1,\ldots,w_k(x_k)u_k\bigr).
\]
Then
\[
 \sum_{a\in\mathcal A_{k,M}}
 \bigl\|[u^a]D_k^{(w)}\bigr\|_{L^2(Q^k)}^2
 \le
 C_M\delta^2N^{-d}\Lambda_N,
 \qquad 2\le k\le M.
\]
\end{lemma}

The coefficients controlled by the lemma are exactly the coefficients that
enter the joint response probabilities after the variance shift has been
incorporated into the formal series. For example, a coefficient containing
$u_i^2$ involves the density-weighted second moment
$p(x_i)F(x_i)^2$ together with the compensating term arising from the
difference between the two variance parameters. Thus $D_1=0$ says that all
terms relevant to the response law at a single design point agree exactly
under the two priors. For $2\le k\le M$,
\eqref{eq:fixed-contract} bounds, in squared $L^2(Q^k)$ norm, the total
difference between the corresponding retained coefficients for $k$ design
points.

The proof of the lemma has two separate parts. In
\zcref{sec:simple-geometry}, we construct the interpolation objects that
produce the factor $N^{-d}\Lambda_N$ in
\eqref{eq:fixed-contract}. These interpolation formulas initially specify
signed corrections to the required moments. In \zcref{sec:local-proof}, we
show that the corrections can be realized by actual probability laws
$\Pi_+$ and $\Pi_-$ on $(p,F)$, while keeping the distribution of $p$ the
same under the two laws and, conditional on $p$, making $F$ and $-F$
equally distributed. Finally, \zcref{sec:gluing} applies independent copies
of these local constructions across the design space and shows that the
resulting full-sample distributions remain close after correcting the random
total mass and renormalizing the design density.

\begin{lemma}\label{lem:gluing}
Fix $s>0$ and integers $d,R\ge1$. There exists $\eta_0>0$, depending
only on $s,d,R$, the lower and upper bounds on the design density, the
H\"older-radius bound for the regression function, the variance bounds
$v_-,v_+$, the fourth-moment bound $C_4$, and the fixed constants
$v,a_Y$ in \eqref{eq:ternary}, with the following property.

For every $K,C_{\rm loc}\ge1$, there exist constants
\[
 c_f>0,\qquad C<\infty,
\]
depending only on $s,d,R,K,C_{\rm loc}$ and the same fixed model
constants.

Let $M\ge2$, $0<\eta\le\eta_0$, $0<\delta\le1$, and $\xi\ge0$.
Suppose that $\Pi_+$ and $\Pi_-$ are two probability laws on pairs
$(p,F)$, where $p$ is a probability density on the unit cube
$Q=[0,1]^d$ and $F$ is a real-valued function on $Q$. Assume the
following.

Under either law,
\[
 \int_Q p(x)\,dx=1,
 \qquad
 |p(x)-1|\le\eta
 \quad\text{for a.e. }x\in Q.
\]
The random density $p$ has the same distribution under $\Pi_+$ and
$\Pi_-$. There is a fixed neighborhood $\mathcal O\supset Q$ such that,
under either law, $F$ almost surely extends to $\mathcal O$ and satisfies
\[
 \|F\|_{C^s(\mathcal O)}\le K.
\]
Moreover, under either law and conditional on $p$, the functions $F$ and
$-F$ have the same distribution.

Let $\E_+$ and $\E_-$ denote expectation under $\Pi_+$ and $\Pi_-$.
For $1\le k\le M$, $x=(x_1,\ldots,x_k)\in Q^k$, and formal variables
$u=(u_1,\ldots,u_k)$, define
\[
 \mathcal M_{\pm,k}(x;u)
 =
 \E_\pm\!\left[
   \prod_{i=1}^k p(x_i)
   \exp\!\left\{\sum_{i=1}^k u_iF(x_i)\right\}
 \right]
\]
and
\[
 D_k(x;u)
 =
 \mathcal T_{k,R}\!\left[
 e^{-\delta\sum_{i=1}^k u_i^2/4}\mathcal M_{+,k}(x;u)
 -
 e^{ \delta\sum_{i=1}^k u_i^2/4}\mathcal M_{-,k}(x;u)
 \right].
\]
Assume that the retained coefficients agree exactly for one design point,
\[
 D_1\equiv0,
\]
and that, for every $2\le k\le M$,
\[
 \sum_{a\in\mathcal A_{k,R}}
 \bigl\|[u^a]D_k\bigr\|_{L^2(Q^k)}^2
 \le
 C_{\rm loc}^k\delta^2\xi.
\]
Assume also that this estimate remains valid after replacing each $u_i$ by
$w_i(x_i)u_i$. More precisely, for any measurable functions
$w_i:Q\to\mathbb R$ satisfying $|w_i|\le1$, define
\[
 D_k^{(w)}(x;u)
 =
 D_k\bigl(
 x;
 w_1(x_1)u_1,\ldots,w_k(x_k)u_k
 \bigr).
\]
Then assume
\[
 \sum_{a\in\mathcal A_{k,R}}
 \bigl\|[u^a]D_k^{(w)}\bigr\|_{L^2(Q^k)}^2
 \le
 C_{\rm loc}^k\delta^2\xi,
 \qquad 2\le k\le M.
\]

Now let $\nu>0$ and let $0<h\le1$ satisfy $h^{-1}\in\mathbb N$. Set
\[
 \kappa=c_fh^s,
 \qquad
 \Delta=\kappa^2\delta,
\]
where $c_f$ is chosen above sufficiently small that $\kappa\le1$.
If
\[
 C\nu h^d\le\frac12,
\]
then there exist two priors on the model class $\Theta$ for which the
conditional variance is deterministic and equal respectively to
\[
 v-\frac{\Delta}{2}
 \qquad\text{and}\qquad
 v+\frac{\Delta}{2}.
\]

Let $\overline P_\pm^{\rm Pois}$ denote the corresponding Poisson mixture
laws defined as follows. First draw a parameter from the respective prior.
Conditional on that parameter, draw a sample size
$N_\nu\sim\operatorname{Pois}(\nu)$ independently, and then draw
$N_\nu$ i.i.d.\ observations from the resulting normalized model. The squared
Hellinger distance between these two mixture laws satisfies
\begin{equation}\label{eq:global-budget}
 \begin{aligned}
 H^2(\overline P_+^{\rm Pois},\overline P_-^{\rm Pois})
 &\le
 C\bigl\{
   \nu^2h^d\kappa^4\delta^2\xi
   +\nu\kappa^{4R+4}(C\nu h^d)^R\\
 &\hspace{6em}
   +\nu\kappa^4(C\nu h^d)^M
   +\nu h^d\eta^4
 \bigr\}.
 \end{aligned}
\end{equation}
If $R\ge M$, the term
\[
 \nu\kappa^{4R+4}(C\nu h^d)^R
\]
may be omitted.

All constants above are independent of
$M,h,\delta,\xi,\eta,\nu$ and of any interpolation cutoff $N$ that may
enter the construction of the local priors.
\end{lemma}

Apply \zcref{lem:gluing} to the two local priors supplied by
\zcref{lem:local}. In \zcref{lem:local} all response coefficients are
retained for observation counts up to $M$, so we take $R=M$. Set
\[
 K=L_M,\qquad
 C_{\rm loc}=\max\{1,C_M^{1/2}\},\qquad
 \xi=N^{-d}\Lambda_N.
\]
Then, for every $2\le k\le M$,
\[
 C_M\delta^2N^{-d}\Lambda_N
 \le
 C_{\rm loc}^k\delta^2\xi,
\]
so the coefficient bound in \zcref{lem:local} has exactly the form
required by \zcref{lem:gluing}. 
The resulting two global priors have deterministic variance parameters
separated by
\[
 \Delta=\kappa^2\delta,
\]
and \zcref{lem:gluing} gives
\begin{equation}\label{eq:fixed-global}
 \begin{aligned}
 H^2(\overline P_+^{\rm Pois},\overline P_-^{\rm Pois})
 \le C_M\bigl\{
 &\nu^2h^d\kappa^4\delta^2N^{-d}\Lambda_N\\
 &+\nu\kappa^4(C_M\nu h^d)^M
 +\nu h^d\eta^4
 \bigr\}.
 \end{aligned}
\end{equation}
Here the Poissonized experiment has a random sample size with mean $\nu$.
There is no response-truncation term in \eqref{eq:fixed-global}: for a
configuration of $m\le M$ observations, the product of the $m$ quadratic
response probabilities has total degree at most $2m\le2M$, and the choice
$R=M$ therefore retains every term.

We next transfer this comparison to the fixed-$n$ experiment. Set
$\nu=2n$. From a sample generated by the Poissonized experiment, independently
randomize the order of the observations. If the Poisson sample contains at
least $n$ observations, retain the first $n$; otherwise output a distinguished
failure symbol $\dagger$. Since the Poisson sample size is independent of
the parameter, the success probability
\[
 p_n=\Pp\{\operatorname{Pois}(2n)\ge n\}
\]
is the same under both priors and satisfies $p_n\to1$. Conditional on
success and on any fixed parameter value, the retained observations are
$n$ iid observations from that parameter's normalized model. Hence, if
$\overline P_{\pm,n}$ denote the corresponding fixed-$n$ mixture laws, the
output laws of this transformation are
\[
 p_n\overline P_{\pm,n}+(1-p_n)\delta_\dagger.
\]
Hellinger distance contracts under this transformation, while the common
failure atom contributes no discrepancy. Therefore
\begin{equation}\label{eq:poisson-transfer}
 p_n
 H^2(\overline P_{+,n},\overline P_{-,n})
 \le
 H^2(\overline P_+^{\rm Pois},\overline P_-^{\rm Pois}).
\end{equation}

It remains to convert closeness of the two fixed-$n$ mixture laws into a
risk lower bound. Suppose the two priors are supported on models with
deterministic variance values $V_+$ and $V_-$ satisfying
\[
 |V_+-V_-|=\Delta.
\]
Given any estimator $\widehat V$, define a test by comparing $\widehat V$
with the midpoint $(V_++V_-)/2$. Whenever this test chooses the wrong
hypothesis, the estimation error under the true hypothesis is at least
$\Delta/2$. Since both priors are supported on $\Theta$, the worst-case
squared risk is at least the larger of the two corresponding Bayes risks.
Consequently,
\begin{align}
 \sup_{\theta\in\Theta}
 \E_\theta(\widehat V-V)^2
 &\ge
 \frac{\Delta^2}{8}
 \left(
   1-\TV(\overline P_{+,n},\overline P_{-,n})
 \right).
 \label{eq:testing}
\end{align}
Indeed, the sum of the two testing error probabilities is at least
$1-\TV(\overline P_{+,n},\overline P_{-,n})$. Thus, whenever
\[
 \TV(\overline P_{+,n},\overline P_{-,n})\le\frac12,
\]
we obtain
\[
 \sup_{\theta\in\Theta}
 \E_\theta(\widehat V-V)^2
 \ge
 \frac{\Delta^2}{16}.
\]
Taking the infimum over estimators and then square roots gives
\[
 \mathfrak r_n\ge\frac{\Delta}{4}.
\]

\subsection{The fixed-order choice of parameters}\label{sec:fixed-order-proof}

\begin{proof}[Proof of the lower bound in \zcref{thm:exponent}]
Recall that we are in the regime $d>4s$, and write
\[
 \lambda=\frac{2(s+1)}{d+4}.
\]
Fix $\epsilon>0$. We will choose the largest matched observation count $M$
once and for all, and then let the remaining parameters depend on $n$.
Choose an integer $M$ large enough that
\[
 M>\frac{2(d+4s)}{d-4s},
 \qquad
 \frac{4(s-1)}{M(d+4)}<\frac{\epsilon}{2}.
\]
The first inequality will ensure that the interpolation cutoff $N$ tends to
infinity, while the second makes the loss caused by keeping $M$ fixed
smaller than $\epsilon/2$ in the final exponent.

With this $M$ fixed, first choose a fixed $0<\eta\le\eta_0$, where
$\eta_0$ is supplied by \zcref{lem:gluing}. Apply \zcref{lem:local}
with this value of $\eta$, and let $L_M$ and $C_M$ be the resulting
smoothness and coefficient constants. We may then apply
\zcref{lem:gluing} with $R=M$, $K=L_M$, and
$C_{\rm loc}=\max\{1,C_M^{1/2}\}$. All constants are now fixed and
may depend on $M$ but not on $n$.

Set
\[
 \nu=2n,
 \qquad
 h=\left\lceil\nu^{(1+2/M)/d}\right\rceil^{-1},
 \qquad
 \kappa=c_fh^s,
\]
and choose the interpolation cutoff
\[
 N
 =
 \left\lceil
   \bigl(\nu^2h^{d+4s}\bigr)^{1/(d+4)}
 \right\rceil.
\]
The definition of $h$ gives
\[
 h\asymp \nu^{-(1+2/M)/d},
 \qquad
 \nu h^d\asymp \nu^{-2/M}.
\]
In particular, $\nu h^d\to0$, so the sparsity condition
$C_M\nu h^d\le1/2$ required by \zcref{lem:gluing} holds for all
sufficiently large $n$.

We next verify that $N$ grows polynomially. Since
\[
 \bigl(\nu^2h^{d+4s}\bigr)^{1/(d+4)}
 \asymp \nu^{b_M},
\]
where
\[
 b_M
 =
 \frac{1-4s/d-(2/M)(1+4s/d)}{d+4},
\]
the first condition on $M$ is exactly the condition $b_M>0$. Hence
\[
 N\asymp\nu^{b_M}\longrightarrow\infty.
\]

We now check the Hellinger bound in \eqref{eq:fixed-global}. Recall that
\[
 \delta=a_MN^{-2}\Lambda_N^{-1},
 \qquad
 \Lambda_N=(1+\log N)^{M-2}.
\]
For the first term in \eqref{eq:fixed-global},
\begin{align*}
 \nu^2h^d\kappa^4\delta^2N^{-d}\Lambda_N
 &=
 C_M\nu^2h^{d+4s}
      N^{-d-4}\Lambda_N^{-1}\\
 &=
 O_M(\Lambda_N^{-1}),
\end{align*}
because $N^{d+4}\asymp\nu^2h^{d+4s}$. Since $M\ge3$ and
$N\to\infty$, this term tends to zero.

For the second term, using $\nu h^d\asymp\nu^{-2/M}$ gives
\[
 \nu\kappa^4(C_M\nu h^d)^M
 =
 O_M\!\left(
   \nu\kappa^4\nu^{-2}
 \right)
 =
 O_M(\nu^{-1}\kappa^4),
\]
which also tends to zero. Finally,
\[
 \nu h^d\eta^4=O_M(\nu^{-2/M})\longrightarrow0.
\]
It follows from \eqref{eq:fixed-global} that
\[
 H^2(\overline P_+^{\rm Pois},\overline P_-^{\rm Pois})
 \longrightarrow0.
\]
Since $p_n=\Pp\{\operatorname{Pois}(2n)\ge n\}\to1$,
\eqref{eq:poisson-transfer} implies
\[
 H^2(\overline P_{+,n},\overline P_{-,n})\longrightarrow0.
\]
The total variation distance therefore also tends to zero, and in
particular
\[
 \TV(\overline P_{+,n},\overline P_{-,n})\le\frac12
\]
for all sufficiently large $n$.

It remains to compute the separation between the two variance parameters.
By construction,
\[
 \Delta=\kappa^2\delta
 =
 c_f^2a_Mh^{2s}N^{-2}\Lambda_N^{-1}.
\]
Using
\[
 h\asymp\nu^{-(1+2/M)/d},
 \qquad
 N\asymp\nu^{b_M},
\]
we obtain
\[
 h^{2s}N^{-2}
 \asymp
 \nu^{-\lambda-\frac{4(s-1)}{M(d+4)}}.
\]
Also $N\asymp\nu^{b_M}$ with $b_M>0$, so
\[
 \Lambda_N
 \asymp_M
 (1+\log\nu)^{M-2}.
\]
Consequently,
\[
 \Delta
 \ge
 c_M
 \nu^{-\lambda-\frac{4(s-1)}{M(d+4)}}
 (1+\log\nu)^{-(M-2)}.
\]
By our choice of $M$,
\[
 \frac{4(s-1)}{M(d+4)}<\frac{\epsilon}{2}.
\]
Since $M$ is now fixed,
\[
 (1+\log\nu)^{M-2}\le \nu^{\epsilon/2}
\]
for all sufficiently large $\nu$. Therefore
\[
 \Delta
 \ge
 c_\epsilon\nu^{-\lambda-\epsilon}
 \ge
 c_\epsilon' n^{-\lambda-\epsilon}.
\]

The testing bound \eqref{eq:testing} now applies and gives
\[
 \mathfrak r_n
 \ge
 \frac{\Delta}{4}
 \ge
 c_\epsilon n^{-\lambda-\epsilon}.
\]
Thus, for every $\epsilon>0$, the minimax root-mean-square risk is bounded
below by $c_\epsilon n^{-\lambda-\epsilon}$. Combining this with the upper
bound of order $n^{-\lambda}$ and then letting $\epsilon\downarrow0$
proves the claimed minimax exponent.
\end{proof}

\section{Refinement to the stretched-exponential bracket}\label{sec:refinement}

The proof in \zcref{sec:lower-exponent} keeps the largest matched
observation count $M$ fixed. For that purpose, the logarithmic factor
$\Lambda_N=(1+\log N)^{M-2}$ in \zcref{lem:local} can be absorbed into an
arbitrarily small polynomial loss and therefore does not change the
limiting exponent. To obtain the explicit
stretched-exponential loss in \zcref{thm:bracket}, however, we will let
$M$ grow with $n$. We therefore need a local interpolation construction
whose constants are controlled uniformly in $M$. The sharper geometric
estimate needed for this purpose is proved in \zcref{sec:geometry}.

We continue to use the truncation operator $\mathcal T_{k,R}$ from
\eqref{eq:response-truncation}, but now the two cutoffs play different
roles: the response-degree cutoff $R$ is fixed, while the largest
observation count $M$ is allowed to grow. Consequently, when $k>R$ the
truncated formal series no longer contains the entire $k$-observation
response probability, and the omitted terms must also be controlled.
The following lemma supplies both the required coefficient bound and the
corresponding response-remainder bound. Its proof is given in
\zcref{sec:local-applications}.

\begin{lemma}\label{lem:uniform-local}
Fix $s>1$ and an integer $d\ge3$, and fix an integer $R\ge1$.
There exists $\eta_1>0$ such that, for every fixed
$0<\eta\le\eta_1$, there are constants
\[
 0<a\le1,\qquad D>1,\qquad C<\infty,\qquad L<\infty,
\]
depending only on $s,d,R,\eta$, the variance bounds $v_-,v_+$,
the fourth-moment bound $C_4$, and the fixed constants $v,a_Y$ in
\eqref{eq:ternary}, but not on $M$ or $N$, with the following property.

For every pair of integers $M,N\ge2$ and every
\[
 0<c\le aD^{-M},
 \qquad
 \delta=cN^{-2},
\]
there exist two probability laws $\Pi_+$ and $\Pi_-$ on pairs $(p,F)$,
where $p$ is a probability density on the reference cube
$Q=[0,1]^d$ and $F$ is an auxiliary function that will later be scaled
to form the regression function.

Under either law,
\[
 \int_Q p(x)\,dx=1,
 \qquad
 |p(x)-1|\le\eta
 \quad\text{for a.e. }x\in Q,
\]
and the random density $p$ has the same distribution under $\Pi_+$ and
$\Pi_-$. There is a fixed open neighborhood $\mathcal O\supset Q$,
independent of $M$ and $N$, such that under either law, $F$ almost
surely extends to $\mathcal O$ and satisfies
\[
 \|F\|_{C^s(\mathcal O)}\le L.
\]
Moreover, under either law and conditional on $p$, the functions $F$
and $-F$ have the same distribution.

Let $\E_+$ and $\E_-$ denote expectation under $\Pi_+$ and $\Pi_-$.
For $1\le k\le M$, $x=(x_1,\ldots,x_k)\in Q^k$, and formal variables
$u=(u_1,\ldots,u_k)$, define
\[
 \mathcal M_{\pm,k}(x;u)
 =
 \E_\pm\!\left[
   \prod_{i=1}^k p(x_i)
   \exp\!\left\{\sum_{i=1}^k u_iF(x_i)\right\}
 \right]
\]
and
\[
 D_k(x;u)
 =
 \mathcal T_{k,R}\!\left[
 e^{-\delta\sum_{i=1}^k u_i^2/4}\mathcal M_{+,k}(x;u)
 -
 e^{ \delta\sum_{i=1}^k u_i^2/4}\mathcal M_{-,k}(x;u)
 \right].
\]
For one design point, all coefficients retained by
$\mathcal T_{1,R}$ agree exactly:
\[
 D_1\equiv0.
\]
For every $2\le k\le M$, the retained coefficient differences satisfy
\begin{equation}\label{eq:uniform-contract}
 \sum_{a\in\mathcal A_{k,R}}
 \bigl\|[u^a]D_k\bigr\|_{L^2(Q^k)}^2
 \le
 C^k\delta^2N^{-d}.
\end{equation}

The estimate in \eqref{eq:uniform-contract} remains valid after replacing
each $u_i$ by a bounded deterministic multiple depending on $x_i$.
More precisely, if $w_i:Q\to\mathbb R$ are measurable and
$|w_i|\le1$, define
\[
 D_k^{(w)}(x;u)
 =
 D_k\bigl(
 x;
 w_1(x_1)u_1,\ldots,w_k(x_k)u_k
 \bigr).
\]
Then
\[
 \sum_{a\in\mathcal A_{k,R}}
 \bigl\|[u^a]D_k^{(w)}\bigr\|_{L^2(Q^k)}^2
 \le
 C^k\delta^2N^{-d},
 \qquad 2\le k\le M.
\]

There is also a constant $\kappa_0>0$, depending only on $L$, the
variance bounds $v_-,v_+$, and the fixed constants $v,a_Y,C_4$ in the
ternary response construction, such that the following holds for
every $0<\kappa\le\kappa_0$. Let
$\operatorname{Rem}_{\pm,k}$ be the part of the density-weighted
$k$-response probability discarded by the degree-$2R$ truncation, as
defined in \eqref{eq:response-remainder-definition}. Then
\begin{equation}\label{eq:response-remainder-bound}
 \max_{\pm}
 \operatorname*{ess\,sup}_{x\in Q^k}
 \max_{y\in\{-a_Y,0,a_Y\}^k}
 \left|
   \operatorname{Rem}_{\pm,k}(x,y;\kappa)
 \right|
 \le
 \1_{\{k>R\}}C^k\kappa^{2R+2},
 \qquad 1\le k\le M.
\end{equation}
In particular, the remainder is identically zero when $k\le R$.

Let
\[
 \rho_k(x)
 =
 \E_\pm\prod_{i=1}^k p(x_i)
\]
denote the common density of the $k$ design points under the two priors.
After dividing the density-weighted response probabilities by
$\rho_k(x)$ to obtain conditional response probabilities, the bound
\eqref{eq:response-remainder-bound} remains valid after increasing $C$.
Consequently, both before and after this division,
\[
 \sum_{y\in\{-a_Y,0,a_Y\}^k}
 \bigl|\operatorname{Rem}_{\pm,k}(x,y;\kappa)\bigr|^2
 \le
 \1_{\{k>R\}}C^k\kappa^{4R+4}
\]
for almost every $x\in Q^k$, after increasing $C$ once more if necessary.
\end{lemma}

Apply \zcref{lem:gluing} to the priors supplied by
\zcref{lem:uniform-local}, with
\[
 K=L,\qquad
 C_{\rm loc}=\max\{1,C\},\qquad
 \xi=N^{-d},\qquad
 \delta=cN^{-2}.
\]
Then \eqref{eq:uniform-contract}, including its version after the
substitutions $u_i\mapsto w_i(x_i)u_i$, gives exactly the coefficient
bounds required by \zcref{lem:gluing}. With
\[
 \kappa=c_fh^s,
\]
the resulting two global priors have deterministic variance parameters
separated by
\[
 \Delta=\kappa^2\delta=\kappa^2cN^{-2}.
\] 
The Hellinger bound \eqref{eq:global-budget} becomes
\begin{equation}\label{eq:refined-global}
 \begin{aligned}
 H^2(\overline P_+^{\rm Pois},\overline P_-^{\rm Pois})
 \le C\bigl\{
 &\nu^2h^d\kappa^4c^2N^{-d-4}\\
 &+\nu\kappa^{4R+4}(C\nu h^d)^R
 +\nu\kappa^4(C\nu h^d)^M
 +\nu h^d\eta^4
 \bigr\}.
 \end{aligned}
\end{equation}
The four terms respectively control the retained coefficient mismatch, the
response terms of degree greater than $2R$, configurations containing more
than $M$ observations in one connected local component, and the error caused
by normalizing the perturbed design density. The restriction
\[
 c\le aD^{-M}
\]
from \zcref{lem:uniform-local} is the only requirement that forces the
amplitude $c$ to decrease as $M$ grows.

\begin{proof}[Proof of the lower bound in \zcref{thm:main}]
Recall that $s>1$ and $d>4s$, and write
\[
 \lambda=\frac{2(s+1)}{d+4}.
\]
We now choose the parameters so that
all four terms in \eqref{eq:refined-global} remain small while the variance
separation $\Delta$ is as large as possible.

First choose an integer $R$ such that
\[
 4s(R+1)>d.
\]
This condition will make the response-truncation term in
\eqref{eq:refined-global} tend to zero. Next choose
\[
 0<\eta\le\min\{\eta_0,\eta_1\},
\]
where $\eta_0$ is supplied by \zcref{lem:gluing} and $\eta_1$ by
\zcref{lem:uniform-local}. Apply \zcref{lem:uniform-local} with this
value of $\eta$, and let $C$ and $K_{\rm loc}$ denote its coefficient and
smoothness constants. We may then apply \zcref{lem:gluing} with
\[
 K=K_{\rm loc},\qquad C_{\rm loc}=\max\{1,C\},
\]
which fixes $c_f$ and the constants in the global comparison. From this
point onward, $R,\eta$, and all constants supplied by the two lemmas are
fixed.

For clarity, let $C_{\rm adm}$ denote a constant for which the smallness
condition in \zcref{lem:gluing} is satisfied whenever
\[
 C_{\rm adm}\nu h^d\le\frac12.
\]
Let $C_{\rm occ}$ dominate the constants multiplying $\nu h^d$ inside the
second and third terms of \eqref{eq:refined-global}, and let $C_{\rm bud}$
denote its overall multiplicative constant. These constants are fixed and
do not depend on $n,M,h$, or $N$. Choose
\[
 0<c_0\le\min\{a/2,1\},
\]
and choose
\[
 0<t_0\le
 \min\left\{
  1,\frac{1}{2C_{\rm adm}},\frac{1}{2C_{\rm occ}}
 \right\}.
\]
Finally choose a fixed constant
\[
 B\ge128C_{\rm bud}.
\]

Set
\[
 \nu=2n,
 \qquad
 L=\log\nu,
\]
and define
\begin{equation}\label{eq:scales}
 M=\lceil\sqrt L\rceil,
 \qquad
 c=c_0D^{-M},
 \qquad
 t=t_0e^{-2\sqrt L},
 \qquad
 h=\left\lceil(\nu/t)^{1/d}\right\rceil^{-1},
 \qquad
 \tau=\nu h^d,
 \qquad
 \kappa=c_fh^s.
\end{equation}
Here $M$ is the largest local observation count that we match, while $t$
is the target expected number of observations in a region of volume $h^d$.
The rounding in the definition of $h$ gives
\[
 2^{-d}t\le\tau\le t
\]
for all sufficiently large $\nu$. Hence
\[
 \tau\asymp e^{-2\sqrt L},
 \qquad
 C_{\rm adm}\tau\le\frac12,
\]
and, by the choice of $t_0$,
\[
 C_{\rm occ}\tau
 \le
 \frac12e^{-2\sqrt L}.
\]
Thus the smallness condition of \zcref{lem:gluing} holds.

We choose $N$ to make the leading term in
\eqref{eq:refined-global} a fixed small constant. Define
\[
 N_*
 =
 \bigl(B\nu^2h^d\kappa^4c^2\bigr)^{1/(d+4)}
\]
and set
\[
 N=\left\lceil\max\{2,N_*\}\right\rceil.
\]
To verify that $N_*$ tends to infinity, note that
\[
 h^d=\frac{\tau}{\nu}
\]
and $\log\tau=O(\sqrt L)$. Since $\kappa=c_fh^s$ and
$c=c_0D^{-M}$ with $M=O(\sqrt L)$,
\[
 \log N_*
 =
 \frac{1-4s/d}{d+4}L+O(\sqrt L).
\]
Because $d>4s$, the coefficient of $L$ is positive, and hence
$N_*\to\infty$. In particular, for all sufficiently large $n$,
\[
 N_*\le N\le2N_*.
\]

We now bound the four terms in \eqref{eq:refined-global}. By the definition
of $N_*$ and the fact that $N\ge N_*$,
\[
 \nu^2h^d\kappa^4c^2N^{-d-4}
 \le B^{-1}.
\]
For the response-truncation term, $\kappa=c_fh^s$ and
$h^d=\tau/\nu$ give
\[
 \nu\kappa^{4R+4}(C_{\rm occ}\tau)^R
 \le
 C\nu^{1-4s(R+1)/d}e^{-2R\sqrt L}.
\]
The exponent $1-4s(R+1)/d$ is negative by our choice of $R$, so this
quantity tends to zero.

For configurations containing more than $M$ observations,
\[
 \nu\kappa^4(C_{\rm occ}\tau)^M
 \le
 C\nu^{1-4s/d}e^{-2M\sqrt L}.
\]
Since $M=\lceil\sqrt L\rceil$, we have
$M\sqrt L\ge L$, and therefore this term also tends to zero. Finally,
\[
 \nu h^d\eta^4
 =
 \tau\eta^4
 \le
 t_0\eta^4e^{-2\sqrt L}
 \longrightarrow0.
\]

Multiplying the first estimate by the prefactor $C_{\rm bud}$ gives at most
\[
 \frac{C_{\rm bud}}{B}\le\frac1{128}.
\]
For all sufficiently large $n$, the other three contributions together are
also at most $1/128$. Hence
\[
 H^2(\overline P_+^{\rm Pois},\overline P_-^{\rm Pois})
 \le\frac1{64}.
\]
The transfer inequality \eqref{eq:poisson-transfer} and
\[
 \Pp\{\operatorname{Pois}(2n)\ge n\}\longrightarrow1
\]
then imply, for all sufficiently large $n$,
\[
 H^2(\overline P_{+,n},\overline P_{-,n})
 \le\frac1{32}.
\]
The standard comparison between total variation and Hellinger distance
therefore gives
\[
 \TV(\overline P_{+,n},\overline P_{-,n})<\frac12.
\]

It remains to lower-bound the separation of the two variance parameters.
Since $N\le2N_*$ eventually,
\[
 \Delta
 =
 \kappa^2cN^{-2}
 \ge
 C\kappa^2c
 \bigl(\nu^2h^d\kappa^4c^2\bigr)^{-2/(d+4)}.
\]
Using $\kappa=c_fh^s$ and then $h^d=\tau/\nu$ gives
\begin{align*}
 \Delta
 &\ge
 C_*
 \nu^{-4/(d+4)}
 h^{2d(s-1)/(d+4)}
 c^{d/(d+4)}\\
 &=
 C_*
 \nu^{-\lambda}
 \tau^{2(s-1)/(d+4)}
 c^{d/(d+4)}.
\end{align*}
Now
\[
 \tau\ge2^{-d}t_0e^{-2\sqrt L},
 \qquad
 c=c_0D^{-M},
 \qquad
 M\le\sqrt L+1.
\]
Consequently,
\begin{align}
 \Delta
 &\ge
 C_{**}\nu^{-\lambda}
 \exp\left\{
 -\frac{4(s-1)+d\log D}{d+4}\sqrt{\log\nu}
 \right\}.
 \label{eq:separation}
\end{align}
The constants lost through $t_0,c_0$, and the additive $1$ in
$M\le\sqrt L+1$ have been absorbed into $C_{**}$.

Since the two fixed-$n$ mixture laws have total variation distance below
$1/2$, \eqref{eq:testing} yields
\[
 \mathfrak r_n\ge\frac{\Delta}{4}.
\]
Finally, $\nu=2n$ implies
$\nu^{-\lambda}\asymp n^{-\lambda}$ and
$\sqrt{\log\nu}=\sqrt{\log n}+O((\log n)^{-1/2})$. Absorbing these fixed
changes into the constants gives the lower bound in
\zcref{thm:main}.
\end{proof}

The stretched-exponential loss comes from balancing two competing
requirements. Matching all local configurations up to size $M$ forces the
amplitude $c$ to contain the factor $D^{-M}$, producing a loss exponential
in $M$. On the other hand, configurations with more than $M$ observations
must be made rare by taking the expected local occupancy $\tau$ small; to
make their contribution negligible, $\tau$ must be roughly exponential in
$-(\log n)/M$. These two losses are balanced when
\[
 M\asymp\frac{\log n}{M},
\]
or equivalently
\[
 M\asymp\sqrt{\log n},
\]
which produces the factor $e^{-C\sqrt{\log n}}$.

\section{Constructing priors with prescribed moments}\label{sec:local-proof}
This section proves \zcref{lem:local,lem:uniform-local}. We first
introduce the bounded random coefficients used to perturb the regression
function and specify how signed moment corrections will be represented.
We then prove two lemmas that turn such prescribed corrections into
probability laws with the required boundedness and positivity properties.

In \zcref{sec:local-construction}, we determine how the density-weighted
moments under the two priors must differ in order to compensate for the
change in the variance parameter. We then state
\zcref{prop:local-construction}, which constructs priors realizing these
moment requirements for finitely many design points. Its proof first makes
the moment requirements compatible when design points are integrated out,
then realizes the resulting signed corrections by genuine probability laws,
and finally bounds the remaining coefficient errors. In
\zcref{sec:local-applications}, we choose the interpolation constructions
needed for \zcref{lem:local,lem:uniform-local} and verify the corresponding
bounds on their constants.

\subsection{Representation conventions and preliminary lemmas}
\label{sec:local-preliminaries}

Fix the response-degree cutoff $R\ge1$ and let $Q=[0,1]^d$. We will use
two kinds of spatial variables for different purposes. Variables
\[
 \mathbf U=(U_1,\ldots,U_m)\in Q^m
\]
represent design locations. They carry factors of the design density and
will sometimes be integrated out when we impose consistency between the
moment requirements for different numbers of design points. By contrast,
\[
 T=(T_1,\ldots,T_r)\in Q^r
\]
will be used as placeholder response locations when we convert coefficient
polynomials into multilinear functions of spatial arguments. The variables
$T_i$ remain fixed during integrations in the $U_i$ variables and are set
equal to design locations only after the required consistency relations
have been established. For a set of labels $A\subset[m]$, write
\[
 U_A=(U_i)_{i\in A},
\]
with the coordinates listed in increasing order of their labels.

We first choose the smooth functions from which the random regression
functions will be built. Let $\mathcal I$ be a finite or countable index
set containing $0$, and let
\[
 \phi=(\phi_a)_{a\in\mathcal I}
\]
be a family of real-valued smooth functions. Fix an open neighborhood
$\mathcal O\supset Q$ on which all $\phi_a$ are smooth, and assume
\begin{equation}\label{eq:frame-assumption}
 \phi_0|_Q=1,\qquad
 L_b(\phi):=
 \sum_{a\in\mathcal I}\sum_{|\gamma|\le b}
 \|\partial^\gamma\phi_a\|_{\infty,\mathcal O}<\infty
 \quad\text{for every integer }b\ge0.
\end{equation}
Thus the series formed from uniformly bounded coefficients and the
functions $\phi_a$ may be differentiated term by term to any fixed order.
We will also use
\[
 \sum_{a\in\mathcal I}\|\phi_a\|_{C^s(\mathcal O)}<\infty.
\]
For the families used below, this follows from the corresponding derivative
bounds for compactly supported smooth extensions to $\mathbb R^d$. In
particular, the random functions constructed from this family are smooth,
although the statistical model uses only their $C^s$ bound. We will enlarge
the family $\phi$ once in \zcref{sec:local-construction}.

We next choose bounded random coefficients whose low-order moments agree
with those of independent standard Gaussians. Let
\[
 z=(z_a)_{a\in\mathcal I}
\]
denote formal coefficient variables; these are distinct from the response
variables $u_1,\ldots,u_m$. At this stage we retain coefficient monomials
of total degree at most $2R$, without imposing any restriction such as
$z_a^3=0$.

Choose independent random variables $(B_a)_{a\in\mathcal I}$ with a common
bounded symmetric distribution whose moments agree with those of a standard
Gaussian through degree $4R$. Such a distribution can be obtained from
Gaussian quadrature: take the $2R+1$ roots of the degree-$(2R+1)$ polynomial
orthogonal with respect to the standard Gaussian measure, with the
corresponding quadrature weights. The resulting rule is exact for
polynomials of degree at most $4R+1$. Its weights are positive: if $L_j$
is the Lagrange cardinal polynomial for one quadrature node, exactness
applied to $L_j^2$ gives
\[
 w_j=\int L_j(x)^2\,d\gamma(x)>0,
\]
where $\gamma$ is standard Gaussian measure. Symmetry of the Gaussian
measure gives a symmetric quadrature rule.

Let $b_R<\infty$ bound the support of this coefficient distribution. Then
\[
 F(x)=\sum_{a\in\mathcal I}B_a\phi_a(x)
\]
converges absolutely together with every fixed derivative, by
\eqref{eq:frame-assumption}, and
\[
 \|F\|_{C^s(\mathcal O)}
 \le
 b_R\sum_{a\in\mathcal I}\|\phi_a\|_{C^s(\mathcal O)}.
\]
Thus the bounded coefficient law automatically produces functions with a
uniform $C^s$ bound.

We now record the finite-order Gaussian identities that these bounded
coefficients inherit. We use the probabilists' Hermite polynomials,
normalized by
\[
 e^{tx-t^2/2}
 =
 \sum_{n\ge0}\operatorname{He}_n(x)\frac{t^n}{n!}.
\]
For a finitely supported multi-index $I=(I_a)_{a\in\mathcal I}$, set
\[
 H_I(B)=\prod_{a\in\mathcal I}\operatorname{He}_{I_a}(B_a).
\]
For example, $\operatorname{He}_2(x)=x^2-1$. If $|I|\le2R$, then
moment matching with the Gaussian law gives
\[
 \left[
   \E H_I(B)e^{\langle z,B\rangle}
 \right]_{\deg_z\le2R}
 =
 \left[
   z^I e^{\|z\|_2^2/2}
 \right]_{\deg_z\le2R}.
\]
Indeed, the coefficient of a monomial of degree at most $2R$ on the
left involves moments of $B$ of degree at most
$|I|+2R\le4R$, exactly the range in which the bounded law agrees with
the Gaussian law.

Let $\mathsf R$ denote the joint law of the coefficients $(B_a)$.
For a bounded function $h(B)$, define
\begin{equation}\label{eq:hermite-generating-map}
 \mathcal G(h)(z)
 =
 \left[
 e^{-\|z\|_2^2/2}
 \E_{\mathsf R}\!\left[h(B)e^{\langle z,B\rangle}\right]
 \right]_{\deg_z\le2R}.
\end{equation}
All coefficient calculations in \eqref{eq:hermite-generating-map} are first
performed with only finitely many nonzero $z_a$. When we later substitute
linear combinations of the functions $\phi_a$, the summability in
\eqref{eq:frame-assumption} extends the identities by absolute convergence.

The purpose of the factor $e^{-\|z\|_2^2/2}$ in
\eqref{eq:hermite-generating-map} is to remove the common Gaussian
moment-generating factor. To see what this factor becomes after evaluation
at design points, substitute
\[
 z_a=\sum_{i=1}^m u_i\phi_a(U_i).
\]
The corresponding truncated Gaussian factor is
\begin{equation}\label{eq:gaussian-factor}
 M_{0,m}(\mathbf U;u)
 =
 \mathcal T_{m,R}
 \exp\left\{
   \frac12\sum_{a\in\mathcal I}
   \left(\sum_{i=1}^m u_i\phi_a(U_i)\right)^2
 \right\}.
\end{equation}
Although the actual coefficients $B_a$ are bounded rather than Gaussian,
their moments agree with Gaussian moments to the order used here. It is
this finite-degree identity that allows the factor
\eqref{eq:gaussian-factor} to be removed in the later coefficient
calculations.

We next convert homogeneous polynomials in the coefficient variables into
multilinear functions of spatial locations. Let $P(z)$ be homogeneous of
degree $r$. Define
\begin{equation}\label{eq:external-tensor}
 \Phi_P(T_1,\ldots,T_r)
 =
 \frac1{r!}[x_1\cdots x_r]\,
 P\left(
   \left(\sum_{i=1}^r x_i\phi_a(T_i)\right)_{a\in\mathcal I}
 \right).
\end{equation}
The function $\Phi_P$ is symmetric and multilinear in the $r$ vectors
$\phi(T_i)$. On the diagonal,
\[
 \Phi_P(T,\ldots,T)=P(\phi(T)).
\]
More generally, if
\[
 P\left(
   \left(\sum_i u_i\phi_a(U_i)\right)_{a\in\mathcal I}
 \right)
\]
is differentiated $r$ times with respect to response variables, the
result is $r!$ times $\Phi_P$ evaluated at the corresponding list of
locations, with repetitions when the same response variable is
differentiated more than once. For example, if
\[
 P(z)=\frac{\delta}{2}z_0^2,
\]
then $\phi_0=1$ on $Q$ gives
\[
 \Phi_P(T_1,T_2)=\frac{\delta}{2}.
\]
Accordingly, the coefficient of $u_1^2$ is $\delta/2$, while the second
derivative with respect to $u_1$ is $\delta$. This is the normalization
used later to compensate the change in the variance parameter for a single
observation.

We will realize such multilinear functions by perturbing the bounded
coefficient law with normalized Hermite polynomials. Choose
$h_R\ge1$ large enough that
\[
 \|H_I\|_\infty\le h_R
\]
under the bounded coefficient law for every finitely supported multi-index
$I$ of positive even degree $|I|\le2R$. This bound does not depend on the
size of $\mathcal I$, because such an $I$ involves at most $2R$
coordinates. Define
\[
 h_I=\frac{H_I}{h_R}.
\]
Then
\[
 \E_{\mathsf R}h_I(B)=0,
 \qquad
 \|h_I\|_\infty\le1,
\]
and $h_I(-B)=h_I(B)$ because $|I|$ is even. Moreover,
\eqref{eq:hermite-generating-map} and the Hermite identity above give
\[
 \mathcal G(h_I)(z)=\frac{z^I}{h_R}.
\]
We call
\[
 \Phi_{\mathcal G(h_I)}
\]
the \emph{normalized response atom} associated with $I$. Thus each such
atom is produced by a bounded, mean-zero, sign-symmetric perturbation of
the coefficient law, and its generating polynomial is homogeneous of
degree $|I|$. When we use the unnormalized polynomial $H_I$ instead, the
fixed factor $h_R$ is included in the representation bounds below.

Finally, we specify how interpolation kernels will be represented and how
we measure the size of such a representation. Suppose
$E(\mathbf U)$ is a matrix indexed by
$\mathcal I\times\mathcal I$. A \emph{separated coefficient-kernel
representation} of $E$ consists of a sigma-finite measure space
$(\mathcal Z,\sigma)$ on a standard Borel parameter space, measurable
functions
\[
 a_i(\,\cdot\,;\zeta):Q\to\mathbb R,
 \qquad |a_i(U_i;\zeta)|\le1,
\]
and real matrices
\[
 A(\zeta)=(A_{ab}(\zeta))_{a,b\in\mathcal I},
\]
such that
\begin{equation}\label{eq:coefficient-kernel-cost}
 E(\mathbf U)
 =
 \int_{\mathcal Z}
 \left(\prod_{i=1}^m a_i(U_i;\zeta)\right)
 A(\zeta)\,d\sigma(\zeta)
\end{equation}
and
\[
 \operatorname{cost}(E)
 :=
 \int_{\mathcal Z}
 \sum_{a,b\in\mathcal I}|A_{ab}(\zeta)|\,d\sigma(\zeta)
 <\infty.
\]
Thus each design location appears through its own bounded factor, while
the matrix $A(\zeta)$ records the coefficients of the functions
$\phi_a(t)\phi_b(t')$. Any signs, as well as the density of a finite signed
measure on the parameter space, are absorbed into $A(\zeta)$. The factors
$a_i$ are not required to have mean zero. Absolute summability of the
matrix coefficients ensures that both the integral above and its evaluation
against the functions $\phi_a$ are well defined. The quantity
$\operatorname{cost}(E)$ is the representation cost used in the
interpolation constructions below.

We now introduce a second representation, adapted to the later construction
of positive priors. Fix one normalized design density $p_\omega$, indexed by
$\omega$, and fix an even response degree $r\le2R$. Suppose a kernel
$K(\mathbf U;T)$ depends on $q$ design locations
$\mathbf U=(U_1,\ldots,U_q)$ and on the fixed response locations
$T=(T_1,\ldots,T_r)$. We represent it in the form
\begin{equation}\label{eq:general-response-representation}
 \begin{aligned}
 K(\mathbf U;T)
 &=
 \int_{\mathcal Z}
 \lambda(\omega,\zeta)
 \prod_{i=1}^q a_{\omega,\zeta,i}(U_i)
 \Phi_{\omega,\zeta}(T)\,d\sigma(\zeta),\\
 \operatorname{cost}(K)
 &=
 \int_{\mathcal Z}
 |\lambda(\omega,\zeta)|\,d\sigma(\zeta)<\infty.
 \end{aligned}
\end{equation}
Here each $a_{\omega,\zeta,i}:Q\to\mathbb R$ is measurable and satisfies
\[
 \|a_{\omega,\zeta,i}\|_\infty\le1,
\]
and each $\Phi_{\omega,\zeta}$ is a normalized response atom of degree $r$
in the sense defined above. Thus the dependence on the design points is
carried entirely by the product of the factors
$a_{\omega,\zeta,i}(U_i)$, while the dependence on the response locations
is carried entirely by $\Phi_{\omega,\zeta}(T)$.

At this stage we do not require the design factors to have mean zero, and
we do not require the representation to be invariant under permutations of
the design labels. We also regard the particular decomposition in
\eqref{eq:general-response-representation} as part of the data: the
coefficient $\lambda$, the design factors, and the normalized score used to
produce $\Phi_{\omega,\zeta}$ will later be modified separately when we
construct positive probability laws. Thus two different representations of
the same function $K$ need not be interchangeable for that construction.
If several even response degrees are present, we keep their representations
separate.

For the synthesis lemma below, we will use a more restrictive form in which
the design factors have mean zero and the dependence on the $q$ design
locations is symmetric. Write
\begin{equation}\label{eq:local-separated}
 H_q(\mathbf U;T)
 =
 \int_{\mathcal Z}
 \lambda(\omega,\zeta)
 \Sym\left(
  \prod_{i=1}^q a_{\omega,\zeta,i}(U_i)
\right)
 \Phi_{\omega,\zeta}(T)\,d\sigma(\zeta),
\end{equation}
where
\[
 \int_Q a_{\omega,\zeta,i}(u)\,du=0
\]
for every $i$, and where
\[
 \Sym\left(\prod_{i=1}^q a_i(U_i)\right)
 =
 \frac1{q!}\sum_{\pi\in S_q}
 \prod_{i=1}^q a_i(U_{\pi(i)}).
\]
The cost of this representation is again
\[
 \int_{\mathcal Z}|\lambda(\omega,\zeta)|\,d\sigma(\zeta).
\]

We next explain how to obtain \eqref{eq:local-separated} from
\eqref{eq:general-response-representation}. Suppose first that $K$ has
zero integral in each design variable separately. For a function of
$\mathbf U$, let
\[
 C_iK(\mathbf U;T)
 =
 K(\mathbf U;T)
 -
 \int_Q K(U_1,\ldots,U_{i-1},u,U_{i+1},\ldots,U_q;T)\,du
\]
denote centering in the $i$th design variable. By assumption,
$C_iK=K$ for every $i$, and hence
\[
 C_1\cdots C_qK=K.
\]
Applying these centering operators to the representation
\eqref{eq:general-response-representation} simply replaces each factor
$a_i$ by
\[
 a_i^\circ
 =
 a_i-\int_Q a_i(u)\,du.
\]
Thus the centered factors give another representation of the same kernel
$K$.  Since
$\|a_i^\circ\|_\infty\le2$, write
\[
 a_i^\circ=2\widetilde a_i,
 \qquad
 \|\widetilde a_i\|_\infty\le1,
 \qquad
 \int_Q\widetilde a_i=0.
\]
Absorbing the resulting factor $2^q$ into $\lambda$ increases the
representation cost by at most $2^q$. If $K$ is also symmetric in its
$q$ design arguments, averaging the resulting representation over all
permutations of the design labels leaves $K$ unchanged and does not
increase the cost. We will use the unrestricted representation
\eqref{eq:general-response-representation} while constructing the required
kernels and the centered symmetric form \eqref{eq:local-separated} when
applying the synthesis lemma.

The next lemma provides the scalar coefficients needed to isolate terms
containing exactly $q$ density perturbations. Indeed, if a density factor is
written as $1+t\,a(U)$, then the expansion of a product of such factors
contains powers $t^l$, where $l$ records how many perturbation factors have
been selected. The signed measure below annihilates every power
$t^l$ with $0\le l\le M$ except $t^q$.

\begin{lemma}\label{lem:local-scalar}
Let $M\ge2$ and $0<\eta<1$. For every $q\in\{2,\ldots,M\}$, there exists
a finite signed measure $\nu_{q,M}$ on $[-\eta,\eta]$ such that
\[
 \int_{-\eta}^{\eta} t^l\,d\nu_{q,M}(t)
 =
 \mathbf1_{\{l=q\}},
 \qquad
 0\le l\le M.
\]
Moreover,
\[
 \|\nu_{q,M}\|_{\TV}
 \le I_M,
 \qquad
 I_M\le C\left(\frac{3}{\eta}\right)^M,
\]
where $C$ is an absolute constant independent of $M,q$, and $\eta$.
\end{lemma}
\begin{proof}
Let $J=M+1$ and define
\[
 \theta_j=\frac{(j+1/2)\pi}{J},
 \qquad
 t_j=\eta\cos\theta_j,
 \qquad
 0\le j\le J-1.
\]
These are the $J$ Chebyshev nodes on $[-\eta,\eta]$. We use them to
construct a signed measure whose action on every polynomial of degree at
most $M$ extracts the coefficient of $x^q$.

Let $T_r$ denote the Chebyshev polynomial of degree $r$. The discrete
orthogonality relations
\[
 \frac1J\sum_{j=0}^{J-1}\cos(r\theta_j)=0,
 \qquad 1\le r\le M,
\]
and
\[
 \frac2J\sum_{j=0}^{J-1}
 \cos(r\theta_j)\cos(l\theta_j)
 =
 \mathbf1_{\{r=l\}},
 \qquad 1\le r,l\le M,
\]
give, for every polynomial $P$ of degree at most $M$,
\[
 P(x)
 =
 \frac1J\sum_{j=0}^{J-1}P(t_j)
 +
 \frac2J\sum_{r=1}^M
 T_r(x/\eta)
 \sum_{j=0}^{J-1}P(t_j)\cos(r\theta_j).
\]
Indeed, this is the Chebyshev expansion of $P$ recovered from its values
at the $J$ nodes.

For $2\le q\le M$, define
\[
 \nu_{q,M}
 =
 \frac1J\sum_{j=0}^{J-1}
 \left(
  2\sum_{r=1}^M
  [x^q]T_r(x/\eta)\cos(r\theta_j)
 \right)\delta_{t_j}.
\]
Applying the coefficient operator $[x^q]$ to the interpolation formula
above gives, for every polynomial $P$ of degree at most $M$,
\[
 \int P(t)\,d\nu_{q,M}(t)
 =
 [x^q]P(x).
\]
Taking $P(x)=x^l$, $0\le l\le M$, therefore yields
\[
 \int t^l\,d\nu_{q,M}(t)
 =
 \mathbf1_{\{l=q\}},
\]
which proves the required moment identities.

It remains to bound the total variation. Let $\|P\|_{\ell^1}$ denote the
sum of the absolute values of the coefficients of a polynomial $P$.
The recurrence
\[
 T_{r+1}(x)=2xT_r(x)-T_{r-1}(x),
 \qquad
 T_0(x)=1,\quad T_1(x)=x,
\]
implies inductively that
\[
 \|T_r\|_{\ell^1}\le3^r.
\]
Hence
\[
 \bigl|[x^q]T_r(x/\eta)\bigr|
 =
 \eta^{-q}\bigl|[x^q]T_r(x)\bigr|
 \le
 \eta^{-q}3^r.
\]
Using $|\cos(r\theta_j)|\le1$ in the definition of $\nu_{q,M}$ gives
\begin{align*}
 \|\nu_{q,M}\|_{\TV}
 &\le
 2\sum_{r=1}^M
 \bigl|[x^q]T_r(x/\eta)\bigr|\\
 &\le
 2\eta^{-q}\sum_{r=1}^M3^r
 \le
 C\left(\frac3\eta\right)^M,
\end{align*}
where we used $q\le M$ and $0<\eta<1$ in the last inequality. Thus the
same bound
\[
 I_M=C(3/\eta)^M
\]
works for every $q=2,\ldots,M$.
\end{proof}

The next lemma converts the centered symmetric representations above into
signed measures on actual probability densities. There are two steps. First,
a symmetric product of possibly different centered design factors can be
written as a signed average of products in which the same centered factor is
used at every design location. Second, \zcref{lem:local-scalar} chooses the
scalar perturbation size so that, simultaneously for every observation count
$m\le M$, only products involving exactly $q$ density perturbations remain.

\begin{lemma}\label{lem:local-synthesis}
Fix an even response degree $r$ with $2\le r\le2R$. For
$2\le q\le M$, let
\[
 H_q(U_1,\ldots,U_q;T)
\]
be symmetric in the design variables $U_1,\ldots,U_q$, where
$T=(T_1,\ldots,T_r)$ denotes the fixed response locations. Assume that
$H_q$ has zero integral in each design variable separately:
\[
 \int_Q
 H_q(U_1,\ldots,U_{i-1},u,U_{i+1},\ldots,U_q;T)\,du
 =0,
 \qquad 1\le i\le q.
\]
Suppose further that each $H_q$ has a specified representation of the form
\eqref{eq:local-separated}, with representation cost $A_q$.

Then there exists a finite signed measure $\Lambda$ on pairs $(p,\Phi)$
with the following properties. For every pair in its support,
$\Phi$ is a normalized response atom of degree $r$ and
\[
 p(u)=1+t\,a(u),
 \qquad
 |t|\le\eta,
 \qquad
 \int_Q a(u)\,du=0,
 \qquad
 \|a\|_\infty\le1.
\]
In particular, since $0<\eta<1$,
\[
 \int_Q p(u)\,du=1,
 \qquad
 p(u)\ge1-\eta>0,
\]
so every $p$ is a probability density on $Q$.

For every $0\le m\le M$, the same signed measure $\Lambda$ satisfies
\begin{equation}\label{eq:count-kernel}
 \int
 \left(\prod_{i=1}^m p(U_i)\right)\Phi(T)
 \,d\Lambda(p,\Phi)
 =
 \sum_{\substack{A\subseteq[m]\\ |A|\ge2}}
 H_{|A|}(U_A;T).
\end{equation}
Here the product on the left is equal to one when $m=0$, and the sum on
the right is empty when $m=0$ or $m=1$. Moreover,
\[
 \|\Lambda\|_{\TV}
 \le
 I_M\sum_{q=2}^M e^qA_q.
\]
\end{lemma}

\begin{proof}
We first reduce each centered symmetric design factor in
\eqref{eq:local-separated} to products of a single centered function.
Suppress $\omega,\zeta$ from the notation and consider
\[
 \Sym\left(\prod_{i=1}^q a_i(U_i)\right),
\]
where
\[
 \|a_i\|_\infty\le1,
 \qquad
 \int_Qa_i(u)\,du=0.
\]
Let $\varepsilon_1,\ldots,\varepsilon_q$ be independent random signs,
each taking the values $\pm1$ with probability $1/2$, and define
\[
 a_\varepsilon(u)
 =
 \frac1q\sum_{i=1}^q\varepsilon_i a_i(u).
\]
Then
\[
 \|a_\varepsilon\|_\infty\le1,
 \qquad
 \int_Qa_\varepsilon(u)\,du=0.
\]
Moreover,
\begin{equation}\label{eq:polarization-product}
 \Sym\left(\prod_{i=1}^q a_i(U_i)\right)
 =
 \frac{q^q}{q!}
 \E_\varepsilon\left[
   \left(\prod_{i=1}^q\varepsilon_i\right)
   \prod_{j=1}^q a_\varepsilon(U_j)
 \right].
\end{equation}
To verify this identity, expand the product on the right. A term survives
the expectation over the signs only if every $\varepsilon_i$ occurs an even
number of times after multiplication by $\prod_i\varepsilon_i$. Since
there are exactly $q$ factors $a_\varepsilon(U_j)$, this happens precisely
when each index $i$ is chosen once. The surviving terms are therefore the
$q!$ permutations appearing in the symmetrization. Since
\[
 \frac{q^q}{q!}\le e^q,
\]
this reduction increases the representation cost by at most a factor $e^q$.

It remains to convert a product
\[
 \prod_{i=1}^q a(U_i)
\]
of one centered bounded function into density perturbations that work
simultaneously for all counts $m\le M$. For fixed $q$, let
$\nu_{q,M}$ be the signed measure supplied by
\zcref{lem:local-scalar}. Expanding
\[
 \prod_{i=1}^m(1+t\,a(U_i))
 =
 \sum_{A\subseteq[m]}
 t^{|A|}\prod_{i\in A}a(U_i)
\]
and integrating with respect to $\nu_{q,M}$ gives
\[
 \int
 \prod_{i=1}^m(1+t\,a(U_i))
 \,d\nu_{q,M}(t)
 =
 \sum_{\substack{A\subseteq[m]\\ |A|=q}}
 \prod_{i\in A}a(U_i),
 \qquad 0\le m\le M,
\]
because
\[
 \int t^l\,d\nu_{q,M}(t)=\mathbf1_{\{l=q\}}
 \qquad (0\le l\le M).
\]
Each resulting factor
\[
 p(u)=1+t\,a(u)
\]
is a probability density because $\int_Qa=0$, $|t|\le\eta<1$, and
$\|a\|_\infty\le1$.

Apply this construction to every term in the specified representation of
each $H_q$, keeping its normalized response atom $\Phi$ unchanged, and then
sum over $q=2,\ldots,M$. For a fixed count $m$, the contribution from the
$q$th representation is
\[
 \sum_{\substack{A\subseteq[m]\\ |A|=q}}
 H_q(U_A;T).
\]
Summing over $q$ gives exactly \eqref{eq:count-kernel}, and the same signed
measure works for every $m\le M$.

Finally, \eqref{eq:polarization-product} costs at most a factor $e^q$, while
\zcref{lem:local-scalar} gives
$\|\nu_{q,M}\|_{\TV}\le I_M$. Therefore
\[
 \|\Lambda\|_{\TV}
 \le
 I_M\sum_{q=2}^M e^qA_q,
\]
as claimed.
\end{proof}

\subsection{Construction of the local priors}\label{sec:local-construction}

We now specify the moment differences that the two local priors must realize.
For each $k\ge2$, let $E_{k,N}(\mathbf U)$ be one of the interpolation
matrices constructed in the smooth family $\phi$. Thus, for
$\mathbf U=(U_1,\ldots,U_k)\in Q^k$,
\[
 \phi(U_i)^T E_{k,N}(\mathbf U)\phi(U_\ell)
 =
 w_{k,i}(\mathbf U)\mathbf1_{\{i=\ell\}},
 \qquad
 0\le w_{k,i}(\mathbf U)\le1.
\]
Assume that $E_{k,N}$ has a separated coefficient-kernel representation
of cost at most $A_k$ in the sense of
\eqref{eq:coefficient-kernel-cost}, and that
\[
 \int_{Q^k}
 \sum_{i=1}^k\bigl(1-w_{k,i}(\mathbf U)\bigr)\,d\mathbf U
 \le B_k.
\]
The numbers $A_k$ and $B_k$ therefore measure, respectively, the size of
the representation and the integrated error in replacing the desired
diagonal value $1$ by $w_{k,i}$.

We allow the construction to start from a finite positive measure $\mu$
on normalized design densities having bounded representatives with a
common bound $\|p\|_\infty\le P<\infty$. In the eventual construction
$|p-1|\le\eta$, so $P=1+\eta$ suffices. Define
\[
 \rho_k(U_1,\ldots,U_k)
 =
 \int \prod_{i=1}^k p(U_i)\,d\mu(p),
 \qquad
 \rho_0=\mu(\mathcal P),
\]
where $\mathcal P$ denotes the set of densities on which $\mu$ is
supported. Since every $p$ integrates to one,
\[
 \int_Q
 \rho_k(U_1,\ldots,U_k)\,dU_i
 =
 \rho_{k-1}(U_1,\ldots,U_{i-1},U_{i+1},\ldots,U_k).
\]
Also,
\[
 \rho_0=\|\mu\|_{\TV},
\]
because $\mu$ is positive. When $\mu$ is a point mass at a single
density $p$, we refer below to a per-density estimate. The corresponding
estimate for a general positive measure $\mu$ follows by integration and
is multiplied by at most $\|\mu\|_{\TV}$.

Fix $0<\delta\le1$. We next derive the correction required to compensate
for the difference between the two variance parameters. Recall the common
factor $M_{0,k}(\mathbf U;u)$ from \eqref{eq:gaussian-factor}. We seek two
retained density-weighted generating series which, before the variance
shift is inserted, have the form
\[
 M_{0,k}(\mathbf U;u)
 \left(
   \rho_k(\mathbf U)\pm\frac12L_k(\mathbf U;u)
 \right),
\]
where $L_k$ is to be chosen. The two variance values
$v\mp\kappa^2\delta/2$ contribute the opposite formal factors
\[
 e^{\mp A_0},
 \qquad
 A_0=\frac{\delta}{4}\sum_{i=1}^k u_i^2.
\]
Hence the difference between the two compensated retained series is
\begin{align*}
 D_k
 &=
 \mathcal T_{k,R}\Bigl\{
 M_{0,k}
 \bigl[
 e^{-A_0}(\rho_k+L_k/2)
 -
 e^{A_0}(\rho_k-L_k/2)
 \bigr]
 \Bigr\}\\
 &=
 \mathcal T_{k,R}\Bigl\{
 M_{0,k}
 \bigl[
 L_k\cosh(A_0)-2\rho_k\sinh(A_0)
 \bigr]
 \Bigr\}\\
 &=
 \mathcal T_{k,R}\Bigl\{
 M_{0,k}\cosh(A_0)
 \bigl[
 L_k-2\rho_k\tanh(A_0)
 \bigr]
 \Bigr\}.
\end{align*}
Thus
\begin{equation}\label{eq:compensation-identity}
 D_k
 =
 \mathcal T_{k,R}\Bigl\{
 M_{0,k}\cosh(A_0)
 \bigl[L_k-2\rho_k\tanh(A_0)\bigr]
 \Bigr\}.
\end{equation}
If we could choose
\[
 L_k
 =
 \mathcal T_{k,R}
 \bigl[2\rho_k\tanh(A_0)\bigr],
\]
then all retained coefficients would cancel exactly. The interpolation
construction provides an approximation to this ideal correction.

To describe it, for formal coefficient variables
$z=(z_a)_{a\in\mathcal I}$ define
\[
 z(\mathbf U;u)_a
 =
 \sum_{i=1}^k u_i\phi_a(U_i).
\]
The defining property of $E_{k,N}$ gives
\begin{align*}
 z(\mathbf U;u)^T E_{k,N}(\mathbf U)z(\mathbf U;u)
 &=
 \sum_{i,\ell=1}^k
 u_i u_\ell\,
 \phi(U_i)^T E_{k,N}(\mathbf U)\phi(U_\ell)\\
 &=
 \sum_{i=1}^k w_{k,i}(\mathbf U)u_i^2.
\end{align*}
Thus the interpolation matrix replaces the ideal quadratic form
$\sum_i u_i^2$ by the diagonal approximation
$\sum_i w_{k,i}u_i^2$ without introducing mixed terms $u_i u_\ell$.

Write
\begin{equation}\label{eq:tanh-series}
 \tanh x
 =
 \sum_{\substack{j\ge1\\ j\ {\rm odd}}} t_jx^j,
 \qquad
 t_j=[x^j]\tanh x.
\end{equation}
Because the argument of $\tanh$ is quadratic, the term indexed by $j$
has total response degree $2j$. Only odd $j\le R$ can therefore
contribute to the degree-$2R$ truncation.

For $k\ge2$ and odd $j\le R$, define the homogeneous coefficient
polynomial
\begin{equation}\label{eq:local-target}
 \begin{aligned}
 g_{k,j}(\mathbf U;z)
 &=
 \left[
 2\rho_k(\mathbf U)
 \tanh\!\left(
   \frac{\delta}{4}
   z^T E_{k,N}(\mathbf U)z
 \right)
 \right]_{\deg_z=2j}\\
 &=
 2\rho_k(\mathbf U)t_j
 \left(\frac{\delta}{4}\right)^j
 \bigl(z^TE_{k,N}(\mathbf U)z\bigr)^j.
 \end{aligned}
\end{equation}
This is the degree-$2j$ part of the correction obtained by replacing the
ideal quadratic form in the preceding display by its interpolated version.

We now convert this coefficient polynomial into a function of
$2j$ fixed response locations. Put $r=2j$ and define
\[
 \begin{aligned}
 G_{k,j}(\mathbf U;T_1,\ldots,T_r)
 &=
 \frac1{r!}[x_1\cdots x_r]\,
 g_{k,j}\!\left(
   \mathbf U;
   \left(
     \sum_{\ell=1}^r x_\ell\phi_a(T_\ell)
   \right)_{a\in\mathcal I}
 \right).
 \end{aligned}
\]
Thus $G_{k,j}$ is the polarized response function associated with
$g_{k,j}$. The variables $T_1,\ldots,T_r$ remain fixed while we impose
consistency in the design variables $\mathbf U$. They are identified
with actual design locations only after those consistency relations have
been established. The factor $1/r!$ is deliberate: $G_{k,j}$ is the
polarized coefficient itself, whereas an $r$-fold derivative of the
corresponding response polynomial is $r!$ times this quantity.

When the response polynomial is finally evaluated at a configuration
$\mathbf U$, we substitute
\[
 z_a=\sum_{i=1}^m u_i\phi_a(U_i).
\]
By the calculation above, this turns the quadratic form in
\eqref{eq:local-target} into
\[
 \sum_{i=1}^m w_{k,i}u_i^2
\]
on the relevant $k$ design arguments and hence produces the desired
approximation to the variance-compensating terms.

We also record the representation cost of these targets. If
$E_{k,N}$ has coefficient-kernel representation cost at most $A_k$,
then taking the $j$th power of its separated representation, expanding
the resulting degree-$2j$ coefficient polynomial, and converting its
monomials to the normalized response atoms defined above gives a
representation of $G_{k,j}$ of the form
\eqref{eq:general-response-representation} with cost at most
\[
 \operatorname{cost}(G_{k,j})
 \le
 C_R^k\delta^jA_k^j
 \int\|p\|_\infty^k\,d\mu(p)
 \le
 \|\mu\|_{\TV}(C_RP)^k\delta^jA_k^j.
\]
Here $C_R$ depends only on the fixed response-degree cutoff and on the
bounded coefficient law, not on $N$. For a per-density estimate,
$\mu$ is a unit point mass; a fixed common density bound $P$ can then
be absorbed into the exponential constant.

The next proposition converts these prescribed response kernels into two
genuine probability laws. Its proof has three tasks: it first modifies
the kernels so that their requirements are consistent when design variables
are integrated out, then uses \zcref{lem:local-synthesis} to realize the
resulting corrections by signed measures on normalized densities, and
finally absorbs those signed measures into positive priors while controlling
the remaining coefficient error.

\begin{proposition}\label{prop:local-construction}
Fix $s>0$, integers $M,N\ge2$ and $R\ge1$, and a constant
$0<\eta<1$ small enough that every density satisfying
$|p-1|\le\eta$ lies within the prescribed lower and upper bounds for
the design density. Let $\phi=(\phi_a)_{a\in\mathcal I}$ satisfy
\eqref{eq:frame-assumption} and
\[
 \sum_{a\in\mathcal I}\|\phi_a\|_{C^s(\mathcal O)}<\infty.
\]

For each $2\le k\le M$, suppose that
$E_{k,N}(\mathbf U)$ is a real symmetric coefficient matrix, symmetric
under permutations of the design arguments
$\mathbf U=(U_1,\ldots,U_k)$, and satisfying
\[
 \phi(U_i)^T E_{k,N}(\mathbf U)\phi(U_\ell)
 =
 w_{k,i}(\mathbf U)\mathbf1_{\{i=\ell\}},
 \qquad
 0\le w_{k,i}(\mathbf U)\le1.
\]
Assume that $E_{k,N}$ admits a separated coefficient-kernel
representation in the sense of \eqref{eq:coefficient-kernel-cost} with
cost at most $A_k$, and that its integrated diagonal defect satisfies
\[
 \int_{Q^k}
 \sum_{i=1}^k
 \bigl(1-w_{k,i}(\mathbf U)\bigr)\,d\mathbf U
 \le B_k.
\]
Set
\[
 A_*=\max\{1,A_2,\ldots,A_M\},
 \qquad
 B_*=\max_{2\le k\le M}B_k.
\]

There are constants $C_{M,R}<\infty$ and $0<\delta_0\le1$ with the
following property. If $0<\delta\le\delta_0$ and
\begin{equation}\label{eq:local-budget}
 S
 =
 C_{M,R}
 \sum_{\substack{1\le j\le R\\ j\ {\rm odd}}}
 (\delta A_*)^j
 <\frac12,
\end{equation}
then there exist two probability laws $\Pi_+$ and $\Pi_-$ on pairs
$(p,F)$ with the following properties.

Under either law, $p$ is a probability density on $Q$ satisfying
\[
 \int_Qp(x)\,dx=1,
 \qquad
 |p(x)-1|\le\eta
 \quad\text{for a.e. }x\in Q,
\]
and the random density $p$ has the same distribution under $\Pi_+$
and $\Pi_-$. There is a deterministic constant $L<\infty$ such that,
under either law, $F$ almost surely extends to $\mathcal O$ and
satisfies
\[
 \|F\|_{C^s(\mathcal O)}\le L.
\]
Moreover, conditional on $p$, the functions $F$ and $-F$ have the
same distribution under either law. In particular, the signed
corrections used in the construction are realized by genuine
probability laws.

Let $\E_+$ and $\E_-$ denote expectation under these two laws. For
$1\le m\le M$, define
\[
 \mathcal M_{\pm,m}(\mathbf U;u)
 =
 \E_\pm\!\left[
   \prod_{i=1}^m p(U_i)
   \exp\!\left\{\sum_{i=1}^m u_iF(U_i)\right\}
 \right]
\]
and
\[
 D_m(\mathbf U;u)
 =
 \mathcal T_{m,R}\!\left[
 e^{-\delta\sum_{i=1}^m u_i^2/4}\mathcal M_{+,m}(\mathbf U;u)
 -
 e^{ \delta\sum_{i=1}^m u_i^2/4}\mathcal M_{-,m}(\mathbf U;u)
 \right].
\]
For one design point, all retained coefficients agree exactly:
\[
 D_1\equiv0.
\]
For every $2\le m\le M$, the remaining coefficient differences satisfy
\begin{equation}\label{eq:finite-local-contract}
 \sum_{a\in\mathcal A_{m,R}}
 \bigl\|[u^a]D_m\bigr\|_{L^2(Q^m)}^2
 \le
 C_{M,R}\delta^2B_*.
\end{equation}

The same conclusion remains valid after bounded deterministic
substitutions in the response variables. More precisely, for any
$W<\infty$ and measurable functions $w_i:Q\to\mathbb R$ satisfying
$|w_i|\le W$, define
\[
 D_m^{(w)}(\mathbf U;u)
 =
 D_m\bigl(
   \mathbf U;
   w_1(U_1)u_1,\ldots,w_m(U_m)u_m
 \bigr).
\]
Then
\[
 \sum_{a\in\mathcal A_{m,R}}
 \bigl\|[u^a]D_m^{(w)}\bigr\|_{L^2(Q^m)}^2
 \le
 C_{M,R,W}\delta^2B_*,
 \qquad 2\le m\le M,
\]
where $C_{M,R,W}$ is independent of $N,A_k,B_k$, and $\delta$.

Finally, there are constants $\kappa_0>0$ and $C_R<\infty$,
independent of $M$ and $N$, such that for every
$0<\kappa\le\kappa_0$ the response remainder defined in
\eqref{eq:response-remainder-definition} satisfies
\[
 \max_{\pm}
 \operatorname*{ess\,sup}_{\mathbf U\in Q^m}
 \max_{y\in\{-a_Y,0,a_Y\}^m}
 \left|
   \operatorname{Rem}_{\pm,m}(\mathbf U,y;\kappa)
 \right|
 \le
 \mathbf1_{\{m>R\}}C_R^m\kappa^{2R+2},
 \qquad 1\le m\le M.
\]
Let
\[
 \rho_m(\mathbf U)
 =
 \E_\pm\prod_{i=1}^m p(U_i)
\]
denote the common density of the $m$ design points under the two
priors. After dividing the density-weighted response probabilities by
$\rho_m(\mathbf U)$, the same remainder bound holds after increasing
$C_R$. In particular, the remainder vanishes identically when
$m\le R$.

The constant $C_{M,R}$ may depend on $M,R,\eta$ and the fixed family
$\phi$, but not on $N$, the interpolation bounds $A_k,B_k$, or
$\delta$. The constants $\delta_0$, $L$, $\kappa_0$, and $C_R$ may
depend on the fixed family $\phi$, on $R$, on the variance bounds
$v_-,v_+$, the fourth-moment bound $C_4$, and the fixed constants
$v,a_Y$ in \eqref{eq:ternary}, but not on $M$ or $N$. No uniform
bound in $M$ is asserted for $C_{M,R}$.
\end{proposition}

\begin{proof}
The proof has three steps. We first modify the interpolation targets so that
the targets for different observation counts can all arise from a single
prior. We then realize the resulting signed moment corrections by two
genuine probability laws having the same distribution of the design density.
Finally, we compare the realized moments with the ideal
variance-compensating moments and bound the resulting coefficient and
response-truncation errors.

\leavevmode\par\medskip\noindent
\emph{Step 1: Consistency across observation counts.}
A single prior on the design density necessarily produces consistent
density-weighted moments: integrating any one design argument from the
$k$-point moment gives the corresponding $(k-1)$-point moment. The
interpolation targets $G_{k,j}$ constructed separately at different values
of $k$ need not satisfy this requirement. We first modify them to obtain a
consistent family.

The construction uses auxiliary factors that integrate to one in a design
variable but vanish when that design variable is later set equal to one of
the response locations. Fix an odd $j\le R$ and put $r=2j$. For fixed
response locations $T=(T_1,\ldots,T_r)\in Q^r$, define
\begin{equation}\label{eq:consistency-multiplier}
 b_r(u;T)
 =
 \frac{\prod_{a=1}^r\psi(u-T_a)}{d_r(T)},
 \qquad
 d_r(T)
 =
 \int_Q\prod_{a=1}^r\psi(v-T_a)\,dv,
 \qquad
 \psi(x)=\sum_{\ell=1}^d(1-\cos2\pi x_\ell).
\end{equation}
Here $Q$ is regarded only for this construction as the periodic cube.
The function $\psi$ is nonnegative and vanishes exactly when its argument
is zero modulo $\mathbb Z^d$. Hence, for real $T$, the integrand defining
$d_r(T)$ is positive except at the finitely many points
$T_1,\ldots,T_r$ modulo the torus. Thus $d_r(T)>0$. Since $d_r$ is
continuous and $Q^r$ is compact,
\[
 \inf_{T\in Q^r}d_r(T)>0.
\]
It follows directly from the definition that
\[
 \int_Q b_r(u;T)\,du=1,
 \qquad
 b_r(T_a;T)=0,\quad 1\le a\le r.
\]

We will also need multiplication by $b_r-1$ to preserve the separated
representations introduced above with a controlled increase in cost. To
obtain this property, we enlarge the smooth family used for interpolation.
Write $g=(g_a)$ for the original family of interpolation functions. The
numerator and denominator in \eqref{eq:consistency-multiplier} extend
analytically to complex values of the response locations. By the positive
lower bound above and uniform continuity on a sufficiently small compact
complex neighborhood of $Q^r$, there is a strip of positive width in which
$d_r$ has no zero. Taking the minimum of these widths over the finitely many
$r\le2R$ gives one strip that works for every response degree used below.

Choose $\vartheta>0$ smaller than the corresponding Fourier decay rate and
enlarge the family $g$ by the real and imaginary parts of
\[
 e^{-\vartheta|n|_1}g_a(t)e^{2\pi i n\cdot t},
 \qquad
 n\in\mathbb Z^d,
\]
retaining the original functions at $n=0$. We denote the resulting real
family again by $\phi$. If
\[
 b_r(u;T)-1
 =
 \sum_{k,\ell_1,\ldots,\ell_r}
 c^{(r)}_{k,\ell}\,
 e^{2\pi i k\cdot u}
 \prod_{a=1}^r e^{2\pi i\ell_a\cdot T_a},
\]
then the coefficients with $k=0$ vanish because
$\int_Q(b_r(u;T)-1)\,du=0$. Exponential Fourier decay gives
\begin{equation}\label{eq:multiplier-cost}
 K_R
 :=
 C_R\max_{r\le2R}
 \left(
  1+
  \sum_{k,\ell}
  |c^{(r)}_{k,\ell}|
  e^{\vartheta\sum_a|\ell_a|_1}
 \right)
 <\infty,
\end{equation}
and, for every specified representation of the form
\eqref{eq:general-response-representation},
\begin{equation}\label{eq:multiplier-cost-bound}
 \operatorname{cost}\bigl((b_r(U_i;T)-1)H\bigr)
 \le
 K_R\operatorname{cost}(H).
\end{equation}
Indeed, multiplication by $e^{2\pi i\ell\cdot T_a}$ shifts the Fourier
index of the frame function attached to the response location $T_a$ and
changes its weighted coefficient by at most
\[
 e^{\vartheta(|n+\ell|_1-|n|_1)}
 \le e^{\vartheta|\ell|_1}.
\]
The factor $e^{2\pi i k\cdot U_i}$ has modulus one and can be absorbed into
the bounded design factor. Passing to real and imaginary parts and
symmetrizing the at most $2R$ response arguments costs only a constant
depending on $R$. Summing the absolute Fourier coefficients gives
\eqref{eq:multiplier-cost-bound}. The enlarged family still satisfies
\eqref{eq:frame-assumption} and
\[
 \sum_a\|\phi_a\|_{C^s(\mathcal O)}<\infty.
\]
Indeed, the H\"older product inequality gives
\[
 \bigl\|g_a e^{2\pi i n\cdot(\cdot)}\bigr\|_{C^s(\mathcal O)}
 \le
 C_s(1+|n|_1)^{\lceil s\rceil}\|g_a\|_{C^s(\mathcal O)},
\]
and the factor $e^{-\vartheta|n|_1}$ makes the resulting sum over $n$
finite. The interpolation matrices themselves are unchanged: they use
only the original $n=0$ coordinates.

We now define operators that separate the part of a kernel determined by
one of its design marginals from the part having zero integral in that
design variable. With the response locations $T$ still fixed independently
of the design variables, set
\begin{equation}\label{eq:canonical-projections}
 P_iK
 =
 b_r(U_i;T)\int_QK\,dU_i,
 \qquad
 Q_iK
 =
 K-P_iK.
\end{equation}
Because $\int_Qb_r(u;T)\,du=1$,
\[
 P_i^2=P_i,
 \qquad
 \int_QQ_iK\,dU_i=0.
\]
The operators associated with distinct design variables commute. Integrating
one bounded design factor costs at most one in the representation
\eqref{eq:general-response-representation}, while multiplication by $b_r$
is multiplication by $1+(b_r-1)$. Hence
\[
 \operatorname{cost}(P_iK)
 +
 \operatorname{cost}(Q_iK)
 \le
 C_R\operatorname{cost}(K)
\]
after increasing $C_R$. Applying such operators in $m$ design variables
therefore increases the cost by at most $C_R^m$. Once a resulting kernel
has zero integral in every design variable and is symmetric in those
variables, the centering and symmetrization argument preceding
\zcref{lem:local-synthesis} puts it in the form
\eqref{eq:local-separated}, at an additional cost of at most $2^m$.

We apply these operators to the target kernels $G_{k,j}$ constructed above.
The degree-two case requires a separate baseline because the desired
singleton correction is nonzero. For $j=1$, let
\[
 B(T_1,T_2)=\frac{\delta}{2},
\]
and, for a finite positive input measure on design densities, write
$\rho_k$ for the corresponding density-weighted $k$-point design density.
Define
\[
 K_{0,1}=B\rho_0,
 \qquad
 G_{1,1}=B\rho_1,
 \qquad
 K_{k,1}=Q_1\cdots Q_kG_{k,1},
 \quad k\ge1.
\]
Thus the degree-two coefficient polynomial at one design point is
$\delta\rho_1z_0^2/2$. Since $\phi_0=1$ on $Q$, its second derivative
with respect to the response variable is $\delta\rho_1$, which is exactly
the singleton covariance correction required by the variance shift.

For every odd $j\ge3$, set
\[
 K_{k,j}=0,\qquad k<j,
\]
and
\[
 K_{k,j}=Q_1\cdots Q_kG_{k,j},
 \qquad k\ge j.
\]
The zero definition for $k<j$ is forced by degree: a retained monomial of
total degree $2j$ with individual response degrees at most two must involve
at least $j$ distinct response locations.

For every odd $j\le R$ and $0\le m\le M$, reconstruct a consistent
$m$-point family by
\begin{equation}\label{eq:local-reconstruct}
 L_{m,j}(\mathbf U;T)
 =
 \sum_{A\subseteq[m]}
 b_{A^c}(\mathbf U;T)
 K_{|A|,j}(U_A;T),
 \qquad
 b_B(\mathbf U;T)
 =
 \prod_{i\in B}b_r(U_i;T).
\end{equation}
The two properties built into this definition are
\begin{equation}\label{eq:consistency-invariants}
 \int_Q
 K_{k,j}(\mathbf U;T)\,dU_i
 =
 0,
 \qquad k\ge1,
\end{equation}
and
\begin{equation}\label{eq:consistency-invariants-2}
 \int_Q
 L_{m,j}(\mathbf U;T)\,dU_i
 =
 L_{m-1,j}(\mathbf U_{\widehat i};T).
\end{equation}
The first follows from the factor $Q_i$ in $K_{k,j}$. For the second,
consider a term in \eqref{eq:local-reconstruct}. If $i\in A$, then
$U_i$ appears inside $K_{|A|,j}$ and the term integrates to zero. If
$i\notin A$, then the term contains $b_r(U_i;T)$, whose integral is one.
The surviving terms are therefore exactly those in the reconstruction at
count $m-1$. These identities are established with $T$ fixed, before any
response location is set equal to a design location.

To apply \zcref{lem:local-synthesis}, we next rewrite the consistent family
in terms of kernels that have zero integral in every design variable.
Define
\begin{equation}\label{eq:local-ordinary}
 H_{m,j}(\mathbf U;T)
 =
 \sum_{A\subseteq[m]}
 (b_r-1)_{A^c}(\mathbf U;T)
 K_{|A|,j}(U_A;T),
\end{equation}
where
\[
 (b_r-1)_B
 =
 \prod_{i\in B}(b_r(U_i;T)-1).
\]
Every design variable in each summand appears either in a centered factor
$b_r-1$ or in a kernel $K_{|A|,j}$ having zero integral in that variable.
Hence $H_{m,j}$ has zero integral in every design variable. The sum over all
subsets also makes it symmetric in the design arguments.

The identity
\[
 b_r=1+(b_r-1)
\]
shows, by expanding over subsets, that the family $L_{m,j}$ is recovered
from the kernels $H_{q,j}$ through
\[
 L_{m,j}(\mathbf U;T)
 =
 \sum_{A\subseteq[m]}
 H_{|A|,j}(U_A;T).
\]
For the representation-cost bound, take $\mu=\delta_p$ to be a unit
point mass. The cost bound for $G_{k,j}$, the bounds for the operators
$P_i,Q_i$, the multiplier estimate \eqref{eq:multiplier-cost-bound}, and
the finite number of subset sums give a specified representation of
$H_{m,j}$ of cost at most
\[
 C_{M,R}(\delta A_*)^j,
\]
where the fixed bound $\|p\|_\infty\le P$ is absorbed into $C_{M,R}$. 
This estimate also includes the term arising from $K_{0,1}$, which
contributes
\[
 (b_r-1)_{[m]}K_{0,1}
\]
at every count $m$, not only at count zero.

The degree-two baseline needs one further adjustment before we can apply
the synthesis lemma, because that lemma represents only corrections
involving at least two perturbed density factors. At counts zero and one,
the definitions above give
\[
 L_{0,1}=B\rho_0,
 \qquad
 K_{1,1}=B\rho_1-b_rB\rho_0,
 \qquad
 L_{1,1}=b_rB\rho_0+K_{1,1}=B\rho_1.
\]
Thus the singleton baseline $B\rho_m$ is already a consistent family across
all counts. It must be subtracted as a whole, rather than only at count one.

For example, when $m=2$, writing
\[
 b_i=b_2(U_i;T),
\]
we have
\[
 L_{2,1}
 =
 b_1b_2B\rho_0
 +
 b_2K_{1,1}(U_1;T)
 +
 b_1K_{1,1}(U_2;T)
 +
 Q_1Q_2G_{2,1}.
\]
Substituting the expression for $K_{1,1}$ and using the marginal identities
for $\rho_0,\rho_1,\rho_2$ gives
\begin{equation}\label{eq:two-count-example}
 L_{2,1}-B\rho_2
 =
 Q_1Q_2(G_{2,1}-B\rho_2).
\end{equation}
Thus after removing $B\rho_2$, the count-two correction has zero integral
in each design variable, while the corrections at counts zero and one
vanish.

For general $m$, define
\begin{equation}\label{eq:corrected-canonical}
 \begin{aligned}
 \widetilde H_{m,1}(\mathbf U;T)
 &=
 H_{m,1}(\mathbf U;T)
 -
 B(T)
 \int
 \prod_{i=1}^m(p(U_i)-1)\,d\mu(p),\\
 \widetilde H_{m,j}
 &=
 H_{m,j},
 \qquad
 3\le j\le R,\quad j\ {\rm odd}.
 \end{aligned}
\end{equation}
The empty product at $m=0$ equals one. Expanding
\[
 \prod_{i=1}^mp(U_i)
 =
 \prod_{i=1}^m\bigl[1+(p(U_i)-1)\bigr]
\]
shows that the subtracted terms in
\eqref{eq:corrected-canonical} are exactly the centered components of the
consistent family $B\rho_m$. Since every density $p$ is normalized,
each factor $p-1$ has integral zero. Consequently,
\[
 \widetilde H_{0,j}
 =
 \widetilde H_{1,j}
 =
 0,
\]
and for every $0\le m\le M$,
\begin{equation}\label{eq:corrected-reconstruction}
 \sum_{\substack{A\subseteq[m]\\|A|\ge2}}
 \widetilde H_{|A|,j}(U_A;T)
 =
 \begin{cases}
 L_{m,1}(\mathbf U;T)-B(T)\rho_m(\mathbf U),
    &j=1,\\
 L_{m,j}(\mathbf U;T),
    &3\le j\le R,\quad j\ {\rm odd}.
 \end{cases}
\end{equation}

We may therefore apply \zcref{lem:local-synthesis} to the kernels
$\widetilde H_{q,j}$ with $2\le q\le M$, separately for each response
degree $2j$. Their representation costs are bounded by
$C_{M,R}(\delta A_*)^j$. The synthesis lemma and
\zcref{lem:local-scalar} then give a signed measure whose total variation
is at most
\[
 S
 =
 C_{M,R}
 \sum_{\substack{1\le j\le R\\ j\ {\rm odd}}}
 (\delta A_*)^j.
\]
The fixed factor
\[
 I_M\sum_{q=2}^Me^q
\]
from \zcref{lem:local-synthesis} has been absorbed into $C_{M,R}$.
All of these operations can be carried out for a unit point mass at any
normalized density $p$ satisfying $|p-1|\le\eta$, with the same
total-variation bound $S$. By linearity, integrating these per-density
constructions against an arbitrary finite positive input measure $\mu$
on such densities gives total variation at most
$S\|\mu\|_{\TV}$.

\leavevmode\par\medskip\noindent
\emph{Step 2: Turning the signed corrections into positive priors.}
Step~1 gives the required corrections as signed measures. We now construct
two genuine probability laws that realize those corrections while having
exactly the same distribution of the design density.

For a finite measure $\nu$ on the set of admissible design densities, let
\[
 \rho_m[\nu](\mathbf U)
 =
 \int\prod_{i=1}^mp(U_i)\,d\nu(p),
\]
and let
\[
 \mathcal L_m[\nu](\mathbf U;z)
\]
be the coefficient polynomial obtained, degree by degree, from the
consistent response functions $L_{m,j}[\nu]$ through the polarization
correspondence \eqref{eq:external-tensor}. Define
\[
 \mathcal C_m[\nu](\mathbf U;z)
 =
 \mathcal L_m[\nu](\mathbf U;z)
 -
 \frac{\delta}{2}z_0^2\rho_m[\nu](\mathbf U).
\]
Thus $\mathcal C_m[\nu]$ removes the degree-two baseline that will later
be supplied directly by the coefficient law. Both
$\rho_m[\nu]$ and $\mathcal C_m[\nu]$ depend linearly on the input measure
$\nu$.

For a point mass $\delta_\omega$ at one admissible density
$p_\omega$, Step~1 provides a specified signed representation of
$\mathcal C_m[\delta_\omega]$, valid simultaneously for every $m\le M$,
with total variation at most $S$. We now choose one common measurable
parameter space for all terms in these representations. Let
$\mathcal Z$ be the disjoint union of the parameter spaces used for the
coefficient indices, Fourier indices, signs in the polarization formula,
nodes in the scalar interpolation measure, and parameters in
\eqref{eq:coefficient-kernel-cost}. All these spaces may be taken to be
standard Borel.

After collecting the specified representations, there are jointly measurable
quantities
\[
 \lambda(\omega,\zeta),
 \qquad
 a_{\omega,\zeta}:Q\to\mathbb R,
 \qquad
 h_{\omega,\zeta},
\]
such that
\[
 \|a_{\omega,\zeta}\|_\infty\le1,
 \qquad
 \int_Qa_{\omega,\zeta}(u)\,du=0,
\]
and $h_{\omega,\zeta}$ is one of the normalized, mean-zero, even scores
defined above. Moreover,
\[
 \int_{\mathcal Z}
 |\lambda(\omega,\zeta)|\,d\sigma(\zeta)
 \le S.
\]
If the scalar interpolation node associated with $\zeta$ is $t_\zeta$,
define the corresponding output density by
\[
 p_{\omega\frown\zeta}(u)
 =
 1+t_\zeta a_{\omega,\zeta}(u).
\]
Since $|t_\zeta|\le\eta$, this is normalized and satisfies
\[
 |p_{\omega\frown\zeta}-1|\le\eta.
\]

To iterate the construction measurably, let
\[
 \Omega
 =
 \bigsqcup_{\ell\ge0}\mathcal Z^\ell
\]
be the space of finite paths. The empty path $\varnothing$ is assigned
the uniform density $p_\varnothing\equiv1$ and score zero. A child
$\omega\frown\zeta$ has the density just defined. Keeping different paths
as distinct states, even when they happen to produce the same density,
avoids any ambiguity in the recursive construction. Parameters with zero
coefficient may be assigned the uniform density and zero score.

With this notation, the signed representation from Step~1 becomes
\begin{equation}\label{eq:correction-representation}
 \mathcal C_m[\nu](\mathbf U;z)
 =
 \int_\Omega\int_{\mathcal Z}
 \lambda(\omega,\zeta)
 \left(\prod_{i=1}^mp_{\omega\frown\zeta}(U_i)\right)
 \mathcal G(h_{\omega,\zeta})(z)
 \,d\sigma(\zeta)\,d\nu(\omega).
\end{equation}
We now dominate this signed representation by a positive kernel on the
state space. Define
\[
 \mathcal B(\omega,\cdot)
 =
 \int_{\mathcal Z}
 2|\lambda(\omega,\zeta)|
 \delta_{\omega\frown\zeta}(\cdot)
 \,d\sigma(\zeta)
\]
and assign to the child state the score
\[
 s_{\omega\frown\zeta}
 =
 \frac14
 \operatorname{sgn}(\lambda(\omega,\zeta))
 h_{\omega,\zeta}.
\]
Then
\[
 \mathcal B(\omega,\Omega)
 \le2S
\]
for every $\omega$. Moreover, the factor $2$ in the definition of
$\mathcal B$ and the factor $1/4$ in the score have been chosen so that
\[
 2\,
 \bigl(2|\lambda|\bigr)\,
 \mathcal G\!\left(
   \frac{\operatorname{sgn}(\lambda)}4h
 \right)
 =
 \lambda\,\mathcal G(h).
\]
Thus a child state with positive mass $2|\lambda|$ and the indicated score
reproduces exactly the signed coefficient $\lambda\mathcal G(h)$.

By \eqref{eq:local-budget}, $2S<1$. Hence the series of positive measures
\[
 \bar\mu
 =
 \sum_{\ell\ge0}\mathcal B^\ell\delta_{\varnothing}
\]
converges in total variation. Let
\begin{equation}\label{eq:positive-fixed-point}
 \begin{gathered}
 Z_{\rm path}
 =
 \bar\mu(\Omega),
 \qquad
 \mu
 =
 Z_{\rm path}^{-1}\bar\mu,\\
 1\le Z_{\rm path}\le(1-2S)^{-1},
 \qquad
 \mu
 =
 Z_{\rm path}^{-1}\delta_{\varnothing}
 +
 \mathcal B\mu.
 \end{gathered}
\end{equation}
The probability measure $\mu$ will be the common distribution of the
design-density state under the two priors.

It remains to specify, conditional on each state $\omega$, the two laws of
the random coefficients $B=(B_a)$. Let
\[
 q_0(B)
 =
 \frac{\delta}{4}(B_0^2-1),
\]
where $B_0$ is the coefficient attached to the distinguished constant
function $\phi_0=1$. Since
$\operatorname{He}_2(B_0)=B_0^2-1$, the generating map satisfies
\[
 \mathcal G(q_0)(z)
 =
 \frac{\delta}{4}z_0^2.
\]
At state $\omega$, define the two conditional coefficient laws by
\[
 \bigl[1\pm(q_0+s_\omega)\bigr]\,d\mathsf R.
\]
Both $q_0$ and $s_\omega$ have mean zero under $\mathsf R$, so these
densities integrate to one. Choose $\delta_0$ sufficiently small that
\[
 \delta_0\|B_0^2-1\|_\infty\le1.
\]
Then
\[
 \|q_0\|_\infty\le\frac14,
 \qquad
 \|s_\omega\|_\infty\le\frac14,
\]
and therefore
\[
 \frac12
 \le
 1\pm(q_0+s_\omega)
 \le
 \frac32.
\]
Thus both conditional coefficient laws are genuine probability laws,
uniformly over the path length. Notice that $s_\omega$ is only the score
attached to the final edge leading to $\omega$; scores from earlier edges
are not accumulated. This is why the positivity bound does not deteriorate
with the path length.

We now verify that these two positive laws realize the desired correction
for the same density-state distribution $\mu$. Since the root has
$s_\varnothing=0$, the fixed-point identity
\eqref{eq:positive-fixed-point} gives
\begin{align*}
 2\int_\Omega
 \left(\prod_{i=1}^mp_\omega(U_i)\right)
 \mathcal G(s_\omega)(z)
 \,d\mu(\omega)
 &=
 2\int_\Omega
 \left(\prod_{i=1}^mp_\omega(U_i)\right)
 \mathcal G(s_\omega)(z)
 \,d(\mathcal B\mu)(\omega).
\end{align*}
Using the definition of $\mathcal B$ and the normalization of
$s_{\omega\frown\zeta}$, the right-hand side is exactly the representation
\eqref{eq:correction-representation}. Hence
\begin{equation}\label{eq:positive-realization}
 \begin{aligned}
 2\int_\Omega
 \left(\prod_{i=1}^mp_\omega(U_i)\right)
 \mathcal G(s_\omega)(z)
 \,d\mu(\omega)
 &=
 \mathcal C_m[\mu](\mathbf U;z),\\
 2\int_\Omega
 \left(\prod_{i=1}^mp_\omega(U_i)\right)
 \mathcal G(q_0+s_\omega)(z)
 \,d\mu(\omega)
 &=
 \frac{\delta}{2}z_0^2\rho_m[\mu](\mathbf U)
 +
 \mathcal C_m[\mu](\mathbf U;z)\\
 &=
 \mathcal L_m[\mu](\mathbf U;z).
 \end{aligned}
\end{equation}
Thus the nonroot states reproduce the correction generated from the same
measure $\mu$ that serves as their common density-state law, while $q_0$
supplies the degree-two baseline at every observation count.

Finally, the base law $\mathsf R$ is invariant under $B\mapsto-B$, and
both $q_0$ and every $s_\omega$ are even functions of $B$. Therefore each
of the two conditional coefficient laws is also invariant under
$B\mapsto-B$. Since
\[
 F=\sum_aB_a\phi_a
\]
changes to $-F$ under this transformation, conditional on the density
$p_\omega$, the functions $F$ and $-F$ have the same distribution. The
bounded support of the coefficients and the summability
$\sum_a\|\phi_a\|_{C^s(\mathcal O)}<\infty$, preserved under the Fourier
enlargement above, give the same deterministic $C^s(\mathcal O)$ bound
for $F$ at every state, independently of $N$ and of the path length.

\leavevmode\par\medskip\noindent
\emph{Step 3: Bounding the compensated response coefficients.}
We now compare the coefficient polynomial realized in
\eqref{eq:positive-realization} with the ideal correction
$2\rho_m\tanh(A_0)$ from \eqref{eq:compensation-identity}. Evaluate
$\mathcal L_m[\mu]$ at the design points by setting
\begin{equation}\label{eq:evaluated-contrast}
 L_m(\mathbf U;u)
 =
 \mathcal T_{m,R}
 \mathcal L_m[\mu]\!\left(
   \mathbf U;
   \left(
     \sum_{i=1}^mu_i\phi_a(U_i)
   \right)_{a\in\mathcal I}
 \right).
\end{equation}
Define the difference
\[
 \mathcal E_m(\mathbf U;u)
 =
 L_m(\mathbf U;u)
 -
 \mathcal T_{m,R}
 \bigl[2\rho_m(\mathbf U)\tanh(A_0)\bigr].
\]
We first bound the derivatives of $\mathcal E_m$ corresponding to its
retained coefficients.

Fix an odd $j\le R$ and a set $A\subseteq[m]$ with $|A|=j$. For a
degree-$2j$ polarized response function $K(\mathbf U;T_1,\ldots,T_{2j})$,
define
\[
 \operatorname{ev}_A K
 =
 (2j)!\,
 K\bigl(
   \mathbf U;
   (U_i,U_i)_{i\in A}
 \bigr).
\]
By the definition \eqref{eq:external-tensor}, this is exactly
\[
 \left.
 \prod_{i\in A}\partial_{u_i}^2
 \right|_{u=0}
\]
applied to the associated response polynomial.

For the target $G_{k,j}$ at any count $k\ge j$, the interpolation identity
gives
\[
 \operatorname{ev}_AG_{k,j}
 =
 \beta_j\delta^j\rho_k
 \prod_{i\in A}w_{k,i},
 \qquad
 \beta_j=2^{1-j}j!t_j.
\]
Indeed,
\[
 [u^{2\mathbf1_A}]
 \left[
  2\rho_kt_j
  \left(
    \frac{\delta}{4}
    \sum_iw_{k,i}u_i^2
  \right)^j
 \right]
 =
 2^{1-2j}j!t_j\delta^j\rho_k
 \prod_{i\in A}w_{k,i},
\]
and applying
$\prod_{i\in A}\partial_{u_i}^2$ multiplies this coefficient by
$(2!)^j$, giving the stated constant $\beta_j$.

We next evaluate the consistent reconstruction
\eqref{eq:local-reconstruct}. Under $\operatorname{ev}_A$, each marked
response location $U_i$, $i\in A$, appears among the fixed response
arguments $T_1,\ldots,T_{2j}$. Hence
\[
 b_r(U_i;T)=0.
\]
Consequently, any term in \eqref{eq:local-reconstruct} that omits a marked
design location $i\in A$ vanishes, because it contains the factor
$b_r(U_i;T)$. Similarly, in a surviving term the $P_i$ part of
$Q_i=I-P_i$ vanishes for every marked $i$, so $Q_i$ acts as the identity
after evaluation. For the remaining, unmarked design variables, evaluation
commutes with integration and hence with $P_i$ and $Q_i$.

Let
\[
 J=[m]\setminus A.
\]
For $S\subseteq J$, write
\[
 P_S=\prod_{i\in S}P_i,
 \qquad
 Q_S=\prod_{i\in S}Q_i,
\]
with empty products equal to the identity. The evaluated reconstruction is
therefore
\[
 \sum_{S\subseteq J}
 b_{J\setminus S}
 Q_S\left[
   \beta_j\delta^j
   \rho_{j+|S|}
   \prod_{i\in A}w_{j+|S|,i}
 \right].
\]
For $j=1$, the count-one target is the baseline defined above, which is
equivalent to setting
\[
 w_{1,1}=1.
\]

To identify the ideal part of this expression, replace every
$w_{k,i}$ by one. The marginal identity for the design densities gives
\[
 b_{J\setminus S}\rho_{j+|S|}
 =
 P_{J\setminus S}\rho_m.
\]
Since the operators $P_i$ and $Q_i$ commute,
\begin{equation}\label{eq:ideal-reconstruction}
 \begin{aligned}
 \sum_{S\subseteq J}
 b_{J\setminus S}Q_S\rho_{j+|S|}
 &=
 \sum_{S\subseteq J}
 P_{J\setminus S}Q_S\rho_m\\
 &=
 \prod_{i\in J}(P_i+Q_i)\rho_m
 =
 \rho_m.
 \end{aligned}
\end{equation}
This is exactly the derivative of the ideal correction
$2\rho_m\tanh(A_0)$.

Subtracting the ideal expression from the actual interpolation expression
therefore gives
\begin{equation}\label{eq:local-defect}
 e_{m,A}
 =
 \sum_{S\subseteq[m]\setminus A}
 b_{([m]\setminus A)\setminus S}
 Q_S
 \left[
   \beta_j\delta^j\rho_{j+|S|}
   \left(
     \prod_{i\in A}w_{j+|S|,i}-1
   \right)
 \right].
\end{equation}
When $j=1$ and the relevant count is one, this difference is zero because
$w_{1,1}=1$.

We now bound \eqref{eq:local-defect} in $L^2$. Since
$|p-1|\le\eta$ under every density state,
\[
 \rho_k\le(1+\eta)^k.
\]
Also, because $0\le w_i\le1$,
\[
 \left|
  1-\prod_{i\in A}w_i
 \right|^2
 \le
 \sum_{i\in A}(1-w_i).
\]
Hence the integrated defect assumption gives
\[
 \left\|
  \rho_k
  \left(
   \prod_{i\in A}w_{k,i}-1
  \right)
 \right\|_{L^2(Q^k)}
 \le
 C_M B_*^{1/2}.
\]
Integration in an unmarked design variable is an $L^2$ contraction because
$Q$ has unit volume. Multiplication by $b_r$ is uniformly bounded, including
after its response arguments have been evaluated. The definitions of
$P_i,Q_i$ therefore imply
\[
 \|Q_SH\|_2
 \le
 C_R^{|S|}\|H\|_2.
\]
There are only finitely many subsets $S\subseteq[m]\setminus A$, with
$m\le M$, so
\[
 \|e_{m,A}\|_{L^2(Q^m)}
 \le
 C_{M,R}\delta^jB_*^{1/2}.
\]

The quantity $e_{m,A}$ is the derivative
$\prod_{i\in A}\partial_{u_i}^2\mathcal E_m|_{u=0}$. Therefore
\[
 [u^{2\mathbf1_A}]\mathcal E_m
 =
 2^{-j}e_{m,A}.
\]
Before multiplication by the common Gaussian factor, these are the only
possibly nonzero retained coefficients. To see this, consider a retained
multi-index $a$ in the degree-$2j$ part and mark
\[
 S=\{i:a_i>0\}.
\]
In the corresponding polarization, the response list contains $U_i$
exactly $a_i$ times for each $i\in S$. The marked-variable argument above
then applies verbatim: any reconstruction term omitting an index $i\in S$
vanishes because $b_r(U_i;T)=0$, while in every surviving term the $P_i$
part of $Q_i$ vanishes at the marked indices and evaluation commutes with
the integrations in the unmarked variables. Thus the surviving raw targets
are evaluated with every marked design location present. Since each such
target is a power of the diagonal quadratic form
$\sum_i w_{k,i}u_i^2$ (with the degree-two baseline likewise diagonal),
this evaluation is zero whenever some $a_i=1$. Hence every possibly
nonzero retained coefficient in degree $2j$ has $a_i\in\{0,2\}$ for all
$i$, and is therefore of the form $u^{2\mathbf1_A}$ with $|A|=j$.
Since $\delta\le1$, summing the squared
bounds over the finitely many odd $j\le R$ and sets
$A\subseteq[m]$ gives
\[
 \sum_{a\in\mathcal A_{m,R}}
 \bigl\|[u^a]\mathcal E_m\bigr\|_{L^2(Q^m)}^2
 \le
 C_{M,R}\delta^2B_*.
\]

The exact compensation identity \eqref{eq:compensation-identity} now gives
\[
 D_m
 =
 \mathcal T_{m,R}
 \bigl[
   M_{0,m}\cosh(A_0)\mathcal E_m
 \bigr].
\]
For fixed $m$ and $R$, the sum of the absolute values of the coefficients
of $M_{0,m}\cosh(A_0)$ through total degree $2R$ is bounded uniformly over
$\mathbf U$. Indeed,
\[
 \sup_{t\in Q}\sum_a|\phi_a(t)|<\infty
\]
by \eqref{eq:frame-assumption}, and $\delta\le1$. Convolution of coefficient
arrays therefore satisfies
\[
 \|c*d\|_{\ell^2}
 \le
 \|c\|_{\ell^1}\|d\|_{\ell^2},
\]
and integration over $Q^m$ yields
\[
 \sum_{a\in\mathcal A_{m,R}}
 \bigl\|[u^a]D_m\bigr\|_{L^2(Q^m)}^2
 \le
 C_{M,R}\delta^2B_*,
\]
which is \eqref{eq:finite-local-contract}. At $m=1$, the realized
correction is exactly the degree-two baseline, so
\[
 D_1\equiv0.
\]

If each response variable is replaced by
$w_i(U_i)u_i$, then the coefficient indexed by $a$ is multiplied by
\[
 \prod_iw_i(U_i)^{a_i}.
\]
If $|w_i|\le W$, the total degree bound $|a|\le2R$ gives an additional
factor at most $\max\{1,W\}^{2R}$ at the coefficient level, and hence at
most $\max\{1,W\}^{4R}$ after squaring. This proves the substituted version
of \eqref{eq:finite-local-contract}, with a constant depending on $W$.

It remains to prove the response-remainder bound. For a fixed density state,
design configuration, and response vector, each factor in
\eqref{eq:response-remainder-definition} is
\[
 p(U_i)
 \left\{
  \beta_{y_i}
  +
  t\ell_{y_i}F(U_i)
  +
  t^2r_{y_i}\bigl(F(U_i)^2\mp\delta/2\bigr)
 \right\}.
\]
The density is bounded by $1+\eta$, while the bounded coefficient law and
the fixed frame give a deterministic uniform bound on $|F|$. Consequently,
the sum of the absolute values of the scalar coefficients in the product of
$m$ such factors is at most
\[
 C_R^m,
\]
uniformly over the state, the design points, and the response vector.
Averaging over either prior preserves this estimate. Conditional symmetry
of $F$ eliminates all odd powers of $t$.

The retained response expansion contains all even powers through degree
$2R$. Therefore, for $0<\kappa\le1$,
\[
 \left|
 \operatorname{Rem}_{\pm,m}(\mathbf U,y;\kappa)
 \right|
 \le
 \mathbf1_{\{m>R\}}
 C_R^m\kappa^{2R+2}.
\]
If $m\le R$, the full response polynomial has degree at most
$2m\le2R$, so the remainder is identically zero.

Finally,
\[
 \rho_m(\mathbf U)
 =
 \E_\pm\prod_{i=1}^mp(U_i)
 \ge
 (1-\eta)^m.
\]
Dividing by $\rho_m$ therefore changes the preceding bound only by another
factor exponential in $m$, which can be absorbed into $C_R^m$. Squaring
and summing over the $3^m$ possible ternary response vectors gives
\[
 \sum_{y\in\{-a_Y,0,a_Y\}^m}
 \left|
  \operatorname{Rem}_{\pm,m}(\mathbf U,y;\kappa)
 \right|^2
 \le
 \mathbf1_{\{m>R\}}
 C_R^m\kappa^{4R+4},
\]
both before and after division by $\rho_m$. All support, smoothness, and
score bounds depend only on the fixed frame and the response-degree cutoff
$R$, and not on $N$ or on the path length. This completes the proof.
\end{proof}

\subsection{Proofs of the local comparison lemmas}\label{sec:local-applications}

\begin{proof}[Proof of \zcref{lem:local}]
Fix $M\ge3$ and set $R=M$, so that every response coefficient arising
from at most $M$ observations is retained. Apply
\zcref{lem:simple-geometry}. For $2\le k\le M$, its interpolation
matrices satisfy
\[
 A_k\le C_MN^2(1+\log N)^{k-2},
 \qquad
 B_k\le C_MN^{-d}(1+\log N)^{k-2}.
\]
With
\[
 \Lambda_N=(1+\log N)^{M-2},
\]
we therefore have
\[
 A_*:=\max\{1,A_2,\ldots,A_M\}
 \le C_MN^2\Lambda_N,
 \qquad
 B_*:=\max_{2\le k\le M}B_k
 \le C_MN^{-d}\Lambda_N.
\]

Choose
\[
 \delta=a_MN^{-2}\Lambda_N^{-1},
\]
where $a_M>0$ is fixed sufficiently small. Then
\[
 \delta A_*\le C_Ma_M.
\]
Since $R=M$ is fixed, choosing $a_M$ small enough makes both
$\delta\le\delta_0$ and the budget in
\eqref{eq:local-budget} satisfy
\[
 C_{M,M}
 \sum_{\substack{1\le j\le M\\ j\ {\rm odd}}}
 (\delta A_*)^j
 <\frac12
\]
uniformly in $N$. We may therefore apply
\zcref{prop:local-construction}.

The proposition produces two probability laws with the same distribution
of the normalized design density, with $|p-1|\le\eta$, conditional sign
symmetry of $F$, and a deterministic $C^s$ bound independent of $N$.
It also gives
\[
 D_1\equiv0
\]
and, for $2\le k\le M$,
\[
 \sum_{a\in\mathcal A_{k,M}}
 \bigl\|[u^a]D_k\bigr\|_{L^2(Q^k)}^2
 \le
 C_M\delta^2B_*
 \le
 C_M\delta^2N^{-d}\Lambda_N.
\]
The same estimate holds after the substitutions
$u_i\mapsto w_i(x_i)u_i$ with $|w_i|\le1$, by the corresponding
conclusion of \zcref{prop:local-construction}.

Finally, if $k\le M=R$, the product of the $k$ quadratic response
probabilities has total degree at most $2k\le2R$. Hence the truncation
$\mathcal T_{k,R}$ discards no response term, and the response remainder
is identically zero. These are exactly the conclusions of
\zcref{lem:local}.
\end{proof}

\begin{proof}[Proof of \zcref{lem:uniform-local}]
Here $R$ is fixed while the largest observation count $M$ may grow. We
therefore use the sharper interpolation matrices from
\zcref{lem:geometry}, which satisfy, for every $k\ge2$,
\[
 A_k\le C^kN^2,
 \qquad
 B_k\le C^kN^{-d},
\]
with a constant $C$ independent of $k,M$, and $N$. The smooth family
$\phi$ used by this construction is fixed once $R$ and the other fixed
parameters are chosen.

Write
\[
 c=\delta N^2.
\]
We first track the total variation required to realize the prescribed
moment corrections. For an odd $j\le R$, the degree-$2j$ target
$G_{k,j}$ is obtained from the $j$th power of the interpolation matrix.
Since $A_k\le C^kN^2$, the representation bound preceding
\zcref{prop:local-construction} gives
\[
 \operatorname{cost}(G_{k,j})
 \le
 C_R^k\delta^jA_k^j
 \le
 C_R^k c^j,
\]
after increasing $C_R$ to absorb the factor $C^{kj}$; this is possible
because $j\le R$ and $R$ is fixed.

The consistency projections, multiplication by the factors $b_r-1$,
the subset reconstruction, and the subtraction of the degree-two
baseline are all described explicitly in
\eqref{eq:multiplier-cost-bound}--\eqref{eq:corrected-canonical}.
Their costs grow at most exponentially in the number $k$ of design
arguments. Consequently, after increasing a constant $K_R>1$ depending
only on the fixed construction,
\[
 \operatorname{cost}(\widetilde H_{k,j})
 \le
 K_R^k c^j,
 \qquad
 2\le k\le M.
\]

We now apply \zcref{lem:local-synthesis}. Its scalar inversion contributes
the factor
\[
 I_M\le C(3/\eta)^M,
\]
and its polarization step contributes at most $e^k$ at design order $k$.
Thus, for $0<c\le1$, the total variation $S$ of all signed corrections
satisfies
\begin{align}
 S
 &\le
 I_M
 \sum_{k=2}^M e^k
 \sum_{\substack{1\le j\le R\\ j\ {\rm odd}}}
 \operatorname{cost}(\widetilde H_{k,j})
 \notag\\
 &\le
 C(3/\eta)^M
 \sum_{k=2}^M(eK_R)^k
 \sum_{\substack{1\le j\le R\\ j\ {\rm odd}}}c^j
 \notag\\
 &\le
 cD_R^M,
 \label{eq:uniform-cost-accounting}
\end{align}
for a fixed $D_R>1$. In the last step we used $c\le1$, so
$\sum_{j\le R,\ j\ {\rm odd}}c^j\le Rc$, and enlarged $D_R$ to absorb
the finite sum over $k$, the factor $R$, and the scalar-inversion
constant.

Choose $a_R>0$ sufficiently small that
\[
 0<c\le a_RD_R^{-M}
\]
implies
\[
 S\le\frac14
 \qquad\text{and}\qquad
 \delta=cN^{-2}\le\delta_0.
\]
Then the positive-measure construction in Step~2 of
\zcref{prop:local-construction} applies. The smooth family, bounded
coefficient distribution, and normalized Hermite scores depend only on
the fixed response cutoff $R$ and not on $M$ or $N$. Hence the resulting
functions $F$ have a deterministic $C^s$ bound that is uniform in
$M,N$, and the two priors have the same design-density distribution and
conditional sign symmetry.

It remains to track the coefficient error uniformly in $M$. Fix an odd
$j\le R$, a count $m\le M$, and a set $A\subseteq[m]$ with $|A|=j$.
In the defect formula \eqref{eq:local-defect}, the interpolation bound
\[
 B_k\le C^kN^{-d}
\]
and the estimates for the consistency operators give
\[
 \|e_{m,A}\|_{L^2(Q^m)}
 \le
 K_R^m\delta^jN^{-d/2}.
\]
Indeed, each term in \eqref{eq:local-defect} contains one diagonal
interpolation defect, whose squared integral is bounded by
$C^kN^{-d}$, while each of the at most $m$ consistency operators changes
the norm by only a fixed factor. Summing over the subsets in that formula
is absorbed into $K_R^m$.

For fixed $R$, the number of retained multi-indices satisfies
\[
 |\mathcal A_{m,R}|
 \le C_R(1+m)^{2R}.
\]
We also need the coefficient size of the common multiplier
$M_{0,m}\cosh(A_0)$. Through total degree $2R$, its coefficient
$\ell^1$ norm satisfies
\[
 \bigl\|M_{0,m}\cosh(A_0)\bigr\|_{\ell^1}
 \le
 C_R(1+m)^{2R}.
\]
To see this, expand both exponentials only through degree $2R$. Every
coefficient is a sum of products of at most $R$ entries of the covariance
kernel
\[
 \sum_a\phi_a(U_i)\phi_a(U_\ell),
\]
which is uniformly bounded because the fixed frame is absolutely
summable.

Using the coefficient convolution inequality
\[
 \|c*d\|_{\ell^2}
 \le
 \|c\|_{\ell^1}\|d\|_{\ell^2}
\]
as in Step~3 of \zcref{prop:local-construction}, and using
$\delta^j\le\delta$ for $j\ge1$, we obtain
\begin{align}
 \sum_{a\in\mathcal A_{m,R}}
 \bigl\|[u^a]D_m\bigr\|_{L^2(Q^m)}^2
 &\le
 C_R(1+m)^{6R}K_R^{2m}\delta^2N^{-d}
 \notag\\
 &\le
 C_R^m\delta^2N^{-d},
 \qquad m\ge2,
 \label{eq:uniform-error-accounting}
\end{align}
after increasing the final constant. This proves
\eqref{eq:uniform-contract}. If each $u_i$ is replaced by
$w_i(x_i)u_i$ with $|w_i|\le1$, every retained coefficient is merely
multiplied by a factor of absolute value at most one, so the same bound
continues to hold.

Finally, the response-remainder argument at the end of
\zcref{prop:local-construction} uses only the fixed $C^s$ support bound,
the fixed ternary response law, and the fixed degree cutoff $R$. It
therefore gives, uniformly in $M$ and $N$,
\[
 \max_{\pm}
 \operatorname*{ess\,sup}_{x\in Q^m}
 \max_{y\in\{-a_Y,0,a_Y\}^m}
 |\operatorname{Rem}_{\pm,m}(x,y;\kappa)|
 \le
 \mathbf1_{\{m>R\}}C_R^m\kappa^{2R+2},
\]
together with the same estimate after division by $\rho_m$ and the
corresponding squared-sum bound over response vectors. This is
\eqref{eq:response-remainder-bound} and completes the proof of
\zcref{lem:uniform-local}.
\end{proof}

\section{From one block to the full sample}\label{sec:gluing}
We now prove \zcref{lem:gluing}, stated in
\zcref{sec:lower-inputs}. One step in the gluing construction reweights
the independent local perturbations by a common exponential factor, chosen
to compensate for the random total mass of the resulting design density.
This reweighting introduces dependence between the local perturbations. The
following lemma shows that, provided the reweighting couples any one
coordinate only weakly to the others, bounded functions still satisfy a
product-type variance bound. We will use this estimate to control the
random normalizing constant of the design density.

We first define the quantities that measure the dependence introduced by
the reweighting. Let
\[
 \mathcal X=\prod_{j=1}^J\mathcal X_j
\]
be a product of finitely many spaces, and let
$x=(x_1,\ldots,x_J)\in\mathcal X$. If $y\in\mathcal X_j$, write
$x^{j\leftarrow y}$ for the point obtained from $x$ by replacing its
$j$th coordinate by $y$. For a bounded function $G:\mathcal X\to\mathbb R$,
define
\[
 \operatorname{osc}_j(G)
 =
 \sup_{x\in\mathcal X}\sup_{y\in\mathcal X_j}
 \bigl|G(x)-G(x^{j\leftarrow y})\bigr|.
\]
For distinct $i,j$, if $y\in\mathcal X_i$ and $z\in\mathcal X_j$, define
\[
 \operatorname{osc}_{ij}(G)
 =
 \sup_{x,y,z}
 \left|
 G(x)-G(x^{i\leftarrow y})-G(x^{j\leftarrow z})
 +G(x^{i\leftarrow y,j\leftarrow z})
 \right|.
\]
Thus $\operatorname{osc}_j(G)$ measures the largest effect of changing
coordinate $j$, while $\operatorname{osc}_{ij}(G)$ measures how much the
effect of changing coordinate $i$ can itself change when coordinate $j$
is changed.

\begin{lemma}\label{lem:tilted-variance}
Let
\[
 \pi=\bigotimes_{j=1}^J\pi_j
\]
be a product probability measure on
$\mathcal X=\prod_{j=1}^J\mathcal X_j$, where each $\mathcal X_j$ is a
standard Borel space. Let $m:\mathcal X\to\mathbb R$ be bounded and
measurable. For $\lambda\ge0$, define the tilted probability measure
$\pi_\lambda$ by
\[
 \frac{d\pi_\lambda}{d\pi}(x)
 =
 \frac{e^{\lambda(m(x)-1)}}{
       \int e^{\lambda(m(y)-1)}\,d\pi(y)}.
\]

There is an absolute constant $c_*>0$ with the following property. For
$i\ne j$, set
\[
 c_{ij}
 =
 c_*\lambda\,\operatorname{osc}_{ij}(m),
 \qquad
 c_{ii}=0,
\]
and define
\[
 q=\max_{1\le i\le J}\sum_{j=1}^J c_{ij}.
\]
If
\[
 q<\frac12,
\]
then every bounded measurable function
$G:\mathcal X\to\mathbb R$ satisfies
\begin{equation}\label{eq:finite-coordinate-variance}
 \operatorname{Var}_{\pi_\lambda}(G)
 \le
 \frac{1}{4(1-q)^2}
 \sum_{j=1}^J\operatorname{osc}_j(G)^2.
\end{equation}
\end{lemma}

\begin{proof}
We first bound how strongly one coordinate can affect the conditional law of
another under the tilted measure. Fix distinct $i,j$. If all coordinates
except $i$ are held fixed, the conditional density of $x_i$ under
$\pi_\lambda$ is proportional to
\[
 \exp\{\lambda m(x)\}\,d\pi_i(x_i).
\]
Changing only coordinate $j$ therefore changes the logarithm of the ratio of
two such conditional densities, as a function of $x_i$, by at most
\[
 \lambda\,\operatorname{osc}_{ij}(m).
\]
An elementary comparison between two probability measures whose log-density
ratio has oscillation at most $\alpha$ gives total variation distance at
most $c_*\alpha$, for an absolute constant $c_*$. Hence changing coordinate
$j$ changes the conditional law of coordinate $i$ by at most
\[
 c_{ij}=c_*\lambda\,\operatorname{osc}_{ij}(m).
\]
The matrix $C_0=(c_{ij})$ is symmetric, because
$\operatorname{osc}_{ij}(m)=\operatorname{osc}_{ji}(m)$, and by assumption
every row sum, and hence every column sum, is at most $q<1/2$.

We next quantify how changing one exposed coordinate can propagate through
the remaining coordinates. Fix a set $S$ of coordinates that have not yet
been exposed, and fix two different boundary values of a coordinate
$i\notin S$. Let $\mu$ and $\mu'$ be the corresponding conditional laws of
the coordinates in $S$ under $\pi_\lambda$. We construct a coupling of
$\mu$ and $\mu'$ by random-scan single-coordinate updates.

Start from any coupling of $\mu$ and $\mu'$. At each step, choose a
coordinate $j\in S$ uniformly at random and update that coordinate in both
copies using a maximal coupling of its two conditional laws given the
current values of all other coordinates. Since each marginal chain starts
in its stationary conditional law, both marginals remain equal to $\mu$
and $\mu'$ at every time.

Let
\[
 p_j(t)
 =
 \Pp\{\text{the two copies disagree at coordinate $j$ after $t$ updates}\},
 \qquad j\in S,
\]
and write $p(t)=(p_j(t))_{j\in S}$. When coordinate $j$ is updated, its two
conditional laws may differ because the boundary value at coordinate $i$
is different and because the two copies may disagree at some coordinates
$\ell\in S$. By changing those coordinates one at a time and using the
conditional total-variation bounds above, the maximal-coupling failure
probability is at most
\[
 c_{ji}+\sum_{\ell\in S}c_{j\ell}
 \mathbf1_{\{X_\ell\ne X'_\ell\}}.
\]
Taking expectations and averaging over the uniformly chosen update
coordinate gives the componentwise inequality
\[
 p(t+1)
 \le
 \left(
 I-\frac{I-(C_0)_S}{|S|}
 \right)p(t)
 +
 \frac{(C_0)_{Si}}{|S|},
\]
where $(C_0)_S$ is the submatrix of $C_0$ indexed by $S$ and
$(C_0)_{Si}=(c_{ji})_{j\in S}$.

The matrix
\[
 I-\frac{I-(C_0)_S}{|S|}
\]
has nonnegative entries and row sums at most
\[
 1-\frac{1-q}{|S|}<1.
\]
Iterating the preceding inequality therefore yields
\[
 \limsup_{t\to\infty}p(t)
 \le
 (I-(C_0)_S)^{-1}(C_0)_{Si}.
\]
We do not need convergence of the coupled chains themselves; the bound on
the disagreement probabilities is enough.

We now apply this coupling estimate to conditional expectations of $G$.
Set
\[
 g_j=\operatorname{osc}_j(G),
 \qquad
 g=(g_1,\ldots,g_J)^T.
\]
If two configurations differ at coordinate $i$ and, after coupling the
remaining coordinates in $S$, also differ at a set of coordinates in $S$,
then telescoping one coordinate at a time gives
\[
 |G(X)-G(X')|
 \le
 g_i+\sum_{j\in S}g_j\mathbf1_{\{X_j\ne X'_j\}}.
\]
Taking expectations under the coupling above shows that the range of the
conditional expectation of $G$, as the boundary value of coordinate $i$
varies, is at most
\[
 g_i
 +
 g_S^T(I-(C_0)_S)^{-1}(C_0)_{Si}.
\]
Since $C_0$ has nonnegative entries, expanding the inverse in a Neumann
series gives
\[
 g_i
 +
 g_S^T(I-(C_0)_S)^{-1}(C_0)_{Si}
 \le
 [(I-C_0)^{-1}g]_i.
\]

Finally, expose the coordinates successively and let
\[
 M_i
 =
 \E_{\pi_\lambda}[G\mid X_1,\ldots,X_i],
 \qquad
 0\le i\le J,
\]
be the associated Doob martingale. The preceding bound shows that, conditional
on $X_1,\ldots,X_{i-1}$, the range of the increment $M_i-M_{i-1}$ is at
most
\[
 [(I-C_0)^{-1}g]_i.
\]
A random variable supported on an interval of length $a$ has variance at
most $a^2/4$. Orthogonality of the martingale increments therefore gives
\[
 \operatorname{Var}_{\pi_\lambda}(G)
 \le
 \frac14
 \bigl\|(I-C_0)^{-1}g\bigr\|_2^2.
\]
Because $C_0$ is symmetric and every row sum is at most $q$,
\[
 \|C_0\|_{2\to2}\le q.
\]
Hence
\[
 \|(I-C_0)^{-1}\|_{2\to2}
 \le
 \frac1{1-q},
\]
and so
\[
 \operatorname{Var}_{\pi_\lambda}(G)
 \le
 \frac1{4(1-q)^2}\|g\|_2^2
 =
 \frac1{4(1-q)^2}
 \sum_{j=1}^J\operatorname{osc}_j(G)^2.
\]
This proves \eqref{eq:finite-coordinate-variance}.
\end{proof}

\begin{proof}[Proof of \zcref{lem:gluing}]
The proof has four steps. We first place independent copies of the local
construction across the design space and verify that their variance shifts
combine to a single constant shift. The resulting product of local density
perturbations is not exactly normalized, so we next reweight their common
density law in a way that cancels this random mass in a Poisson experiment.
We then combine the local coefficient comparisons over the connected
components of the observed design. Finally, we normalize the design density
and bound the Hellinger cost of doing so.

\leavevmode\par\medskip\noindent
\emph{Step 1: Combining the local constructions.}
Temporarily identify opposite faces of $[0,1]^d$, so that the design space
is the unit torus. Choose cubes $(Q_j)$ of diameter at most $Ch$ and volume
comparable to $h^d$ such that every point belongs to at most a fixed number
of cubes and each cube intersects at most a fixed number of the others.
There are at most $Ch^{-d}$ cubes. We can choose smooth functions $w_j$
supported strictly inside $Q_j$ such that
\begin{equation}\label{eq:smooth-windows}
 \sum_jw_j(x)^2=1,
 \qquad
 \|\partial^\gamma w_j\|_\infty
 \le C_\gamma h^{-|\gamma|}
 \quad\text{for every multi-index }\gamma.
\end{equation}
For example, start from translates of one nonnegative smooth bump on the
$h$-grid whose positive sets cover the torus, and divide each translate by
the square root of the sum of the squared translates.

In each block $Q_j$, take an independent copy of the local pair $(p,F)$
from the two priors in the hypothesis of the lemma, rescaled from the
reference cube $Q$ to $Q_j$. Denote the resulting functions by
$(p_j,F_j)$. Extend $p_j$ by $1$ outside $Q_j$. The random density $p_j$
has the same distribution under the two signs, and conditional on $p_j$,
$F_j$ and $-F_j$ have the same distribution. Define
\begin{equation}\label{eq:global-fields}
 \widetilde p(x)=\prod_jp_j(x),
 \qquad
 f(x)=\kappa\sum_jw_j(x)F_j(x),
 \qquad
 \kappa=c_fh^s.
\end{equation}

Because only a bounded number of blocks overlap at any point and
$|p_j-1|\le\eta$, the unnormalized density $\widetilde p$ has deterministic
positive upper and lower bounds. The same bounded-overlap property,
\eqref{eq:smooth-windows}, and the uniform $C^s$ bound on each rescaled
$F_j$ give
\[
 \|f\|_{C^s}\le Cc_f.
\]
The support of each $w_j$ lies strictly inside $Q_j$, so $w_jF_j$ vanishes
in a collar of the boundary of the block and therefore extends as a
$C^s$ function across that boundary. Periodically lifting the torus
construction therefore gives a $C^s$ extension to a neighborhood of
$[0,1]^d$. 
We choose $c_f>0$ sufficiently small that the prescribed
$C^s$ bound and the response-probability and fourth-moment constraints are
satisfied. Since $\delta\le1$ and $\kappa\le c_f$, the variance values
\[
 v\mp\frac{\Delta}{2},
 \qquad
 \Delta=\kappa^2\delta,
\]
also remain in the admissible variance interval.

Let $\E_\pm$ denote expectation under the corresponding product of the
independent local priors, before any reweighting. For fixed design points
$x_1,\ldots,x_k$, independence of the blocks gives
\begin{equation}\label{eq:block-product}
 \E_\pm\!\left[
   \prod_{i=1}^k\widetilde p(x_i)
   \exp\left\{\sum_{i=1}^ku_i f(x_i)/\kappa\right\}
 \right]
 =
 \prod_j
 \E_\pm\!\left[
   \prod_{i:x_i\in Q_j}p_j(x_i)
   \exp\left\{
     \sum_{i:x_i\in Q_j}w_j(x_i)u_iF_j(x_i)
   \right\}
 \right].
\end{equation}
A point belonging to more than one block appears in each corresponding
factor, exactly as it does in the definitions of $\widetilde p$ and $f$.

We now apply the local compensated identities inside the product on the
right. In block $j$, replace each formal response variable by
$w_j(x_i)u_i$. The hypothesis of the lemma guarantees that the local
coefficient estimate is unchanged by this substitution. The quadratic
variance-compensation factor from block $j$ is
\[
 \exp\left\{
   \mp\frac{\delta}{4}
   \sum_iw_j(x_i)^2u_i^2
 \right\}.
\]
Multiplying these factors over $j$ and using
$\sum_jw_j(x)^2=1$ gives
\[
 \prod_j
 \exp\left\{
   \mp\frac{\delta}{4}
   \sum_iw_j(x_i)^2u_i^2
 \right\}
 =
 \exp\left\{
   \mp\frac{\delta}{4}\sum_i u_i^2
 \right\}.
\]
Thus all local variance corrections combine into the single constant
variance shift $v\mp\Delta/2$.

Let
\[
 q_\pm(y\mid x)
 =
 q_y\left(f(x),v\mp\frac{\Delta}{2}\right).
\]
As in \eqref{eq:full-retained-response-map}, applying
$\mathcal Q_y\mathcal T_{k,k}$ to the untruncated compensated product,
after the physical substitution $u_i\mapsto\kappa u_i$, gives
\[
 \E_\pm
 \prod_{i=1}^k
 \widetilde p(x_i)q_\pm(y_i\mid x_i).
\]
Replacing $\mathcal T_{k,k}$ by $\mathcal T_{k,R}$ keeps only the terms
of total response degree at most $2R$. The difference is the corresponding
response remainder. Since the global function
$\sum_jw_jF_j$ has a deterministic uniform bound and is conditionally
symmetric, the same scalar-polynomial estimate used in
\eqref{eq:response-remainder-bound} applies to this global remainder.

\leavevmode\par\medskip\noindent
\emph{Step 2: Removing the random total mass.}
Although each local factor $p_j$ is normalized on its block, their product
$\widetilde p$ need not integrate to one. Write
\[
 p_j=1+b_j,
 \qquad
 \int_{Q_j}b_j(x)\,dx=0,
 \qquad
 |b_j|\le\eta,
\]
and define its total mass
\[
 m=\int_{[0,1]^d}\widetilde p(x)\,dx.
\]
Let $\pi$ denote the product law of the independent density states
$(p_j)_j$. This law is common to the two priors.

We first quantify how much $m$ can depend on any one local density. Because
the linear terms in the expansion of $\prod_j(1+b_j)$ integrate to zero,
only products involving overlapping perturbations contribute to $m-1$.
Bounded overlap therefore gives
\begin{equation}\label{eq:mass-osc}
 |m-1|\le C\eta^2,
 \qquad
 \operatorname{osc}_j(m)\le C\eta^2h^d,
 \qquad
 \operatorname{osc}_{ij}(m)
 \le
 C\eta^2h^d
 \mathbf1_{\{Q_i\cap Q_j\ne\varnothing\}}.
\end{equation}
For example, if only the $j$th density is changed from $1+b_j$ to
$1+b_j'$, the change in $m$ is
\[
 \int_{Q_j}
 (b_j'-b_j)
 \left[
   \prod_{i\ne j}(1+b_i)-1
 \right]dx.
\]
The subtracted constant may be inserted because
$\int_{Q_j}(b_j'-b_j)=0$. On $Q_j$, only a bounded number of the remaining
$b_i$ can be nonzero, so the bracket is $O(\eta)$. This proves the
first-coordinate oscillation bound. Taking an alternating difference in
two density coordinates gives the final bound in
\eqref{eq:mass-osc}; it vanishes unless the two blocks overlap. Summing the
terms in the product expansion gives the deterministic bound on $|m-1|$.

For $\lambda\ge0$, define the exponentially reweighted density-state law
\begin{equation}\label{eq:mass-tilt}
 Z_{\rm mass}(\lambda)
 =
 \E_\pi e^{\lambda(m-1)},
 \qquad
 \frac{d\pi_\lambda}{d\pi}
 =
 \frac{e^{\lambda(m-1)}}{Z_{\rm mass}(\lambda)}.
\end{equation}
The overlap graph of the blocks has bounded degree. Hence
\eqref{eq:mass-osc} implies that the row sums in
\zcref{lem:tilted-variance} are bounded by
\[
 C\lambda\eta^2h^d.
\]
After fixing $\eta$ sufficiently small and increasing the constant in the
hypothesis $C\nu h^d\le1/2$ if necessary, this quantity is below $1/2$
for every value of $\lambda$ used below. Applying
\zcref{lem:tilted-variance} to $G=m$ gives
\[
 \operatorname{Var}_{\pi_\lambda}(m)
 \le
 C\sum_j\operatorname{osc}_j(m)^2.
\]
There are at most $Ch^{-d}$ blocks, so
\begin{equation}\label{eq:tilted-variance}
 \operatorname{Var}_{\pi_\lambda}(m)
 \le
 C\eta^4h^d.
\end{equation}

We next choose the overall intensity scale. Let
\[
 F(a)=a\,\E_{\pi_{\nu a}}m,
 \qquad a>0.
\]
The deterministic bound $m=1+O(\eta^2)$ shows that $m$ is bounded above
and away from zero. Hence $F(a)$ is below $1$ at the reciprocal of a
sufficiently large deterministic upper bound for $m$, and above $1$ at the
reciprocal of a deterministic lower bound. Moreover,
\[
 F'(a)
 =
 \E_{\pi_{\nu a}}m
 +
 \nu a\operatorname{Var}_{\pi_{\nu a}}(m)
 >0.
\]
There is therefore a unique $a>0$ such that
\[
 a\,\E_{\pi_{\nu a}}m=1.
\]
The bound $m=1+O(\eta^2)$ also gives
\[
 a=1+O(\eta^2).
\]

Use the common tilted density-state law $\pi_{\nu a}$ under both priors,
and retain, conditional on each density state, the original conditional
laws of the local functions $F_j$. Before normalizing $\widetilde p$,
consider the auxiliary marked Poisson experiment having, conditional on the
state, intensity
\[
 \nu a\,\widetilde p(x)q_\pm(y\mid x)\,dx
\]
on $[0,1]^d\times\mathcal Y$, where
\[
 \mathcal Y=\{-a_Y,0,a_Y\}.
\]

We record its mixture density explicitly because this is where the choice
of tilt is used. Represent a finite marked configuration by an ordered
tuple and use the reference measure
\begin{equation}\label{eq:configuration-measure}
 \Lambda_{\rm conf}\big|_k
 =
 \frac1{k!}
 (dx\otimes\#_{\mathcal Y})^{\otimes k}
\end{equation}
on the $k$-point component, with unit mass on the empty configuration.
Conditional on a density state of total mass $m$, the Poisson factor is
$e^{-\nu am}$. Multiplying it by the Radon--Nikodym factor
$e^{\nu a(m-1)}/Z_{\rm mass}(\nu a)$ from the common tilt gives
\[
 e^{-\nu am}
 \frac{e^{\nu a(m-1)}}{Z_{\rm mass}(\nu a)}
 =
 \frac{e^{-\nu a}}{Z_{\rm mass}(\nu a)}.
\]
The random mass has therefore disappeared. After integrating over the
original independent local laws, the auxiliary mixture density on the
$k$-point component is
\begin{equation}\label{eq:janossy}
 \frac{e^{-\nu a}(\nu a)^k}{Z_{\rm mass}(\nu a)}
 \E_\pm
 \prod_{i=1}^k
 \widetilde p(x_i)q_\pm(y_i\mid x_i).
\end{equation}
This cancellation is the reason for using the same tilt under the two
priors.

\leavevmode\par\medskip\noindent
\emph{Step 3: Comparing the two auxiliary experiments.}
We now convert the local coefficient estimates into a Hellinger bound for
the complete auxiliary sample.

Given a finite design configuration, join two observed design points
whenever some block $Q_j$ contains both of them. If two points belong to
different connected components of this graph, no local block contains
points from both components. Since the expectation in
\eqref{eq:janossy} is taken under the original product of independent local
laws, both its numerator and its design marginal factor over these
components. Consequently, conditional on the entire observed design, the
response distribution is a product over the connected components.

Let $\mathsf D$ denote the common design law of the auxiliary experiment.
Although $\mathsf D$ need not itself be a product law, it is a Cox-process
mixture under the tilted density-state law. Its $k$th factorial moment
density is therefore
\begin{equation}\label{eq:design-factorial-density}
 \rho_{\mathsf D}^{(k)}(x_1,\ldots,x_k)
 =
 (\nu a)^k
 \E_{\pi_{\nu a}}
 \prod_{i=1}^k\widetilde p(x_i).
\end{equation}
The deterministic bounds on $a$ and $\widetilde p$ give
\[
 \rho_{\mathsf D}^{(k)}(x)\le(C\nu)^k.
\]

We first bound the conditional response discrepancy for one connected
tuple. For distinct design points
$x=(x_1,\ldots,x_k)$, define probability laws on $\mathcal Y^k$ by
\begin{equation}\label{eq:tuple-response-law}
 \mathsf Q_{\pm,k}(y\mid x)
 =
 \frac{
   \E_\pm
   \prod_{i=1}^k
   \widetilde p(x_i)q_\pm(y_i\mid x_i)
 }{\varrho_k(x)},
 \qquad
 \varrho_k(x)
 =
 \E_\pi
 \prod_{i=1}^k\widetilde p(x_i),
\end{equation}
and set
\[
 d_k(x)
 =
 H^2\bigl(
   \mathsf Q_{+,k}(\cdot\mid x),
   \mathsf Q_{-,k}(\cdot\mid x)
 \bigr).
\]
The expectation in the numerator is under the original, untilted product
of the local priors. This is exactly the law that appears after the
mass cancellation in \eqref{eq:janossy}. If the $k$ points form an entire
connected component of a larger design configuration, the factorization
just described shows that \eqref{eq:tuple-response-law} is its actual
conditional response law. For an arbitrary connected $k$-tuple contained
inside a larger component, we use the same formula only as a nonnegative
comparison quantity. Exact one-point matching gives
\[
 d_1=0.
\]

We next obtain a pointwise estimate for $d_k$. Fix a block $Q_j$ and a
label set $A$ of $q$ observations lying in that block. Let
$D_q^{(j,A)}$ denote the local compensated polynomial, written in reference
coordinates and with the substitutions
\[
 u_i\mapsto w_j(x_i)u_i,
 \qquad i\in A.
\]
Define
\begin{equation}\label{eq:local-error-function}
 \mathfrak e_{j,A}(x_A)
 =
 \sum_{a\in\mathcal A_{q,R}}
 \left|[u^a]D_q^{(j,A)}\right|^2.
\end{equation}
The hypothesis of \zcref{lem:gluing}, including its stability under the
bounded substitutions by $w_j$, and the affine rescaling from the unit
reference cube to $Q_j$ give
\begin{equation}\label{eq:local-error-integral}
 \int_{Q_j^q}
 \mathfrak e_{j,A}(x_A)\,dx_A
 \le
 C^q h^{dq}\delta^2\xi,
 \qquad
 2\le q\le M.
\end{equation}
For $q=0,1$, set $\mathfrak e_{j,A}=0$.

Suppose now that $2\le k\le M$, and put
\[
 I_j=\{i:x_i\in Q_j\}.
\]
Telescope the product of local factors in \eqref{eq:block-product}, replacing
the $+$ factor by the $-$ factor one block at a time. The constant
coefficient of every local difference vanishes because the two local priors
have the same design-density distribution. Its coefficients of odd degree
vanish because the conditional laws of $F_j$ are sign-symmetric. Hence,
after the substitution $u_i\mapsto\kappa u_i$, every local difference
contains at least a factor $\kappa^2$.

The remaining local factors have, through response degree $2R$, coefficient
$\ell^1$ norm at most $C^{|I_j|}$ by the uniform bounds on the densities
and functions. Bounded overlap gives
\[
 \sum_j|I_j|\le Ck
\]
and at most $Ck$ nonconstant factors. The coefficient convolution
inequality, followed by Cauchy--Schwarz in the telescoping sum, therefore
gives
\[
 \|\text{retained coefficient difference}\|_{\ell^2}^2
 \le
 C_R^k\kappa^4
 \sum_j\mathfrak e_{j,I_j}(x_{I_j}).
\]

We convert this coefficient estimate into a Hellinger estimate for the
ternary response cells. The common denominator in
\eqref{eq:tuple-response-law} satisfies
\[
 \varrho_k(x)\ge c^k,
\]
because $\widetilde p$ is uniformly bounded below. By our choice of $c_f$,
each conditional ternary probability is also bounded below by a fixed
positive constant, so every $k$-response cell under
$\mathsf Q_{\pm,k}$ has probability at least $c^k$. Applying
\eqref{eq:response-map}, dividing by $\varrho_k$, and using
\[
 (\sqrt{x}-\sqrt{y})^2
 \le
 \frac{(x-y)^2}{y}
\]
for each response cell changes the preceding coefficient estimate by at
most a factor $C_R^k$.

The response terms beyond degree $2R$ are controlled directly from the
bounded global function $\sum_jw_jF_j$. They vanish when $k\le R$ and,
when $k>R$, their squared contribution is at most
$C_R^k\kappa^{4R+4}$. We therefore obtain
\begin{equation}\label{eq:pointwise-component}
 d_k(x)
 \le
 C_R^k
 \left\{
   \kappa^4
   \sum_{j:|I_j|\ge2}
   \mathfrak e_{j,I_j}(x_{I_j})
   +
   \mathbf1_{\{k>R\}}\kappa^{4R+4}
 \right\},
 \qquad
 2\le k\le M.
\end{equation}

We now integrate this estimate over connected configurations. Fix a term
in the first sum in \eqref{eq:pointwise-component}, let
$q=|I_j|\ge2$, and let $A$ be its set of labels. Summing
\eqref{eq:local-error-integral} over the $O(h^{-d})$ possible blocks gives
\[
 C^q\delta^2\xi\,h^{d(q-1)}.
\]
The $q$ marked points all lie in one block and hence already form a connected
set.

If the full configuration on $k$ labels is connected, choose for every
unmarked point a parent along a path to this marked connected set. This
produces a forest rooted at the $q$ marked labels in which every child lies
within distance $Ch$ of its parent. There are at most $k^{k-q}$ such
forests. Once the marked points are fixed, integrating the remaining
$k-q$ points from the leaves toward the roots contributes at most
\[
 (Ch^d)^{k-q}.
\]
Summing over the choices of the $q$ marked labels therefore gives
\begin{align}
 &\int_{\{x:\,x\ {\rm connected}\}}
 \kappa^4
 \sum_{j:|I_j|\ge2}
 \mathfrak e_{j,I_j}(x_{I_j})\,dx
 \notag\\
 &\qquad\le
 \Delta^2\xi
 \sum_{q=2}^k
 C^k
 \binom{k}{q}
 k^{k-q}
 h^{d(q-1)}
 h^{d(k-q)}
 \notag\\
 &\qquad\le
 C^k k! h^{d(k-1)}\Delta^2\xi,
 \label{eq:forest-error-accounting}
\end{align}
where we used $\Delta^2=\kappa^4\delta^2$. For the last inequality,
\[
 \binom{k}{q}k^{k-q}
 =
 \frac{k!}{q!}\frac{k^{k-q}}{(k-q)!}
 \le
 e^k\frac{k!}{q!},
\]
and $\sum_{q\ge2}1/q!<\infty$.

The same geometric argument without marked points gives
\[
 \bigl|
 \{x\in([0,1]^d)^k:x\text{ is connected}\}
 \bigr|
 \le
 C^k k!h^{d(k-1)}.
\]
Indeed, sum over spanning trees, choose one root freely, and integrate each
remaining vertex in a ball of volume $Ch^d$ around its parent. Combining
this volume bound with \eqref{eq:pointwise-component} yields
\begin{equation}\label{eq:component-bound}
 \int_{\{x:\,x\ {\rm connected}\}}d_k(x)\,dx
 \le
 C_R^k k!h^{d(k-1)}
 \left\{
   \Delta^2\xi
   +
   \mathbf1_{\{k>R\}}\kappa^{4R+4}
 \right\},
 \qquad
 2\le k\le M.
\end{equation}

For $k>M$, we no longer use the local moment-matching estimate. The common
density marginal cancels the constant term in the difference of the two
response probabilities, while sign symmetry cancels all terms of odd degree.
Uniform boundedness therefore gives the crude pointwise bound
\[
 |\text{response-probability difference}|
 \le C^k\kappa^2,
\]
and hence a squared bound $C^k\kappa^4$. Combining this with the connected
volume estimate above gives
\[
 \int_{\{x:\,x\ {\rm connected}\}}d_k(x)\,dx
 \le
 C^k k!h^{d(k-1)}\kappa^4,
 \qquad k>M.
\]

It remains to sum these component estimates over the random auxiliary
design. Let $\mathcal C(X)$ denote the set of connected components of the
observed design graph, and let $X_C$ be the tuple of design points in a
component $C$. Conditional on the design, the response laws under the two
signs are products over the components. Hellinger distance for product
measures therefore gives
\[
 H^2\bigl(
   P_+(\,\cdot\,\mid X),
   P_-(\,\cdot\,\mid X)
 \bigr)
 \le
 \sum_{C\in\mathcal C(X)}d_{|C|}(X_C).
\]
Averaging over the common design law and then enlarging the sum from actual
components to all connected subsets gives
\begin{align}
 H^2(\overline P_+^{\rm aux},\overline P_-^{\rm aux})
 &\le
 \E_{\mathsf D}
 \sum_{C\in\mathcal C(X)}d_{|C|}(X_C)
 \notag\\
 &\le
 \sum_{k\ge1}\frac1{k!}
 \int_{\{x:\,x\ {\rm connected}\}}
 d_k(x)\rho_{\mathsf D}^{(k)}(x)\,dx
 \notag\\
 &\le
 \sum_{k\ge1}
 \frac{(C\nu)^k}{k!}
 \int_{\{x:\,x\ {\rm connected}\}}
 d_k(x)\,dx.
 \label{eq:component-factorial-bridge}
\end{align}
The factor $1/k!$ converts ordered factorial tuples into unordered subsets.
The last inequality uses \eqref{eq:design-factorial-density}.

Set
\[
 \tau=\nu h^d.
\]
The coefficient-error contribution starts at $k=2$, since $d_1=0$. Using
\eqref{eq:component-bound} and $C\tau\le1/2$,
\[
 \sum_{k=2}^M
 C^k\nu^k h^{d(k-1)}\Delta^2\xi
 \le
 C\nu^2h^d\Delta^2\xi.
\]
The response-remainder contribution begins at $k=R+1$ and is bounded by
\[
 C\nu\kappa^{4R+4}
 \sum_{k=R+1}^{M}(C\tau)^{k-1}
 \le
 C\nu\kappa^{4R+4}(C\tau)^R.
\]
This sum is empty when $R\ge M$. Finally, the contribution from
components larger than $M$ is
\[
 C\nu\kappa^4
 \sum_{k=M+1}^\infty(C\tau)^{k-1}
 \le
 C\nu\kappa^4(C\tau)^M.
\]
Consequently,
\begin{equation}\label{eq:aux-bound}
 \begin{aligned}
 H^2(\overline P_+^{\rm aux},\overline P_-^{\rm aux})
 \le
 C\bigl\{
 &\nu^2h^d\Delta^2\xi
 +\nu\kappa^{4R+4}(C\nu h^d)^R\\
 &+\nu\kappa^4(C\nu h^d)^M
 \bigr\}.
 \end{aligned}
\end{equation}

\leavevmode\par\medskip\noindent
\emph{Step 4: Normalizing the design density.}
The auxiliary experiment uses the unnormalized density factor
$\widetilde p$ and conditional intensity scale $\nu a$. We now replace it,
under the same common tilted state law, by the normalized density
\[
 p(x)=\frac{\widetilde p(x)}{m}
\]
and the marked Poisson intensity
\[
 \nu p(x)q_\pm(y\mid x)\,dx.
\]
Because $\widetilde p$ is uniformly bounded above and below and
$m=1+O(\eta^2)$, choosing the fixed constant $\eta$ sufficiently small
ensures that every resulting $p$ lies between the prescribed lower and
upper bounds of the model class. The regression function, conditional
variance, and conditional response law are unchanged, so the smoothness and
fourth-moment bounds verified in Step~1 remain valid.

For a fixed density and function state, the squared $L^2$ distance between
the square roots of the auxiliary and normalized marked intensities is
\begin{align*}
 &\int_{[0,1]^d}
 \sum_{y\in\mathcal Y}
 \left[
   \sqrt{\nu a\,\widetilde p(x)q_\pm(y\mid x)}
   -
   \sqrt{\nu\,\frac{\widetilde p(x)}m q_\pm(y\mid x)}
 \right]^2dx\\
 &\qquad=
 \nu\bigl(\sqrt{am}-1\bigr)^2.
\end{align*}
Indeed, $\sum_yq_\pm(y\mid x)=1$ and
$\int\widetilde p=m$. Since
\[
 a\,\E_{\pi_{\nu a}}m=1,
\]
we have
\[
 \bigl(\sqrt{am}-1\bigr)^2
 \le
 (am-1)^2
 =
 a^2\bigl(m-\E_{\pi_{\nu a}}m\bigr)^2.
\]
The Hellinger formula for Poisson processes and convexity of squared
Hellinger distance under mixing therefore give, for either sign,
\[
 H^2(
   \overline P_\pm^{\rm aux},
   \overline P_\pm^{\rm Pois}
 )
 \le
 C\nu a^2
 \operatorname{Var}_{\pi_{\nu a}}(m).
\]
Using \eqref{eq:tilted-variance} and the deterministic bound on $a$,
\begin{equation}\label{eq:normalization-cost}
 H^2(
   \overline P_\pm^{\rm aux},
   \overline P_\pm^{\rm Pois}
 )
 \le
 C\nu h^d\eta^4.
\end{equation}

Finally, apply the triangle inequality for Hellinger distance to
\[
 \overline P_+^{\rm Pois},
 \qquad
 \overline P_+^{\rm aux},
 \qquad
 \overline P_-^{\rm aux},
 \qquad
 \overline P_-^{\rm Pois},
\]
and use $(a+b+c)^2\le3(a^2+b^2+c^2)$. Combining
\eqref{eq:aux-bound} and \eqref{eq:normalization-cost}, and recalling
$\Delta^2=\kappa^4\delta^2$, gives
\[
 \begin{aligned}
 H^2(
   \overline P_+^{\rm Pois},
   \overline P_-^{\rm Pois}
 )
 \le
 C\bigl\{
 &\nu^2h^d\kappa^4\delta^2\xi
 +\nu\kappa^{4R+4}(C\nu h^d)^R +\nu\kappa^4(C\nu h^d)^M
 +\nu h^d\eta^4
 \bigr\},
 \end{aligned}
\]
which is \eqref{eq:global-budget}.

Every reweighting and normalization above is performed at the level of the
prior before the sample is drawn. In particular, the exponential tilt is
the same under the two priors and depends only on the latent density state,
not on the realized Poisson count or on any observed data. This completes
the proof.
\end{proof}

\appendix

\section{Elementary interpolation at a fixed count}\label{sec:simple-geometry}

This section constructs the interpolation matrices used in
\zcref{lem:local}. Recall that their separated representation cost is
defined in \eqref{eq:coefficient-kernel-cost}. For a configuration of
$m$ design points, we will construct a positive semidefinite matrix whose
evaluation in the smooth family $\phi$ is diagonal. The diagonal entry at
a point is either retained or reduced according to the distances from that
point to the other design points. Retaining more configurations improves
this diagonal approximation but increases the coefficients needed to
represent the matrix. At fixed $m$, the logarithmic losses resulting from
this tradeoff are sufficient for the fixed-order lower bound.

Write
\[
 L_N=1+\log N.
\]

\begin{lemma}\label{lem:simple-geometry}
Fix a bounded cube $Q\subset\mathbb R^d$ and an integer $M\ge2$.
There is a finite family $\phi$ of smooth real-valued functions, depending
on $Q$ and $M$ but not on $m$ or $N$, with the following property. For
every $2\le m\le M$ and $N\ge2$, there is a positive semidefinite
coefficient matrix
\[
 E_{m,N}(\mathbf U),
 \qquad
 \mathbf U=(U_1,\ldots,U_m)\in Q^m,
\]
which is invariant under relabeling the design points and satisfies
\begin{align}
 \phi(U_i)^T E_{m,N}(\mathbf U)\phi(U_j)
 &=
 w_i(\mathbf U)\mathbf1_{\{i=j\}},
 \qquad
 0\le w_i(\mathbf U)\le1,
 \label{eq:simple-diag}\\
 \int_{Q^m}
 \sum_{i=1}^m\bigl(1-w_i(\mathbf U)\bigr)\,d\mathbf U
 &\le
 C_MN^{-d}L_N^{m-2},
 \label{eq:simple-tail}\\
 \operatorname{cost}(E_{m,N})
 &\le
 C_MN^2L_N^{m-2}.
 \label{eq:simple-cost}
\end{align}
Here the last quantity is the separated coefficient-kernel cost from
\eqref{eq:coefficient-kernel-cost}.

The family $\phi$ contains a function equal to one on $Q$ and has smooth
compactly supported extensions to $\mathbb R^d$. In particular, for every
fixed open neighborhood $\mathcal O\supset Q$ and every $s>0$,
\[
 \sum_a\|\phi_a\|_{C^s(\mathcal O)}
 \le C_{M,s}.
\]
The construction depends only on the geometry of the design points and
uses no smoothness assumption on the design density.
\end{lemma}

\begin{proof}
A fixed affine change of variables changes only constants depending on
$Q$, so we may assume
\[
 \operatorname{diam}(Q)\le1.
\]
We first divide the possible distances between two design points into
dyadic shells. Choose a nonincreasing smooth radial function
$\psi:\mathbb R^d\to[0,1]$ that equals one on the unit ball and zero
outside the ball of radius two, and define
\[
 \rho_j(v)
 =
 \psi(2^jv)-\psi(2^{j+1}v),
 \qquad
 j=0,1,\ldots.
\]
For every $0<|v|\le1$,
\[
 \rho_j(v)\ge0,
 \qquad
 \sum_{j\ge0}\rho_j(v)=1.
\]
Moreover,
\[
 \rho_j(v)\ne0
 \quad\Longrightarrow\quad
 2^{-j-1}\le |v|\le2^{1-j}.
\]
Thus the index $j$ records the distance scale $|v|\asymp2^{-j}$.

We next construct a linear polynomial that separates one design point from
one neighbor. For $x\ne y$, define
\[
 L_{x,y}(t)
 =
 1-\frac{(y-x)\cdot(t-x)}{|y-x|^2}.
\]
Then
\[
 L_{x,y}(x)=1,
 \qquad
 L_{x,y}(y)=0.
\]
If $\rho_j(y-x)\ne0$, then $|y-x|\asymp2^{-j}$. Writing a polynomial in
$t$ in the usual monomial basis and denoting by $\|\cdot\|_{\ell^1}$
the sum of the absolute values of its coefficients, we therefore have
\[
 \|L_{x,y}\|_{\ell^1}\le D2^j
\]
for a constant $D$ depending only on $Q$ and $d$.

We also need a separated representation of the coefficient matrix of
\[
 \rho_j(y-x)L_{x,y}(t)L_{x,y}(t').
\]
We claim that
\begin{equation}\label{eq:simple-unary}
 \operatorname{cost}
 \left\{
   \rho_j(y-x)L_{x,y}(t)L_{x,y}(t')
 \right\}
 \le C2^{2j}.
\end{equation}
Here the variables $x$ and $y$ are treated as two separate design
arguments in the sense of \eqref{eq:coefficient-kernel-cost}.

To prove \eqref{eq:simple-unary}, set
\[
 h=2^{-j},
 \qquad
 v=\frac{y-x}{h}.
\]
On the support of $\rho_j(y-x)$, the vector $v$ belongs to the fixed
annulus
\[
 \frac12\le |v|\le2.
\]
Moreover,
\[
 hL_{x,y}(t)
 =
 h+\frac{v\cdot x}{|v|^2}
   -\frac{v\cdot t}{|v|^2}.
\]
Hence, after multiplying by
\[
 \rho_j(y-x)=\psi(v)-\psi(2v)
\]
and by a fixed smooth cutoff in $x$ that equals one on $Q$, every
coefficient of
\[
 h^2\rho_j(y-x)L_{x,y}(t)L_{x,y}(t')
\]
is a smooth compactly supported function of $(x,v)$ with derivative bounds
independent of $j$.

Its Fourier transform in $(x,v)$ therefore has uniformly bounded
$L^1$ norm: integrate by parts more than $2d$ times. Fourier inversion
gives a representation by characters
\[
 e^{i(\xi\cdot x+\zeta\cdot v)}
 =
 e^{i(\xi-\zeta/h)\cdot x}
 e^{i(\zeta/h)\cdot y}.
\]
The first factor depends only on the design variable $x$ and the second
only on $y$, and both have absolute value one. Thus Fourier inversion is
a separated representation of the required type. Passing to real and
imaginary parts changes the cost only by a fixed constant. Finally,
undoing the factor $h^2$ contributes
\[
 h^{-2}=2^{2j},
\]
which proves \eqref{eq:simple-unary}.

We now construct a cardinal polynomial for one point in an
$m$-point configuration. Fix a label $i$ and put
\[
 q=m-1.
\]
For each neighbor $r\ne i$, choose a shell index $j_r\ge0$ and write
\[
 J=\sum_{r\ne i}j_r.
\]
Define
\[
 W_{i,\mathbf j}(\mathbf U)
 =
 \prod_{r\ne i}
 \rho_{j_r}(U_r-U_i)
\]
and
\[
 P_i(t)
 =
 \prod_{r\ne i}L_{U_i,U_r}(t).
\]
Since each factor in $P_i$ equals one at $U_i$ and vanishes at its
corresponding neighbor,
\[
 P_i(U_r)=\mathbf1_{\{i=r\}},
 \qquad
 r=1,\ldots,m.
\]
Thus $P_i$ is a cardinal polynomial for the point $U_i$.

Let $p_i(\mathbf U)$ denote the coefficient vector of $P_i$ in the
monomial basis. On the region where
$W_{i,\mathbf j}(\mathbf U)\ne0$, the preceding coefficient bound gives
\[
 \|p_i(\mathbf U)\|_{\ell^1}
 \le
 D^q2^J.
\]
Multiplying the separated representations from
\eqref{eq:simple-unary}, one for each neighbor of $U_i$, gives
\begin{equation}\label{eq:simple-block}
 \operatorname{cost}
 \left\{
   W_{i,\mathbf j}(\mathbf U)
   p_i(\mathbf U)p_i(\mathbf U)^T
 \right\}
 \le
 C^q4^J.
\end{equation}
The shared center $U_i$ causes no difficulty: the factors depending on
$U_i$ in the individual separated representations all have absolute value
at most one, so their product does as well. For the polynomial
coefficients we use the elementary submultiplicative bound
\[
 \|PQ\|_{\ell^1}
 \le
 \|P\|_{\ell^1}\|Q\|_{\ell^1}.
\]
Thus \eqref{eq:simple-block} is a separated representation in all $m$
design variables.

All cardinal polynomials arising in this construction have total degree at
most $m-1\le M-1$. We can therefore place their coefficient vectors in one
fixed finite smooth family. Choose a smooth compactly supported cutoff
$\chi$ that equals one on a neighborhood of $Q$, and let
\[
 \phi_\beta(t)=\chi(t)t^\beta,
 \qquad
 |\beta|\le M-1.
\]
On $Q$, these functions agree with the corresponding monomials, so the
coefficient vectors $p_i$ and all evaluations of the cardinal polynomials
are unchanged. The family also contains $\phi_0=\chi$, which equals one
on $Q$.

We now decide which dyadic configurations to retain. Put
\[
 K
 =
 \left\lfloor
 \log_2\frac{N}{mD^q}
 \right\rfloor
\]
and define
\[
 E_{m,N}(\mathbf U)
 =
 \sum_{i=1}^m
 \sum_{\mathbf j:\,J\le K}
 W_{i,\mathbf j}(\mathbf U)
 p_i(\mathbf U)p_i(\mathbf U)^T,
\]
where the inner sum is interpreted as zero if $K<0$.

Each summand is a nonnegative scalar times a rank-one positive
semidefinite matrix, so $E_{m,N}$ is positive semidefinite. Evaluating at
two design points and using the cardinal property gives
\[
 \phi(U_i)^TE_{m,N}(\mathbf U)\phi(U_\ell)
 =
 w_i(\mathbf U)\mathbf1_{\{i=\ell\}},
\]
where
\[
 w_i(\mathbf U)
 =
 \sum_{\mathbf j:\,J\le K}
 W_{i,\mathbf j}(\mathbf U).
\]
Because the functions $\rho_j$ form a partition of unity for every
nonzero neighbor difference,
\[
 0\le w_i(\mathbf U)\le1.
\]
At configurations with coincident design points, we retain the same
localized summands, interpreting each coefficient kernel
$W_{i,\mathbf j}p_i p_i^T$ by the smooth extension supplied by the
Fourier representation in \eqref{eq:simple-block}. If a neighbor
coincides with its center, the corresponding shell cutoff vanishes in a
neighborhood of that coincidence, so the summand extends there by zero.
Coincidences between neighbors away from the center retain their ordinary
values. Define $w_i$ by the same retained partition-weight sum above at
every configuration. A repeated location has zero weight when used as a
center, while every retained cardinal polynomial centered elsewhere
vanishes at that location. Hence \eqref{eq:simple-diag} holds pointwise
also at coincident configurations.

The construction is also invariant under relabeling the design points.
For a fixed center, the product
$W_{i,\mathbf j}$ and the condition
$\sum_{r\ne i}j_r\le K$ are symmetric in its neighbors. Summing over all
choices of the center therefore makes the full matrix invariant under every
permutation of the $m$ labels.

It remains to bound the representation cost and the amount of diagonal
mass that was discarded by the cutoff $J\le K$. There are
\[
 \binom{J+q-1}{q-1}
\]
$q$-tuples of nonnegative shell indices with total $J$. For fixed $q$ and
$K\ge0$,
\begin{align}
 \sum_{J=0}^K
 4^J\binom{J+q-1}{q-1}
 &\le
 C_q4^K(1+K)^{q-1},
 \label{eq:simple-sum-up}\\
 \sum_{J>K}
 2^{-dJ}\binom{J+q-1}{q-1}
 &\le
 C_{q,d}2^{-dK}(1+K)^{q-1}.
 \label{eq:simple-sum-down}
\end{align}
For \eqref{eq:simple-sum-up}, use
\[
 \binom{J+q-1}{q-1}
 \le
 C_q(1+K)^{q-1},
 \qquad
 J\le K,
\]
and sum the geometric series in $4^J$. For
\eqref{eq:simple-sum-down}, write $J=K+r$ and use
\[
 (1+K+r)^{q-1}
 \le
 (1+K)^{q-1}(1+r)^{q-1};
\]
the remaining series
$\sum_{r\ge1}2^{-dr}(1+r)^{q-1}$ is finite.

We first prove the cost bound. Combining
\eqref{eq:simple-block} with \eqref{eq:simple-sum-up}, and then summing
over the $m$ possible centers, gives
\[
 \operatorname{cost}(E_{m,N})
 \le
 C_M4^K(1+K)^{m-2}.
\]
When $K\ge0$, the definition of $K$ and the fact that
$mD^{m-1}$ is bounded in terms of $M$ give
\[
 2^K\asymp_M N,
 \qquad
 1+K\le C_M(1+\log N)=C_ML_N.
\]
Therefore
\[
 \operatorname{cost}(E_{m,N})
 \le
 C_MN^2L_N^{m-2},
\]
which is \eqref{eq:simple-cost}.

We next bound the integrated diagonal defect. For almost every
configuration, the full shell partition gives
\[
 1-w_i(\mathbf U)
 =
 \sum_{\mathbf j:\,J>K}
 W_{i,\mathbf j}(\mathbf U).
\]
Fixing the center $U_i$, the neighbor variables occur separately in
$W_{i,\mathbf j}$. Since the support of $\rho_j$ has volume at most
$C2^{-dj}$ and $0\le\rho_j\le1$,
\[
 \int_Q
 \rho_j(U_r-U_i)\,dU_r
 \le
 C2^{-dj}.
\]
Hence
\[
 \int_{Q^m}
 \sum_{i=1}^m
 \bigl(1-w_i(\mathbf U)\bigr)\,d\mathbf U
 \le
 C_M
 \sum_{J>K}
 2^{-dJ}
 \binom{J+q-1}{q-1}.
\]
Using \eqref{eq:simple-sum-down} and again
$2^K\asymp_MN$ gives
\[
 \int_{Q^m}
 \sum_{i=1}^m
 \bigl(1-w_i(\mathbf U)\bigr)\,d\mathbf U
 \le
 C_MN^{-d}L_N^{m-2},
\]
which is \eqref{eq:simple-tail}.

If $K<0$, then
\[
 N<mD^q,
\]
so, because $m\le M$, the value of $N$ is bounded by a constant depending
only on $M$. Taking $E_{m,N}=0$ and $w_i=0$ then satisfies
\eqref{eq:simple-diag}--\eqref{eq:simple-cost} after increasing $C_M$.

Finally, the family $\phi$ is finite and consists of smooth functions with
one fixed compact support. Hence, for every prescribed neighborhood
$\mathcal O\supset Q$ and every $s>0$,
\[
 \sum_a\|\phi_a\|_{C^s(\mathcal O)}
 \le
 C_{M,s}.
\]
For integer $s$, this uses the convention
$C^s=C^{s-1,1}$ adopted for the statistical model. This completes the
proof.
\end{proof}

\section{The interpolation estimate}\label{sec:geometry}

This appendix proves the sharper interpolation estimate needed only for the
explicit lower bound in \zcref{thm:bracket}. As in
\zcref{sec:simple-geometry}, the goal is to construct a positive
semidefinite coefficient matrix whose evaluation at the design points is
diagonal, while controlling both the integrated diagonal defect and the
separated representation cost from
\eqref{eq:coefficient-kernel-cost}. The fixed-count construction in
\zcref{sec:simple-geometry} incurs logarithmic factors that are harmless
when the observation count is fixed. Here the observation count $m$ will
grow, so we need bounds of the form $C^mN^{-d}$ and $C^mN^2$, with no
additional logarithmic loss. The construction below achieves this by
building three-point cardinal polynomials and combining them recursively in
a balanced way, so that the coefficient growth remains exponential in $m$.

\begin{lemma}\label{lem:geometry}
Let $d\ge3$. There exist a countable family
\[
 \phi=(\phi_a)_{a\in\mathcal I}
\]
of real-valued smooth functions and an open neighborhood
$\mathcal O\supset[0,1]^d$ with the following properties. The family
satisfies \eqref{eq:frame-assumption},
\[
 \sum_{a\in\mathcal I}\|\phi_a\|_{C^s(\mathcal O)}<\infty,
\]
and contains a function equal to one on $[0,1]^d$.

For every pair of integers $m,N\ge2$, there is a positive semidefinite
coefficient matrix
\[
 E_{m,N}(\mathbf U),
 \qquad
 \mathbf U=(U_1,\ldots,U_m)\in[0,1]^{dm},
\]
which is invariant under permutations of the $m$ design labels and
satisfies
\begin{align}
 \phi(U_i)^T E_{m,N}(\mathbf U)\phi(U_j)
 &=
 w_i(\mathbf U)\mathbf1_{\{i=j\}},
 \qquad
 0\le w_i(\mathbf U)\le1,
 \label{eq:cardinal}\\
 \int_{([0,1]^d)^m}
 \sum_{i=1}^m\bigl(1-w_i(\mathbf U)\bigr)\,d\mathbf U
 &\le
 C^mN^{-d},
 \label{eq:geometry-error}\\
 \operatorname{cost}(E_{m,N})
 &\le
 C^mN^2.
 \label{eq:geometry-cost}
\end{align}
Here $\operatorname{cost}(E_{m,N})$ is the separated
coefficient-kernel cost defined in
\eqref{eq:coefficient-kernel-cost}, and the constant $C<\infty$ is
independent of $m$ and $N$.

For each fixed response-degree cutoff $R$, the family $\phi$ may also be
enlarged by the Fourier functions used in
\eqref{eq:multiplier-cost-bound}. After this enlargement it still satisfies
\eqref{eq:frame-assumption} and
\[
 \sum_a\|\phi_a\|_{C^s(\mathcal O)}<\infty,
\]
with constants that may depend on $R$ but remain independent of $m$ and
$N$.
\end{lemma}

\begin{proof}
A fixed affine rescaling changes only constants depending on the original
cube, so we may assume throughout that
\[
 \operatorname{diam}(Q)\le1.
\]
The construction below recursively multiplies cardinal polynomials. At each
internal step we will introduce an extra factor $B^{-\gamma}$ in the cutoff,
where $B\le1$ is the larger of the two selected radial scales. This makes the
sum of the retained representation costs more convergent, at the expense of
discarding more configurations. If $B=2^{-L}$, the relevant series will
decay respectively like
\[
 2^{-(2\gamma-1)L}
 \qquad\text{and}\qquad
 2^{-\{d(1-\gamma)-1\}L}.
\]
We therefore fix
\[
 \frac12<\gamma<1-\frac1d.
\]
Such a choice is possible exactly when $d>2$; for example,
$\gamma=3/5$ works for every integer $d\ge3$. All constants below may
depend on this fixed $\gamma$. We first construct and estimate the
cardinal polynomials in ordinary polynomial coordinates, and only at the
end place all of them in one fixed smooth frame.

\medskip\noindent\emph{Step 1: A three-point cardinal polynomial.}
We begin with a center and two neighbors. After translating the center to
the origin, write the neighbors as
\[
 ae,\qquad bz,
\]
where
\[
 |e|=|z|=1,\qquad a\le Cb.
\]
Set
\[
 c_\angle=e\cdot z,
 \qquad
 \sigma^2=1-c_\angle^2.
\]
Thus $\sigma$ measures the angle between the two radial directions. Define
\[
 P(t)
 =
 \frac{
 b^2(1-e\cdot t/a)(1-z\cdot t/b)
 +\sigma^2
 -\left\{
   \frac{e-c_\angle z}{a}
   +\frac{z-c_\angle e}{b}
  \right\}\cdot t
 }{b^2+\sigma^2}.
\]
Direct substitution gives
\[
 P(0)=1,\qquad P(ae)=P(bz)=0.
\]
Hence $P$ is cardinal for the center relative to these two neighbors.
Moreover,
\[
 |e-c_\angle z|
 =
 |z-c_\angle e|
 =
 \sigma.
\]
Expanding the constant, linear, and quadratic coefficients of $P$ therefore
gives
\[
 \|P\|_{\ell^1}
 \le
 \frac{C}{a\sqrt{b^2+\sigma^2}},
\]
where $\|\cdot\|_{\ell^1}$ denotes the sum of the absolute values of the
ordinary polynomial coefficients. An important feature of this estimate is
that its denominator depends only on the two radial scales and their angle,
not on the distance between the two neighbors themselves.

We next localize these quantities smoothly. Use nonnegative dyadic
partitions that restrict
\[
 a\asymp A,\qquad
 b\asymp B,
\]
and an angular partition that restricts
\[
 \max\{B,\sigma\}\asymp\Theta,
 \qquad
 \Theta\ge B.
\]
For one such choice of radial and angular scales, the coefficient matrix of
the entire product $P(t)P(t')$, multiplied by the corresponding partition
functions, has a separated representation of cost
\begin{equation}\label{eq:packet}
 C(A\Theta)^{-2}\Theta^{-1}.
\end{equation}
We now justify this estimate.

Let $U_0$ be the center and $U_1,U_2$ the two neighbors, and set
\[
 x=U_0,\qquad
 v=\frac{U_1-U_0}{A},
 \qquad
 z'=\frac{U_2-U_0}{B}.
\]
The radial partitions restrict $|v|$ and $|z'|$ to a fixed compact
annulus. To obtain a separated representation in the nearly collinear
regime, define
\[
 \alpha=\frac{v\cdot z'}{|v|^2}.
\]
Partition the bounded range of $\alpha$ into intervals of width $\Theta$,
with centers $\alpha_0$. At most $C\Theta^{-1}$ such intervals meet the
radial support. On one interval, set
\[
 u=\frac{z'-\alpha_0v}{\Theta}.
\]
Also write
\[
 r_v=|v|,
 \qquad
 r_z=|z'|,
 \qquad
 e=\frac{v}{r_v},
 \qquad
 z=\frac{z'}{r_z},
 \qquad
 u_\perp=u-(u\cdot e)e.
\]
Projecting the identity
\[
 z'=\alpha_0v+\Theta u
\]
onto the direction $e$ and its orthogonal complement gives
\begin{equation}\label{eq:angular-cancellation}
 \begin{aligned}
 z-c_\angle e
 &=\frac{\Theta}{r_z}u_\perp,
 &
 \frac{e-c_\angle z}{\Theta}
 &=
 \frac{\Theta|u_\perp|^2}{r_z^2}e
 -\frac{c_\angle}{r_z}u_\perp,\\
 \frac{z-c_\angle e}{\Theta}
 &=\frac{u_\perp}{r_z},
 &
 D:=
 \frac{b^2+\sigma^2}{\Theta^2}
 &=
 \left(\frac B\Theta\right)^2r_z^2
 +\frac{|u_\perp|^2}{r_z^2}.
 \end{aligned}
\end{equation}
In particular,
\[
 \sigma^2
 =
 \frac{\Theta^2|u_\perp|^2}{r_z^2}.
\]

The interval cutoff gives
\[
 \left|\frac{\alpha-\alpha_0}{\Theta}\right|\le C,
\]
while the angular cutoff gives $\sigma\le C\Theta$. Hence $|u|\le C$.
Multiply the partition functions by fixed smooth compactly supported
cutoffs in $x,v,u$ that equal one on the relevant supports, and by radial
cutoffs supported where
\[
 r_v,r_z\ge c_1>0.
\]
The radial cutoffs depend smoothly on
$z'=\alpha_0v+\Theta u$ with derivative bounds independent of the scales.
The angular partition may be written using smooth functions of
\[
 \frac{\sigma^2}{\Theta^2}
 =
 \frac{|u_\perp|^2}{r_z^2},
\]
and the interval partition using smooth functions of
\[
 \frac{\alpha-\alpha_0}{\Theta}
 =
 \frac{v\cdot u}{|v|^2}.
\]
Thus all cutoffs have fixed compact support and uniformly bounded
derivatives in $(x,v,u)$, independently of $A,B,\Theta$, and $\alpha_0$.

On a genuine angular annulus, $\sigma/\Theta$ is bounded below by a fixed
positive constant. In the final angular bin, $\Theta\asymp B$. In either
case, \eqref{eq:angular-cancellation} implies
\[
 D\ge c>0.
\]
After normalizing the cardinal polynomial, a direct expansion gives
\begin{equation}\label{eq:normalized-cardinal-coefficients}
 \begin{aligned}
 A\Theta P(t)
 &=
 A\Theta
 -D^{-1}\mathcal V\cdot t
 +D^{-1}\frac B\Theta\frac{r_z}{r_v}
      (e\cdot t)(z\cdot t),\\
 \mathcal V
 &=
 \frac{B^2r_z^2}{\Theta r_v}e
 +\frac{ABr_z}{\Theta}z
 +\frac1{r_v}\frac{e-c_\angle z}{\Theta}
 +\frac{A}{Br_z}\frac{z-c_\angle e}{\Theta}.
 \end{aligned}
\end{equation}
All scalar ratios appearing here are uniformly bounded:
\[
 \frac AB\le C,
 \qquad
 \frac B\Theta\le1,
 \qquad
 A,B,\Theta\le C.
\]
Together with \eqref{eq:angular-cancellation} and the lower bounds on
$r_v,r_z$, and $D$, this shows that every coefficient of
$A\Theta P(t)$ has derivatives of every fixed order that are uniformly
bounded in $(x,v,u)$. The same is therefore true for every coefficient of
\[
 (A\Theta)^2P(t)P(t').
\]
In particular, no negative power of the angular scale remains after the
normalization.

Let $f_\mu(x,v,u)$ denote any coefficient of this normalized covariance
polynomial after multiplication by the cutoffs for one interval. For every
fixed integer $J$,
\begin{equation}\label{eq:uniform-fourier-derivatives}
 \sum_{|\alpha|\le2J}
 \|\partial^\alpha f_\mu\|_{L^1(\mathbb R^{3d})}
 \le C_J,
 \qquad
 \|\widehat f_\mu\|_{L^1(\mathbb R^{3d})}
 \le C_J
 \quad\text{if }2J>3d.
\end{equation}
For the second estimate, apply $(1-\Delta)^J$ before Fourier inversion
and use the integrability of
$(1+|\xi|^2)^{-J}$.

For fixed $A,B,\Theta$, and $\alpha_0$, the variables $(x,v,u)$ are
linear combinations of the three design points. Hence every Fourier
character separates into a product of three functions, one of each design
point, all with absolute value one. Fourier inversion therefore gives a
separated representation. There are only finitely many polynomial
coefficients, and passing from complex characters to real and imaginary
parts changes the cost only by a fixed factor. Restoring the normalization
$(A\Theta)^{-2}$ and summing over the at most $C\Theta^{-1}$ interval
centers proves \eqref{eq:packet}.

\medskip\noindent\emph{Step 2: Balanced recursive multiplication.}
We now construct a cardinal polynomial for one center and an arbitrary
number of neighbors. Fix a center $x$ and $q$ labeled neighbors. First
assign a radial dyadic shell to every neighbor, using the same nonnegative
partition throughout the recursion. Among these assigned shells, choose the
two smallest radial scales, breaking ties by the labels. Denote these two
scales by
\[
 A\le CB,
 \qquad
 B=2^{-L}.
\]
The three-point construction in Step~1 gives a polynomial that equals one at
the center and zero at these two selected neighbors.

Split the remaining $q-2$ neighbors, using only their labels, into two
groups of sizes
\[
 q_1+q_2=q-2,
 \qquad
 |q_1-q_2|\le1.
\]
Multiply the three-point polynomial by recursively constructed cardinal
polynomials for these two groups, all with the same center $x$. The recursion
therefore remains balanced: each child contains at most about half of the
remaining neighbors.

We attach to every recursively constructed polynomial a dyadic number
$\tau$ that bounds its ordinary coefficient $\ell^1$ norm. For no neighbors,
use the constant polynomial $1$ and set
\[
 \tau_0=1.
\]
For one neighbor $y$ at assigned radial scale $A$, use
\[
 L_{x,y}(t)
 =
 1-\frac{(y-x)\cdot(t-x)}{|y-x|^2},
 \qquad
 L_{x,y}(x)=1,
 \qquad
 L_{x,y}(y)=0,
\]
and choose $\tau$ to be a dyadic number at least $C/A$.

At an internal node, let $\tau_1,\tau_2$ be the cutoffs of the two child
polynomials and let $\Theta$ be the angular scale of the selected
three-point polynomial. Define, with upward dyadic rounding,
\begin{equation}\label{eq:recursive-cutoff}
 \tau
 =
 C(A\Theta)^{-1}B^{-\gamma}\tau_1\tau_2.
\end{equation}
The three-point coefficient bound is $C(A\Theta)^{-1}$, while polynomial
multiplication is submultiplicative in coefficient $\ell^1$ norm. Since
$B^{-\gamma}\ge1$, the recursively constructed polynomial has coefficient
norm at most $\tau$. It is cardinal at the center and at every neighbor,
and its total degree is at most $q$.

All shell, angular, and interval partitions used in this recursion are
nonnegative and sum to one. The selection of the two largest shell
\emph{indices} (equivalently, the two smallest radial scales) and the subsequent split by labels are determined by the
discrete partition indices, not by an additional discontinuous ordering of
the actual design locations. Thus discarding selected partition indices
later simply removes nonnegative terms from one smooth partition.

We now formulate the two estimates needed from the recursion. A partition
index
\[
 \beta\in\mathcal B_q
\]
records every radial shell choice and, at every internal node, the angular
bin and interval center used in Step~1. For such an index, let
\[
 W_\beta(x,\mathbf U)\ge0
\]
be the product of all its smooth partition functions,
let
\[
 p_\beta(x,\mathbf U)
\]
be the ordinary coefficient vector, in powers of $t-x$, of its cardinal
polynomial, and let $\tau_\beta$ be the dyadic cutoff obtained from
\eqref{eq:recursive-cutoff}. Finally, let $c_\beta$ be the cost of the
specific separated representation of
\[
 W_\beta\,p_\beta p_\beta^T
\]
obtained by multiplying the representations from Step~1. The center $x$ is
included as one of the design variables in this representation.

Away from coincident configurations,
\[
 \sum_{\beta\in\mathcal B_q}W_\beta(x,\mathbf U)=1,
 \qquad
 \|p_\beta(x,\mathbf U)\|_{\ell^1}\le\tau_\beta
 \quad\text{whenever }W_\beta(x,\mathbf U)>0.
\]
We claim that there are constants $K_q\ge1$ such that, for every $T>0$,
\begin{equation}\label{eq:raw-block-induction}
 \begin{aligned}
 \sum_{\beta:\,\tau_\beta\le T}c_\beta
 &\le K_qT^2,\\
 \sup_{x\in Q}
 \left|
 \left\{
 \mathbf U\in Q^q:
 \begin{array}{c}
 \text{there exists }\beta\in\mathcal B_q
 \text{ with }W_\beta(x,\mathbf U)>0\\[-1mm]
 \text{and }\tau_\beta>T
 \end{array}
 \right\}
 \right|
 &\le K_qT^{-d}.
 \end{aligned}
\end{equation}
The vertical bars denote $dq$-dimensional Lebesgue measure. The first
estimate controls the total representation cost of the partition pieces
whose coefficient cutoff is at most $T$. The second controls the measure of
configurations on which at least one positive-weight partition piece has
cutoff larger than $T$. We use a union event rather than an average of the
discarded weights; this stronger statement will make the final diagonal
defect estimate immediate.

Both conclusions are preserved if some partition indices are removed:
the first sum only decreases, and the event in the second line only
shrinks.

The cases $q=0$ and $q=1$ are immediate. For $q=0$, the only polynomial
is $1$ with cutoff one. For $q=1$, a shell of radial scale $A$ has
representation cost at most $CA^{-2}$ by
\eqref{eq:simple-unary} and cutoff comparable to $A^{-1}$. Summing the
shells with $C/A\le T$ gives a bound $CT^2$. If a positive-weight shell
has $C/A>T$, then the neighbor lies within distance $C/T$ of the center,
whose volume is at most $CT^{-d}$. When $T<1$, the retained cost sum is
empty and the volume bound is trivial after increasing the constant.

We will also use a deterministic upper bound on the recursive cutoff. If
every radial scale appearing in a $q$-neighbor construction is at least
$2^{-L}$, then
\begin{equation}\label{eq:raw-cutoff-maximum}
 \tau_\beta
 \le
 C^q2^{CqL}.
\end{equation}
Indeed, at each internal node the logarithm of
\eqref{eq:recursive-cutoff} increases by at most a fixed multiple of
$1+L$, and every internal node consumes at least one previously unused
neighbor. Summing over the whole recursion therefore gives a bound of
order $q(1+L)$ for $\log_2\tau_\beta$; the depth of the balanced tree does
not create an additional factor.

We now prove \eqref{eq:raw-block-induction} by induction on $q$.
Consider an internal node with child sizes $q_1,q_2$. There are at most
$q^2$ choices for the two shell-labeled neighbors selected at the parent.
Fix the second selected radial scale
\[
 B=2^{-L}.
\]
By \eqref{eq:raw-cutoff-maximum}, a child with $q_i$ neighbors can occupy
at most
\[
 C(q_i+1)(L+1)
\]
dyadic cutoff intervals. Fix child intervals
\[
 [S_1,2S_1),
 \qquad
 [S_2,2S_2).
\]
By the induction hypothesis, the total representation costs of the two
children in these intervals are at most
\[
 CK_{q_1}S_1^2,
 \qquad
 CK_{q_2}S_2^2.
\]

For a retained parent, the recursive condition $\tau_\beta\le T$ and
\eqref{eq:recursive-cutoff} imply
\[
 A
 \ge
 cT^{-1}\Theta^{-1}B^{-\gamma}S_1S_2.
\]
The parent three-point representation contributes
\[
 C A^{-2}\Theta^{-3}
\]
by \eqref{eq:packet}. Summing $A^{-2}$ over all dyadic radial scales above
the displayed lower bound is controlled by the inverse square of that
lower bound. Multiplying by the two child cost bounds gives
\[
 CK_{q_1}K_{q_2}T^2B^{2\gamma}\Theta^{-1};
\]
the factors $S_1,S_2$ cancel completely. Summing over the dyadic angular
scales $\Theta\ge B$ gives
\[
 \sum_{\Theta\ge B}\Theta^{-1}\le\frac CB.
\]
Finally, summing over the possible child cutoff bins and over
$B=2^{-L}$ yields
\[
 Cq^2(q_1+1)(q_2+1)K_{q_1}K_{q_2}T^2
 \sum_{L\ge0}
 2^{-(2\gamma-1)L}(L+1)^2.
\]
The series converges because $2\gamma-1>0$. This proves the first line of
\eqref{eq:raw-block-induction}, provided the constants $K_q$ satisfy the
recursive bound stated below.

We next prove the discarded-volume estimate. If a parent partition piece
has $\tau_\beta>T$, then \eqref{eq:recursive-cutoff} implies
\[
 A
 <
 CT^{-1}B^{-\gamma}\Theta^{-1}\tau_1\tau_2.
\]
Once all other variables are fixed, integrating the closest neighbor over
this radial range contributes at most
\[
 CT^{-d}B^{-d\gamma}\Theta^{-d}\tau_1^d\tau_2^d.
\]
The second selected radial shell has volume at most $CB^d$. An angular bin
of scale $\Theta$ about either collinear direction has surface measure at
most $C\Theta^{d-1}$; the same estimate covers the final bin
$\Theta\asymp B$. After integrating these two selected neighbors, the
remaining factor is therefore
\[
 CT^{-d}
 B^{d(1-\gamma)}
 \Theta^{-1}
 \tau_1^d\tau_2^d.
\]

We must now average the child cutoff factors without assuming that they are
uniformly bounded independently of $L$. For a child with $q_1$ neighbors,
define
\[
 \widehat\tau_{q_1,L}(x,\mathbf U_1)
\]
to be the maximum $\tau_\beta$ among its positive-weight partition pieces
whose radial shell indices are at most $L$, with the fixed endpoint
convention for adjacent shells. Set this maximum to zero when no piece
qualifies and to one for an empty child. Since the parent has second
smallest radial scale $2^{-L}$, every child cutoff occurring in that parent
is bounded by $\widehat\tau_{q_1,L}$.

The second line of the induction hypothesis gives
\[
 \left|
 \{\mathbf U_1:\widehat\tau_{q_1,L}>t\}
 \right|
 \le
 K_{q_1}t^{-d},
\]
while \eqref{eq:raw-cutoff-maximum} gives the deterministic upper bound
\[
 \widehat\tau_{q_1,L}
 \le
 C^{q_1}2^{Cq_1(L+1)}.
\]
Using the layer-cake formula,
\begin{align}
 \int_{Q^{q_1}}
 \widehat\tau_{q_1,L}^{\,d}\,d\mathbf U_1
 &=
 \int_0^\infty
 d\,t^{d-1}
 \left|
 \{\mathbf U_1:\widehat\tau_{q_1,L}>t\}
 \right|\,dt
 \notag\\
 &\le
 1+
 dK_{q_1}
 \log\!\left(C^{q_1}2^{Cq_1(L+1)}\right)
 \notag\\
 &\le
 CK_{q_1}(q_1+1)(L+1).
 \label{eq:child-cutoff-moment}
\end{align}
For an empty child we use the value one directly. The same estimate holds
for the other child. Conditional on the center, the two child tuples are
disjoint, so after dropping the remaining radial restrictions their
integrals factor.

The interval partition used in Step~1 creates no additional multiplicative
factor in this volume calculation: its union is already contained in the
angular region of scale $\Theta$ whose measure was just estimated. Summing
over the angular scales $\Theta\ge B$ and then over $B=2^{-L}$ therefore
gives
\[
 Cq^2(q_1+1)(q_2+1)K_{q_1}K_{q_2}T^{-d}
 \sum_{L\ge0}
 2^{-\{d(1-\gamma)-1\}L}(L+1)^2.
\]
This series converges because
\[
 d(1-\gamma)-1>0.
\]
This proves the second line of \eqref{eq:raw-block-induction}.

We may consequently choose the constants recursively so that
\[
 K_q
 \le
 Cq^2(q_1+1)(q_2+1)K_{q_1}K_{q_2}.
\]
Because the recursion is balanced,
$q_1+q_2=q-2$ and $|q_1-q_2|\le1$. We claim that, for a sufficiently
large fixed $K$,
\[
 K_q\le\frac{K^q}{(q+1)^4}
 \qquad(q\ge0).
\]
Indeed, inserting the corresponding induction bounds for $K_{q_1}$ and
$K_{q_2}$ reduces the required estimate to the boundedness of
\[
 \frac{q^2(q+1)^4}
 {(q_1+1)^3(q_2+1)^3},
\]
and the two missing powers of $K$ absorb this bounded factor after $K$ is
chosen sufficiently large. In particular,
\[
 K_q\le C^q.
\]
This is the uniform exponential dependence on the number of neighbors that
we need.

\medskip\noindent\emph{Step 3: Passing to one fixed smooth frame.}
The preceding construction used polynomial coefficients in coordinates
centered at the current design point. We now embed all such polynomials,
for every degree, into one countable smooth family independent of $m$ and
$N$.

Choose a smooth compactly supported function $\chi$ that equals one on a
fixed neighborhood of $Q$, and choose $A_0$ larger than twice the largest
absolute coordinate on the support of $\chi$. For every multi-index
$\alpha\in\mathbb N^d$, define
\[
 \phi_\alpha(t)
 =
 \chi(t)A_0^{-|\alpha|}t^\alpha.
\]
Because $A_0$ dominates the coordinate size on the support of $\chi$, for
every fixed derivative order $b$,
\[
 \sum_{\alpha}
 \sum_{|\gamma|\le b}
 \|\partial^\gamma\phi_\alpha\|_\infty
 <\infty.
\]
Thus this countable family satisfies \eqref{eq:frame-assumption}; it also
contains the constant function $\phi_0=\chi=1$ on $Q$.

A degree-$q$ polynomial written in powers of $t-x$ can be translated into
powers of $t$ by the binomial theorem. Since $x$ ranges over the fixed cube,
and because of the factor $A_0^{-|\alpha|}$ in the frame, this conversion
increases its coefficient $\ell^1$ norm by at most $C^q$. We perform this
conversion only after all recursive multiplications have been completed.
The coefficients arising from the translation depend only on the center
$x$, so this conversion affects only the center factor in the separated
representation.

Fix a center with
\[
 q=m-1
\]
neighbors. Retain only partition indices satisfying
\[
 \tau_\beta
 \le
 T_{m,N}
 :=
 \frac{N}{mC^q}.
\]
For every retained index, convert its cardinal coefficient vector to the
fixed frame and denote the resulting vector by
$\widetilde p_\beta$. Define
\[
 E_{m,N}(\mathbf U)
 =
 \sum_{i=1}^m
 \sum_{\substack{\beta\in\mathcal B_{m-1}\\
                 \tau_\beta\le T_{m,N}}}
 W_{\beta,i}(\mathbf U)
 \widetilde p_{\beta,i}(\mathbf U)
 \widetilde p_{\beta,i}(\mathbf U)^T,
\]
where the subscript $i$ indicates that the corresponding construction is
centered at $U_i$.

Every summand is a nonnegative multiple of a rank-one positive
semidefinite matrix, so $E_{m,N}$ is positive semidefinite. Each retained
polynomial is cardinal for its center, and therefore
\[
 \phi(U_i)^TE_{m,N}(\mathbf U)\phi(U_j)
 =
 w_i(\mathbf U)\mathbf1_{\{i=j\}},
\]
where $w_i$ is the total retained partition weight for center $i$.
If the center differs from all of its neighbors, the full partition
weights are nonnegative and sum to one. If a neighbor coincides with the
center, every radial partition weight for that center vanishes. Thus
\[
 0\le w_i\le1
\]
in all cases.

At coincident configurations, retain the definition above, interpreting
each localized coefficient kernel
$W_{\beta,i}\widetilde p_{\beta,i}\widetilde p_{\beta,i}^T$ by the
pointwise extension supplied by the separated Fourier construction in
Steps~1--2. A summand whose center coincides with a neighbor extends by
zero, since its radial cutoff vanishes near that coincidence; coincidences
between neighbors away from the center retain their ordinary values.
Repeated locations therefore receive weight zero when used as centers,
while every retained cardinal polynomial centered elsewhere vanishes
there. Hence \eqref{eq:cardinal} holds pointwise at coincident
configurations as well.

We next prove the cost bound. For one center,
\eqref{eq:raw-block-induction} gives total raw representation cost at most
\[
 K_qT_{m,N}^2.
\]
Translation to the fixed frame and the outer product of the converted
coefficient vector change this by at most a factor $C^q$. Summing over the
$m$ centers and using $K_q\le C^q$ therefore gives
\[
 \operatorname{cost}(E_{m,N})
 \le
 C^mN^2,
\]
after enlarging $C$. This is \eqref{eq:geometry-cost}.

For the diagonal defect, the discarded weight for a fixed center is at
most one and can be nonzero only if there exists a positive-weight
partition index with
\[
 \tau_\beta>T_{m,N}.
\]
The second line of \eqref{eq:raw-block-induction} therefore gives
\[
 \int
 \bigl(1-w_i(\mathbf U)\bigr)\,d\mathbf U
 \le
 K_qT_{m,N}^{-d}.
\]
Summing over the $m$ centers and using $K_q\le C^q$ yields
\[
 \int
 \sum_{i=1}^m
 \bigl(1-w_i(\mathbf U)\bigr)\,d\mathbf U
 \le
 C^mN^{-d},
\]
which is \eqref{eq:geometry-error}.

The recursive construction used label-dependent tie breaking and splitting,
so the matrix just constructed need not yet be invariant under every
permutation of the design labels. Average it over all label permutations.
Positive semidefiniteness and the diagonal evaluation property are
preserved. The separated representation cost is convex under this average,
and the integral of the total diagonal defect is unchanged. Thus the
symmetrized kernel satisfies all three conclusions
\eqref{eq:cardinal}--\eqref{eq:geometry-cost} with the same bounds after
enlarging $C$.

Finally, for the consistency multipliers used later in
\eqref{eq:consistency-multiplier}, enlarge this fixed frame by the Fourier
functions introduced in the proof of
\eqref{eq:multiplier-cost-bound}. For each fixed response-degree cutoff
$R$, the exponential Fourier weights make the derivative sums of the
enlarged family finite, uniformly in $m$ and $N$. The original compact
cutoff gives the required smooth extensions to the prescribed open
neighborhood. No normalization such as
$\sum_a\phi_a^2=1$ is needed.

This proves \zcref{lem:geometry}.
\end{proof}

\section{The case \texorpdfstring{$0<s\le1$}{0 < s <= 1}}
\label{app:small-s}

We prove \zcref{thm:main2}. Throughout this section,
$0<s\le1$. All constants may depend on the fixed parameters defining the
model class but not on $n$.

\subsection{Upper bound}

\begin{proof}
We estimate $V$ by differencing two responses whose design points fall in
the same small cube. This removes the regression function up to its
H\"older variation across that cube.

Let
\[
 m
 =
 \left\lceil
 \max\left\{
   n^{2/(d+4s)},n^{1/d}
 \right\}
 \right\rceil,
 \qquad
 h=\frac1m,
\]
and partition $[0,1]^d$ into $m^d$ cubes of side length $h$. Call a cube
\emph{good} if it contains exactly two observations. If $b$ is good, let
$i_b,j_b$ be the indices of these observations and define
\[
 T_b=\frac{(Y_{i_b}-Y_{j_b})^2}{2}.
\]
Let $G$ be the number of good cubes, and set
\[
 \widehat V
 =
 \frac1G\sum_{b\ {\rm good}}T_b
\]
when $G>0$, with $\widehat V=0$ when $G=0$.

We first bound the conditional bias and variance of $T_b$. Write
\[
 Y_i=f(X_i)+\varepsilon_i.
\]
Conditional on the design, $\varepsilon_{i_b}$ and
$\varepsilon_{j_b}$ are independent, have mean zero, and have variance
$V$. Hence
\[
 \E(T_b\mid X_{1:n})
 =
 V+\frac12
 \bigl(f(X_{i_b})-f(X_{j_b})\bigr)^2.
\]
The two design points lie in a cube of diameter at most $\sqrt d\,h$.
Since $f$ is $s$-H\"older,
\[
 |f(X_{i_b})-f(X_{j_b})|
 \le Ch^s,
\]
and therefore
\[
 \left|
 \E(T_b\mid X_{1:n})-V
 \right|
 \le Ch^{2s}.
\]
The uniform conditional fourth-moment bound on the regression errors,
together with the uniform bound on $f$, also gives
\[
 \Var(T_b\mid X_{1:n})\le C.
\]
Statistics coming from distinct good cubes involve disjoint observations
and are therefore conditionally independent. Consequently, on $\{G>0\}$,
\begin{equation}\label{eq:small-s-risk-given-design}
 \E\!\left[
   (\widehat V-V)^2
   \,\middle|\,X_{1:n}
 \right]
 \le
 C\left(h^{4s}+\frac1G\right).
\end{equation}

It remains to control the number of good cubes. For a cube $b$, let
\[
 u_b=\Pp(X\in b).
\]
The upper and lower bounds on the design density imply
\[
 u_b\asymp h^d
\]
uniformly over the model. Moreover, the definition of $m$ gives
\[
 nh^d\le1.
\]
Thus the probability that a fixed cube contains exactly two observations is
\[
 \pi_b
 =
 \binom n2u_b^2(1-u_b)^{n-2}
 \asymp
 n^2h^{2d}.
\]
Since there are $m^d=h^{-d}$ cubes,
\[
 \mu:=\E G
 =
 \sum_b\pi_b
 \asymp
 n^2h^d.
\]

We also need concentration of $G$. For two distinct cubes, the same
two-block calculation used in \eqref{eq:ratio}, now with $k=2$, gives
\[
 \left|
 \frac{\Pp(b,c\text{ good})}{\pi_b\pi_c}-1
 \right|
 \le\frac Cn,
\]
because $u_b,u_c\le C/n$. Hence
\[
 \Var(G)
 \le
 \mu+\frac Cn\mu^2.
\]
Since every good cube contains two observations,
\[
 G\le\frac n2
\]
and therefore $\mu\le n/2$. It follows that
\[
 \Var(G)\le C\mu.
\]
The argument leading to \eqref{eq:inverse-count} consequently gives
\[
 \Pp(G=0)
 +
 \E\left[\frac{\mathbf1_{\{G>0\}}}{G}\right]
 \le
 \frac C\mu
 \le
 \frac{C}{n^2h^d}.
\]

On $\{G=0\}$ we have $\widehat V=0$, and
\[
 V^2
 \le
 \E(\varepsilon^4\mid X)
 \le
 C_4.
\]
Taking expectations in \eqref{eq:small-s-risk-given-design} and using the
preceding inverse-count estimate therefore gives
\begin{equation}\label{eq:small-s-upper}
 \sup_{\theta\in\Theta}
 \E_\theta(\widehat V-V)^2
 \le
 C\left(
   h^{4s}+\frac1{n^2h^d}
 \right).
\end{equation}

We now evaluate the two terms. If $d>4s$, then
\[
 m\asymp n^{2/(d+4s)},
 \qquad
 h\asymp n^{-2/(d+4s)},
\]
and hence
\[
 h^{4s}
 \asymp
 \frac1{n^2h^d}
 \asymp
 n^{-8s/(d+4s)}.
\]
If $d\le4s$, then
\[
 m\asymp n^{1/d},
 \qquad
 h\asymp n^{-1/d},
\]
so
\[
 h^{4s}\le Cn^{-1},
 \qquad
 \frac1{n^2h^d}\asymp n^{-1}.
\]
Thus, in all cases,
\[
 \sup_{\theta\in\Theta}
 \E_\theta(\widehat V-V)^2
 \le
 C\left(
   n^{-1}\vee n^{-8s/(d+4s)}
 \right).
\]
The preceding estimates hold for all sufficiently large $n$. For the
finitely many remaining sample sizes, the estimator identically equal to
zero has uniformly bounded risk, so enlarging $C$ proves the upper bound.
\end{proof}

\subsection{Lower bound}

\begin{proof}
We construct a mixture of regression functions whose variance parameter is
smaller than that of a fixed null model, while arranging that the
one-observation mixture distribution agrees exactly with the null. The
remaining discrepancy then appears only through correlations created by
using the same random regression function across all observations.

Fix a scale
\[
 h=\frac1m,
\]
where $m$ will be chosen below. Use the smooth periodic functions from
\eqref{eq:smooth-windows} at this scale. Thus there are at most
$Ch^{-d}$ functions $w_j$, each supported on a set of volume at most
$Ch^d$, with bounded overlap and
\[
 \sum_jw_j(x)^2=1,
 \qquad
 \|\nabla w_j\|_\infty\le\frac Ch.
\]
Let $(\omega_j)_j$ be independent random signs and define
\[
 f_\omega(x)
 =
 \kappa\sum_jw_j(x)\omega_j,
 \qquad
 \kappa=b_0h^s,
\]
where $b_0>0$ will be chosen sufficiently small.

Bounded overlap gives
\[
 \|f_\omega\|_\infty\le C\kappa.
\]
Moreover,
\[
 |f_\omega(x)-f_\omega(y)|
 \le
 C\kappa
 \min\left\{1,\frac{\|x-y\|}{h}\right\}.
\]
Since $0<s\le1$,
\[
 \min\{1,t\}\le t^s,
 \qquad t\ge0,
\]
and therefore
\[
 |f_\omega(x)-f_\omega(y)|
 \le
 Cb_0\|x-y\|^s.
\]
Thus, after taking $b_0$ sufficiently small, every realization
$f_\omega$ belongs to the prescribed $C^s$ ball, including at $s=1$.
The periodic construction also gives the required smooth extension to the
fixed neighborhood of the design cube.

The sign moments are
\[
 \E_\omega f_\omega(x)=0
\]
and, using $\sum_jw_j(x)^2=1$,
\[
 \E_\omega f_\omega(x)^2
 =
 \kappa^2.
\]
These two identities are the reason for this choice of prior.

Take the uniform design density $p\equiv1$. Let $Q_0$ be the
one-observation law with regression function $f=0$, variance $V=v$, and
the ternary conditional response distribution \eqref{eq:ternary}. For a
sign vector $\omega$, let $Q_\omega$ have regression function
$f_\omega$ and variance
\[
 V_\omega=v-\kappa^2.
\]
For sufficiently small $b_0$, we have
$|f_\omega|\le C\kappa$ and $\kappa^2$ arbitrarily small. The neighborhood
of $(f,V)=(0,v)$ verified after \eqref{eq:ternary} therefore ensures that
every $Q_\omega$ belongs to the model: the three response probabilities
remain positive, the variance remains in $(v_-,v_+)$, and the conditional
fourth-moment bound is satisfied.

Let
\[
 L_\omega
 =
 \frac{dQ_\omega}{dQ_0}
\]
be the one-observation likelihood ratio. From
\eqref{eq:ternary}, for each response value $y$,
\[
 q_y(f_\omega(x),v-\kappa^2)
 =
 \beta_y
 +\ell_yf_\omega(x)
 +r_y\bigl(f_\omega(x)^2-\kappa^2\bigr),
\]
where $\beta_y=q_y(0,v)$. Averaging over the signs and using the two
identities above gives
\[
 \E_\omega
 q_y(f_\omega(x),v-\kappa^2)
 =
 q_y(0,v)
\]
for every $x$ and $y$. Equivalently,
\begin{equation}\label{eq:small-s-one-observation-match}
 \E_\omega L_\omega(x,y)=1
\end{equation}
pointwise. Since the null cell probabilities are bounded below by a fixed
positive constant and $\|f_\omega\|_\infty\le C\kappa$, the same formula
also gives
\[
 \|L_\omega-1\|_\infty\le C\kappa.
\]

We now bound the discrepancy for $n$ observations. Let $\omega'$ be an
independent copy of $\omega$ and define
\[
 A(\omega,\omega')
 =
 \E_{Q_0}
 \bigl[
   (L_\omega-1)(L_{\omega'}-1)
 \bigr].
\]
For every fixed $\omega'$, the exact one-observation identity
\eqref{eq:small-s-one-observation-match} gives
\[
 \E_\omega A(\omega,\omega')=0.
\]

Changing one sign $\omega_j$ changes $f_\omega$ only on the support of
$w_j$. On this support it changes $f_\omega$ by at most $C\kappa$, and
the ternary formula therefore changes $L_\omega$ by at most $C\kappa$.
Since
\[
 \|L_{\omega'}-1\|_\infty\le C\kappa
\]
and the support of $w_j$ has volume at most $Ch^d$, changing $\omega_j$
changes $A(\omega,\omega')$ by at most
\[
 C\kappa^2h^d.
\]
There are at most $Ch^{-d}$ independent signs. The bounded-differences
version of Hoeffding's lemma therefore gives, uniformly in $\omega'$,
\[
 \E_\omega
 e^{tA(\omega,\omega')}
 \le
 \exp\left(Ct^2\kappa^4h^d\right),
 \qquad
 t\in\mathbb R.
\]

Let
\[
 \overline Q_n
 =
 \E_\omega Q_\omega^{\otimes n}
\]
be the mixture of the $n$-observation alternative laws. Its likelihood
ratio relative to $Q_0^{\otimes n}$ is
\[
 \frac{d\overline Q_n}{dQ_0^{\otimes n}}
 =
 \E_\omega
 \prod_{\ell=1}^nL_\omega(Z_\ell).
\]
Squaring and integrating under $Q_0^{\otimes n}$ gives
\begin{align*}
 1+
 \chi^2(\overline Q_n,Q_0^{\otimes n})
 &=
 \E_{\omega,\omega'}
 \left[
   \E_{Q_0}(L_\omega L_{\omega'})
 \right]^n\\
 &=
 \E_{\omega,\omega'}
 \bigl(1+A(\omega,\omega')\bigr)^n.
\end{align*}
Since
\[
 1+A(\omega,\omega')
 =
 \E_{Q_0}(L_\omega L_{\omega'})>0,
\]
we may use $1+x\le e^x$ to obtain
\begin{equation}\label{eq:small-s-chi-square}
 \begin{aligned}
 \chi^2(\overline Q_n,Q_0^{\otimes n})
 &\le
 \E_{\omega,\omega'}e^{nA(\omega,\omega')}-1\\
 &\le
 \exp\left(Cn^2\kappa^4h^d\right)-1.
 \end{aligned}
\end{equation}

Choose
\[
 m=\left\lceil2n^{2/(d+4s)}\right\rceil,
 \qquad
 h=\frac1m.
\]
Then
\[
 n^2\kappa^4h^d
 =
 b_0^4n^2h^{d+4s}
 \le
 b_0^4.
\]
By decreasing the fixed constant $b_0$ if necessary,
\eqref{eq:small-s-chi-square} makes
$\chi^2(\overline Q_n,Q_0^{\otimes n})$ arbitrarily small. Since
\[
 \TV(P,Q)\le\frac12\sqrt{\chi^2(P,Q)},
\]
we may in particular arrange
\[
 \TV(\overline Q_n,Q_0^{\otimes n})\le\frac14.
\]

Every model in the alternative mixture has variance
$v-\kappa^2$, whereas the null model has variance $v$. Thus the two
variance values are separated by
\[
 \Delta=\kappa^2=b_0^2h^{2s}
 \asymp
 n^{-4s/(d+4s)}.
\]
Applying the testing inequality \eqref{eq:testing}, with the null point
mass as one prior and the sign mixture as the other, gives
\[
 \mathfrak r_n
 \ge
 c\,n^{-4s/(d+4s)}.
\]
Together with the parametric lower bound
\eqref{eq:parametric-lower},
\[
 \mathfrak r_n\ge cn^{-1/2},
\]
we conclude that
\[
 \mathfrak r_n
 \ge
 c\left(
   n^{-1/2}\vee n^{-4s/(d+4s)}
 \right).
\]
Combining this with the upper bound proved above establishes
\zcref{thm:main2}. When $d=4s$, the two powers coincide, so the same
argument includes the boundary case.
\end{proof}

\bibliographystyle{amsplain}
\enlargethispage{2\baselineskip}
\bibliography{variance_estimation_refs}
\end{document}